\documentclass[10pt,a4paper]{amsart}
\usepackage{amsmath, amssymb, graphicx, subfigure, epsfig}
\usepackage{xcolor}

\numberwithin{equation}{section}

\newcommand{\CA}{\mathcal{A}}

\newcommand{\CB}{\mathcal{B}}
\newcommand{\CD}{\mathcal{D}}
\newcommand{\CE}{\mathcal{E}}
\newcommand{\CF}{\mathcal{F}}
\newcommand{\CH}{\mathcal{H}}

\newcommand{\CBa}{\mathcal{B}_{1d}}
\newcommand{\e}{\epsilon}

\newcommand{\br}{\mathbb{R}}
\newcommand{\ik}{\varphi}
\newcommand{\pa}{\partial}
\newcommand{\bt}{\beta}
\newcommand{\al}{\alpha}
\newcommand{\la}{\lambda}

\newcommand{\coi}{C_0^{\infty}}
\newcommand{\ioi}{\int_0^{\infty}}

\newcommand{\de}{\delta}

\newcommand{\be}{\begin{equation}}
\newcommand{\ee}{\end{equation}}

\newcommand{\ts}{\text{supp}}
\newcommand{\dt}{\text{det}}

\newcommand{\xe}{x_\e}
\newcommand{\xc}{\check x}
\newcommand{\lte}{f_{\Lambda\e}}
\newcommand{\fs}{\tilde\upsilon}
\newcommand{\dir}{\Theta}
\newcommand{\s}{\mathcal S}

\newtheorem{theorem}{Theorem}[section]
\newtheorem{lemma}[theorem]{Lemma}

\theoremstyle{remark}
\newtheorem{remark}[theorem]{Remark}

\usepackage{amsthm}

\theoremstyle{definition}
\newtheorem{definition}{Definition}[section]

\begin{document}

\title[Resolution of reconstruction from discrete data]{Analysis of resolution of tomographic-type reconstruction from discrete data for a class of distributions}
\author[A Katsevich]{Alexander Katsevich$^1$}
\thanks{$^1$Department of Mathematics, University of Central Florida, Orlando, FL 32816.\\ 
This work was supported in part by NSF grants DMS-1615124 and DMS-1906361.}

\begin{abstract} Let $f(x)$, $x\in\mathbb R^2$, be a piecewise smooth function with a jump discontinuity across a smooth surface $\mathcal S$. Let $f_{\Lambda\epsilon}$ denote the Lambda tomography (LT) reconstruction of $f$ from its discrete Radon data $\hat f(\alpha_k,p_j)$. The sampling rate along each variable is $\sim\epsilon$. First, we compute the limit $f_0(\check x)=\lim_{\epsilon\to0}\epsilon f_{\Lambda\epsilon}(x_0+\epsilon\check x)$ for a generic $x_0\in\mathcal S$. Once the limiting function $f_0(\check x)$ is known (which we call the discrete transition behavior, or DTB for short), the resolution of reconstruction can be easily found. Next, we show that straight segments of $\mathcal S$ lead to non-local artifacts in $f_{\Lambda\epsilon}$, and that these artifacts are of the same strength as the useful singularities of $f_{\Lambda\epsilon}$. We also show that $f_{\Lambda\epsilon}(x)$ does not converge to its continuous analogue $f_\Lambda=(-\Delta)^{1/2}f$ as $\epsilon\to0$ even if $x\not\in\mathcal S$. Results of numerical experiments presented in the paper confirm these conclusions. We also consider a class of Fourier integral operators $\mathcal{B}$ with the same canonical relation as the classical Radon transform adjoint, and a class of distributions $g\in\mathcal{E}'(Z_n)$, $Z_n:=S^{n-1}\times\mathbb R$, and obtain easy to use formulas for the DTB when $\mathcal{B} g$ is computed from discrete data $g(\alpha_{\vec k},p_j)$. Exact and LT reconstructions are particlular cases of this more general theory.\end{abstract}
\maketitle

%keywords: reconstruction, resolution, edge response, equidistribution
%ams subj 44A12, 65R10, 92C55, 94A08

\section{Introduction}\label{sec_intro}

Analysis of the resolution of tomographic reconstruction of a function $f$ from its discrete Radon transform data $\hat f(\al_k,p_j)$ is a practically important problem. Usually, it is solved in the setting of the sampling theory, which assumes that $f$ is essentially bandlimited \cite{nat93, pal95, far04}. An extension of this theory allows $f$ to have at most semiclassical singularities \cite{stef18}. Frequently, one would like to know how accurately and with what resolution the {\it classical} singularities of $f$ (e.g., a jump discontinuity across a smooth surface $\s$) are reconstructed. Let $f_\e$ denote the function reconstructed from discrete data, where $\e$ represents the data sampling rate. In \cite{kat_2017, kat19a, kat19b} the author initiated the analysis of reconstruction by focusing specifically on the behavior of $f_\e$ near a jump discontinuity of $f$. One of the main results of these papers is the computation of the limit 
\be\label{tr-beh}
f_0(\check x):=\lim_{\e\to0}f_\e(x_0+\e\check x)
\ee
for a generic point $x_0\in\s$. In \eqref{tr-beh} it is assumed that $\check x$ is confined to a bounded set. 
%Thus we compute the behavior of $f_\e$ in a suitably scaled $\e$-neighhborhood of $x_0$ in the limit as $\e\to0$. 
It is important to emphasize that both the size of the neighborhood around $x_0$ and the data sampling rate go to zero simultaneously in \eqref{tr-beh}. Once the limiting function $f_0(\check x)$ is known (which we call the discrete transition behavior, or DTB for short), the resolution of reconstruction can be easily computed. For simplicity, the dependence of $f_0(\check x)$ on $x_0$ is omitted from notation. In \cite{kat_2017} we find $f_0(\check x)$ for the Radon transform in $\br^2$ in two cases: $f$ is static and $f$ changes during the scan (dynamic tomography). In \cite{kat19a} we find $f_0(\check x)$ for the classical Radon transform in $\br^3$, and in \cite{kat19b} -- for a wide family of generalized Radon transforms in $\br^3$. A common thread through these calculations is that the well-behaved DTB (i.e., the limit in \eqref{tr-beh}) is guaranteed to exist only if $x_0\in\s$ is generic. Derivation of this property is closely connected with the uniform distribution theory \cite{KN_06}. Roughly, a point is generic (or, locally generic, to be more precise) if the available data is in general position relative to the local patch of $\s$ containing $x_0$.

In this paper we extend our results by considering more general {reconstruction} operators $\CB$, whose canonical relation coincides with that of the classical Radon transform adjoint. The first step is to apply a differential or pseudodifferential operator along the affine variable $p$ (which we denote $\CBa$), and the second step is to backproject to the image domain. The operators $\CB$ can preserve the degree of smoothness of $f$ (as is the case with exact reconstruction), and they can enhance the singularities of $f$. A common example of the latter is Lambda (also known as local) tomography \cite{vkk, rk, fbh01}. We also assume that $\CB$ acts on more general $\hat f$, where $f$ may have singularities other than jump discontinuities. 

Let $\ik$ be an interpolation kernel, which is applied to the data with respect to $p$. The discrete version of $\CB$, which is denoted $\CB_\e$, consists of applying $\CBa$ to the interpolated data (the filtering step), and then approximating the integral with respect to $\al$ (the backprojection step) by summing over the available directions $\al_k$. 

The paper is organized as follows. In Section~\ref{ltwosm} we consider Lambda tomography (or, LT for short) in $\br^2$ in the case when $f$ has a jump discontinuity across a smooth and convex surface $\s$. Let $f_\Lambda:=(-\Delta)^{1/2}f$ denote the LT reconstruction from continuous data, and $\lte$ - LT reconstruction from discrete data. In this case, $\CBa=\pa_p^2$. At the beginning of Section~\ref{ltwosm} we introduce necessary notations, key formulas, and give the definition of a generic point. In Subsection~\ref{ernosm} we obtain the DTB (more precisely, the edge response since $f$ has a jump discontinuity) of LT. We show that if $x_0\in\s$ is generic, then the limit  
\be\label{tr-beh-lt}
f_0(\check x):=\lim_{\e\to0}\e\lte(x_0+\e\check x)
\ee
exists. 
%We compute it explicitly as a function of $\check x$, assuming that $\check x$ is confined to a bounded set. 
Since LT enhances singularities by 1 in the Sobolev scale, i.e., $f_\Lambda \in H^{s-1}(\br^2)$ if $f\in H_0^s(\br^2)$, we have to multiply $\lte$ by $\e$ when computing $f_0$. Additionally, it turns out that $f_0$ equals to the convolution of the leading singularity of $f_\Lambda$ at $x_0$ and $\ik$ (see Lemma~\ref{lem-phi-lim}). By analogy, the leading singularity of a distribution across its singular support (e.g., of $f_\Lambda$ across $\s$) will be called continuous transition behavior, or CTB for short.

In Subsection~\ref{sec:LTline} we show that if $f$ has a jump discontinuity along a flat piece of $\s$, then $\lte$ has a non-local artifact along a line containing the flat piece. Moreover, the strength of the artifact is of the same order of magnitude $O(1/\e)$ as the useful singularity (cf. \eqref{tr-beh-lt}), and the artifact does not go to zero as $\e\to0$. In Subsection~\ref{sec:LTremote} we show that the effect of remote singularities is quite dramatic. If $f$ has a jump singularity across a smooth and convex surface $\s$, then, generally, $\lte(x)\not\to f_\Lambda(x)$ as $\e\to0$ even for $x\not\in \s$. The nature of finite sampling artifacts in the conventional tomographic reconstruction in $\br^2$ is well-known (see e.g., Section 12.3 in \cite{eps08} and references therein). Here we use a completely different approach, and discretization artifacts in LT are more severe than in the exact reconstruction.

%As mentioned above, the DTB of $\lte$ across a smooth and strictly convex segment of $\s$ equals to the convolution of the interpolation kernel $\ik$ with the CTB of $f_{\Lambda}$ across $\s$. It is reasonable to expect that a similar pattern holds more generally. 

In Sections~\ref{sec:gen} -- \ref{sec:lots} we extend the computation of the DTB to more general {reconstruction} operators and distributions. In Section~\ref{sec:gen} we start with a sufficiently regular conormal distribution $f\in\CE'(\br^n)$, which is {non-smooth across a smooth, convex surface $\s$ of codimension one. More precisely, the wave front set of $f$ is contained in the conormal bundle of $\s$.} We also introduce a class of Fourier Integral Operators (FIO) $\CB$: $\CE'(Z_n)\to \CD'(\br^n)$, where $Z_n=S^{n-1}\times\br$. To describe the leading singular behavior at a point of a distribution we use the definition of expansion in smoothness introduced in \cite{kat99_a}. This notion is closely related to the asymptotics at infinity of the principal symbol of a conormal distribution with a polyhomogeneous symbol (see e.g. the proof of Proposition 18.2.2 in \cite{hor3} for a related argument). {However, the expansion in smoothness is more convenient for the purposes of this paper as it fits well with the idea of transition behavior}.

In the rest of Section~\ref{sec:gen}, we compute the leading singularities of $f$, $\CB f$ (or, CTB), and $\hat f$ given the asymptotics of the Fourier transform of $f$ at infinity. See Lemmas~\ref{lem:f0sing}, \ref{cont-trans}, and \ref{lem:fhsing}, respectively.
Even though these calculations are fairly straightforward, the obtained formulas are needed in what follows and make the paper self-contained. The leading singularity of $\CB f$ is used in a generalization of Lemma~\ref{lem-phi-lim} (see Theorem~\ref{thrm:lot-main}, where the CTB is denoted $\mu$). The leading singularity of $\hat f$ is used as a starting point when deriving the DTB of the reconstruction $\CB_\e\hat f$ (see \eqref{gfn-exp}). More general calculations relating the singularities of $f$ and $\hat f$ are in \cite{ar01, rz2, rz1}. Our approach is simpler, and it is convenient to have all the necessary formulas in one place. 

In Section~\ref{sec:main-ass} {we introduce a more general class of distributions $g\in\CE'(Z_n)$, whose singularities resemble those of $\hat f$ obtained in Section~\ref{sec:gen}. The singular support of $g$ is a subset of a smooth, convex, codimension one surface in $Z_n$. The generalization is along two directions. First, we relax the requirement that $g$ be in the range of the Radon transform. Second, we impose a fairly weak assumption about the behavior of $g$ near its singular support.} Then we introduce a more general interpolating kernel and the definition of a generic point. In the remainder of Section~\ref{sec:conn} we compute the DTB of $\CB_\e g$ by retaining only the leading order terms in $\CB$ and $g$ (see Theorems~\ref{lead-thrm-pos} and \ref{lead-thrm-zero}). In the spirit of \eqref{tr-beh} and \eqref{tr-beh-lt}, the DTB is computed using the formula
\be\label{tr-beh-v2}
f_0(\check x):=\lim_{\e\to0}\e^a(\CB_\e g)(x_0+\e\check x)
\ee
for some $a\ge0$. The value of $a$ depends on how singular $\CB g$ is at $x_0$. 
In the case of exact reconstruction, if, for example, $f=\CB \hat f$ has a jump discontinuity, then $a=0$ and we get \eqref{tr-beh}. In the case of LT, if $f$ has a jump across $\s$, then $f_\Lambda(x_0+h\dir_0)\sim 1/h$ and $a=1$ (cf. \eqref{tr-beh-lt}). Here $\dir_0$ is a vector normal to $\s$ at $x_0$. 

In Section~\ref{sec:lots} we show that if either $\CB$ or $g$ is missing the leading term, then $\CB_\e g$ does not exhibit transition behavior. At the end of Section~\ref{sec:lots} we state our main result, which describes the DTB of $\CB_\e g$ for the classes of operators $\CB$ and distributions $g$ introduced in Sections~\ref{sec:gen} and \ref{sec:main-ass}, respectively. 

As mentioned above, the DTB of $\lte$ across a smooth and strictly convex segment of $\s$ equals to the convolution of the interpolation kernel $\ik$ with the CTB of $f_{\Lambda}$ across $\s$. The same pattern holds more generally: the DTB of $\CB_\e g$ is the convolution of the interpolation kernel and the CTB of $\CB g$. Our formulas can be used for easy calculation of the resolution for a  variety of tomographic type reconstructions from discrete data. {In turn, this can be used  for optimizing both the data collection protocol and the reconstruction algorithm. For example, if one is interested in locating a faint jump in a reconstructed image, one can design an edge-enhancing reconstruction algorithm (e.g., of LT type) and interpolation kernel, so that the jump stands out most clearly. The choice of the reconstruction operator $\CB$ (and its discrete counterpart $\CB_\e$) affects the CTB (respectively, DTB), and that affects the detectability of the jump. Besides LT, another example of edge enhancing reconstruction to which our theory applies is computing the derivatives of $f$ directly from the data \cite{hl12, louis16}. }

In Section~\ref{sec:pixsize} we show that if the data are the discrete values of $g$ convolved with some detector aperture function, then the DTB remains qualitatively the same. It is obtained by convolving the CTB of $\CB g$ with $\ik$ and with the aperture function. This is consistent with \cite{stef18}, where a similar phenomenon was observed for semiclassical singularities. Nevertheless, smoothing the data over intervals of length $\sim\e$ does not allow one to relax the requirement that $x_0$ be generic. If $x_0$ is not generic, the behavior of reconstruction may differ significantly from the predicted one, and this is confirmed by numerical experiments. Thus, the requirement that $x_0$ be generic is a phenomenon associated with clasical singularities, as it does not arise in the semiclassical case. Results of numerical experiments are in Section~\ref{sec:numerix}. They are in agreement with all the conclusions regarding the behavior of LT obtained in Section~\ref{ltwosm}. In particular, we show that the behavior of $\lte$ is much more sensitive to whether $x_0$ is generic or not than in the case of exact reconstruction (see \cite{kat19a}). For the convenience of the reader, most of the proofs are moved from the main text to the appendices.

{Besides linear algorithms, there exist a variety of other approaches to reconstruction from discrete tomographic data \cite{hny13}. Many of them, for example, iterative algorithms, do not fall under the theory developed in this paper. Some iterative algorithms, e.g. those that use Total Variation (or any other edge-preserving prior) as a regularizer, enhance edges. As a result, they may provide resolution higher than that predicted by the linear theory. Nevertheless, our results are valuable because of several reasons. (a) The linear theory provides a baseline of practically achievable resolution that a nonlinear algorithm can be compared with. (b) Our linear theory describes the resolution as a function of the location of the singularity and its orientation. There is no general theory for nonlinear algorithms, and one has to conduct extensive numerical experiments to study their resolution. For example, to obtain resolution measurements in $\br^3$, one generally has to sample the five-dimensional space $\br^3\times S^2$ of point-direction pairs. Here $S^2$ denotes the unit sphere in $\br^3$. Given that iterative algorithms are computationally intensive, such a comprehensive analysis can be prohibitive. (c) The resolution of nonlinear methods is contrast dependent (e.g., lower contrast features are reconstructed with lower resolution), which makes their resolution analysis even more computationally demanding.} 
%While a large body of literature shows that nonlinear methods do have significant advantages over linear ones, 

{Finally, our analysis is practically important because filtered back-projection (FBP) algorithms, which are linear, are still widely used in cases where the amount of data is high (as is the case in micro CT) or when a simple and easy to use reconstruction algorithm is preferred.  For example, see a recent book \cite{orh20}, where applications of micro CT in areas such as Bone Morphometry and Densitometry, Osteoporosis Research, Cardiovascular Engineering and Bio-inspired Design,  Materials Science and Aerospace Engineering, and many others are described. As is stated on p. 29 of \cite{orh20}, ``The filtered back-projection method is the most common method used in the reconstruction.'' Another important application of (linear) FBP algorithms is where high throughput is essential (e.g., in wood mills \cite{gku16} and airport security scanning \cite{kyf19, tmk19}).
%, and our theory provides a benchmark against which any reconstruction algorithm can be compared.
}

\section{Analysis of Lambda tomography reconstruction}\label{ltwosm}

\subsection{Preliminary material} In this section we consider functions, which can be represented as a finite sum
\begin{equation}\label{f_def}
f(x)=\sum_j \chi_{D_j} f_j(x),
\end{equation}
where $\chi_{D_j}$ is the characteristic function of the domain $D_j\subset \br^2$. For each $j$:
\begin{itemize}
\item[(f1)] $D_j$ is bounded,  
\item[(f2)] The boundary of $D_j$ is piecewise $C^{\infty}$,  
\item[(f3)] $f_j$ is $C^\infty$ in a domain containing the closure of $D_j$.
\end{itemize}
By construction, $\text{singsupp}(f)\subset \s:=\cup_j \pa D_j$.

The Lambda (or, local) tomography (LT) reconstruction is given by \cite{vkk, rk, fbh01}
\be\label{lt-main}
f_\Lambda(x):=(\Lambda f)(x)=-\frac1{2\pi}\int_{-\pi/2}^{\pi/2}\hat f''(\al,\al\cdot x)d\al,
\ee
where $\hat f=Rf$. As is well known \cite{rk}, $\Lambda f={\mathcal F}^{-1}(|\xi|\tilde f(\xi))$, where $\tilde f$ is the Fourier transform of $f$. In this paper, the Fourier transform and its inverse are defined as follows:
\be\label{ftandinv}
\tilde f(\xi)=({\mathcal F}f)(\xi)=\int f(x)e^{i\xi\cdot x}dx,\
f(x)=({\mathcal F}^{-1}\tilde f)(\xi)=\frac1{(2\pi)^n}\int \tilde f(\xi)e^{-i\xi\cdot x}d\xi,
\ee
where $n$ is the dimension of the space.

Suppose $\hat f(\al,p)$ is known at the points
\begin{equation}\label{data_pts}
\al_k=\Delta\al(q_\al+k),\ p_j=j\Delta p,\ \Delta p=\e,\ \Delta\al={\kappa}\e,
\end{equation}
for some fixed $\kappa>0$ and $q_\al\in \br$. All our results are asymptotic as $\e\to0$.

Let $\ik$ be a function, which satisfies the following assumptions:
\begin{itemize}
%\item[IK0.] \red{$\ik$ is an interpolating kernel (IK), i.e. 
%\be\label{int-ker-req}
%\ik(0)=1,\ \ik(j)=0,\ j\in\mathbb Z, j\not=0;
%\ee}
\item[IK1.] $\ik$ is exact up to the degree $2$, i.e.
\be\label{ker-int}
\sum_{j\in \mathbb Z} j^m\ik(t-j)=t^m,\quad 0\le m\le 2,\ t\in\br;
\ee
\item[IK2.] $\ik$ is compactly supported;
\item[IK3.] One has $\ik^{(j)}\in L^{\infty}(\br)$, $0\le j\le 3$; and
\item[IK4.] $\ik$ is normalized, i.e. $\int_{\br}\ik(y)dy=1$.
\end{itemize}
The interpolated in $p$ version of $\hat f$ becomes
\be\label{g-int}
\hat f_\e(\al_k,p):=\sum_j \hat f(\al_k,\e j) \ik\left(\frac{p-\e j}{\e}\right).
\ee

Pick a point $x_0\in \s$ such that the curvature of $\s$ at $x_0$ is not zero. Let $\dir_0=(\cos\theta_0,\sin\theta_0)$ be the normal, which points from $x_0$ towards the center of curvature of $\s$ at $x_0$. We will call the side of $\s$ where $\dir_0$ points ``positive'', and the opposite side - ``negative''. 

\begin{definition}
The point $x_0\in\s$ is generic if the quantity $(\Theta_0^\perp\cdot x_0){{\kappa}}$ is irrational.
\end{definition}

Let $\chi(\al)$ be a smooth cut-off supported in a small neighborhood of $\theta_0$, $\theta_0\in\ts(\chi)\subset(-\pi/2,\pi/2)$, such that $\chi(\theta_0)=1$. If $\theta_0\in \{\pm\pi/2\}$, we can shift the interval of integration in \eqref{lt-main} so that $\theta_0$ is in its interior. By linearity and in view of the partition of unity-type arguments, without loss of generality we insert the cut-off in \eqref{lt-main} and define the reconstruction from discrete data using \eqref{lt-main}, \eqref{data_pts}, and \eqref{g-int}:
\be\label{fe-0}
\lte(x):=-\frac{1}{2\pi\e^2}\sum_k \sum_j\ik''\left(\frac{\al_k\cdot x-\e j}{\e}\right)
\hat f(\al_k,p_j)\chi(\al_k)\Delta\al.
\ee

\subsection{Edge response}\label{ernosm}

Pick a generic $x_0\in\s$. By linearity, we may suppose that (i) $f(x)\equiv0$ outside a small neighborhood of $x_0$, and (ii) $f(x)\equiv0$ on the negative side of $\s$. In this case, near $\text{singsupp}(\hat f)$ we have \cite{ar01, rz2, rz1}
\be\label{g-ass}
\hat f(\al,p)=2 f_+(\al) \sqrt{2R(\al)}(p-H(\al))_+^{1/2}+O\left((p-H(\al))_+^{3/2}\right),
\ee
where the big-O term can be differentiated with respect to $p$. The function $H:\ts(\chi)\to\br$ is defined by the condition that $\{x\in\br^2:\,x\cdot\al=H(\al)\}$, $\al\in\ts(\chi)$, is a family of lines tangent to $\s$ near $x_0$, $f_+(\al)$ is the limiting value of $f$ from the positive side at the point of tangency, and $R(\al)$ is the radius of curvature of $\s$ at the point of tangency. Substitute \eqref{data_pts} and \eqref{g-ass} into \eqref{fe-0}:
\be\label{fe-1-nosm}\begin{split}
\lte(x):=&-\frac{1}{2\pi\e^2}\sum_k \sum_j\ik''\left(\frac{\al_k\cdot x-\e j}{\e}\right)
\hat f(\al_k,p_j)\chi(\al_k){{\kappa}}\e\\
=&\frac{1}{\e^2}\sum_k \sum_j\ik''\left(\frac{\al_k\cdot x-\e j}{\e}\right)
\biggl[\rho(\al_k)(\e j-H(\al_k))_+^{1/2}\\
&\hspace{1cm}+O\left((\e j-H(\al_k))_+^{3/2}\right)\biggr]\chi(\al_k){{\kappa}}\e, \ \rho(\al):=-\frac{f_+(\al) \sqrt{2R(\al)}}{\pi}.
\end{split}
\ee
In what follows, the quantities $\rho(\theta_0)$, $f_+(\theta_0)$, and $R(\theta_0)$ are denoted by $\rho$, $f_+$, and $R$, respectively. Set
\be\label{xe-or-def}
\xe:=x_0+\e \check x,\ h:=\Theta_0\cdot\check x,\ p(\al):=\al\cdot x_0-H(\al),
\ee
where $\check x$ is confined to a bounded set. We have 
\be\label{p0asymp}
H(\al)=\al\cdot x_0-(R/2)(\al-\theta_0)^2+O((\al-\theta_0)^3).
\ee
With a slight abuse of notation, here and in a few other places below we use $\al$ (and $\al_k$) both as a vector and as a scalar. We believe that the meaning of the variable is clear from the context in each particular case.

In view of \eqref{fe-1-nosm}, define
\be\label{psi-def}
\psi(t,p):=\sum_j \ik''(t-j)(j-p)_+^{1/2}.
\ee
The following statements are immediate:
\be\label{psi-props}\begin{split}
&\psi(t,p)=0\text{ if }t-p<c \text{ for some $c<0$},\\
&\psi(t,p)=O((t-p)^{-3/2}),\ t-p\to+\infty,\\
&\psi(t,p)=\psi(t-m,p-m),\ m\in\mathbb Z.
\end{split}
\ee
The leading order term in $\lte$, which is obtained by dropping the big-$O$ term in \eqref{fe-1-nosm}, is given by
\be\label{fe-psi-nosm}\begin{split}
g_\e^{(1)}(\xe)&:= \frac{{{\kappa}}}{\e^{1/2}}\sum_k \rho(\al_k)\psi\left(\al_k\cdot\check x +\frac{\al_k\cdot x_0}{\e}, \frac{H(\al_k)}{\e}\right)\chi(\al_k)\\
&=\frac{1}{\e}\sum_k \rho(\al_k)\psi\left(\al_k\cdot\check x+\frac{\al_k\cdot x_0}{\e}, \frac{\al_k\cdot x_0}{\e}-\frac{p(\al_k)}{\e}\right)\chi(\al_k){{\kappa}}\e^{1/2}.
\end{split}
\ee
Pick a sufficiently large $A>0$, and introduce two sets:
\be\label{two-sets}
\Omega_a:=\{\al\in\ts(\chi):\,|\al-\theta_0|\le A\e^{1/2}\},\ \Omega_b:=\{\al\in\ts(\chi):\,|\al-\theta_0|> A\e^{1/2}\}.
\ee
The sum in \eqref{fe-psi-nosm} splits into two:
\be\label{fe-psi-parts}\begin{split}
g_\e^{(1)}(\xe)=\sum_k (\cdot)=\sum_{\al_k\in\Omega_a} (\cdot)+\sum_{\al_k\in\Omega_b} (\cdot)=:g_\e^{(1a)}(\xe)+g_\e^{(1b)}(\xe).
\end{split}
\ee
 We have
\be\label{ri-term}\begin{split}
\al_k\cdot\check x&=h+O(\e^{1/2}),\\ 
\frac{\al_k\cdot x_0}{\e}&=\frac{\Theta_0\cdot x_0}{\e}+\frac{\Theta_0^\perp\cdot x_0({{\kappa}}\e (q_\al+k)-\theta_0)}{\e}-\frac{\Theta_0\cdot x_0(\al_k-\theta_0)^2}{2\e}+O(\e^{1/2})\\ 
&=A_\e+a k-\frac{\Theta_0\cdot x_0(\al_k-\theta_0)^2}{2\e}+O(\e^{1/2}),\ \al_k\in\Omega_a,\ a:=(\Theta_0^\perp\cdot x_0){{\kappa}}.
\end{split}
\ee
From \eqref{psi-def} and the property IK3 of $\ik$ it follows that 
\be\label{another-psi-pr}
\psi(t+\e,p)-\psi(t,p)=O(\e),\ \psi(t,p+\e)-\psi(t,p)=O(\e^{1/2})
\ee
when $t-p$ is bounded. By \eqref{p0asymp} and the third line in \eqref{psi-props}, this gives
\be\label{g1a-st1}\begin{split}
\e & g_\e^{(1a)}(\xe)\\
&=\sum_{\al_k\in\Omega_a} (\rho+O(\e^{1/2})) \psi\biggl(h+A_\e+a k-\frac{\Theta_0\cdot x_0(\al_k-\theta_0)^2}{2\e}+O(\e^{1/2}), \\
&\qquad\qquad A_\e+a k-\frac{(R+\Theta_0\cdot x_0) (\al_k-\theta_0)^2}{2\e}+O(\e^{1/2})\biggr){{\kappa}}\e^{1/2}\\
&=\rho\sum_{\al_k\in\Omega_a} \psi\left(h+r_k-\frac{\Theta_0\cdot x_0(\al_k-\theta_0)^2}{2\e}, r_k-\frac{(R+\Theta_0\cdot x_0) (\al_k-\theta_0)^2}{2\e}\right){{\kappa}}\e^{1/2}\\
&\hspace{1cm}+O(\e^{1/4}),\ r_k:=\left\{A_\e+a k\right\},
\end{split}
\ee
where we have used that the sums in \eqref{g1a-st1} are bounded as $\e\to0$. Here and in what follows, $\Theta_0^\perp=(-\sin\theta_0,\cos\theta_0)$, and $\{t\}$, $t\in\br$, denotes the fractional part of a number. Set $\tilde \al_k:=(\al_k-\theta_0)/\e^{1/2}$. Clearly, $\tilde \al_{k+1}-\tilde \al_k={{\kappa}}\e^{1/2}$. If $x_0$ is generic, i.e. $a$ is irrational, then $r_k$ are uniformly distributed mod 1 (see \cite{KN_06, kat_2017, kat19a}). Taking the limit as $\e\to0$ and arguing similarly to \cite{kat_2017, kat19a, kat19b} gives:
\be\label{ga-lim}\begin{split}
\lim_{\e\to0}&\e g_\e^{(1a)}(\xe)\\
&= \rho\int_{|\tilde \al|\le A}\int_0^1 \psi\biggl(h+r-\frac{(\Theta_0\cdot x_0)\tilde\al^2}2, 
r-\frac{(R+\Theta_0\cdot x_0)\tilde \al^2}2\biggr)dr d\tilde \al\\
&= -\frac{4f_+}{\pi}\int_0^{A\sqrt{\frac{R}2}}\int_0^1 \psi\left(h+r-\frac{\Theta_0\cdot x_0}{R}t^2, r-\left(1+\frac{\Theta_0\cdot x_0}{R}\right)t^2\right)dr dt.
\end{split}
\ee

Next, consider $g_\e^{(1b)}$. Since $p'(\theta_0)=0$, $p''(\theta_0)>0$, and $\ts(\chi)$ is sufficiently small, there exists $c_1>0$ such that $p(\al)>c_1(\al-\theta_0)^2$ when $\al\in\Omega_b$. Hence, it follows from \eqref{psi-props} that
\be\label{psi-bnd}
\left|\psi\left(\al\cdot\check x+\frac{\al\cdot x_0}{\e}, \frac{\al\cdot x_0}{\e}-\frac{p(\al)}{\e}\right)\right|\le c_2\left[\frac{(\al-\theta_0)^2}\e\right]^{-3/2},\ \al\in\Omega_b,
\ee
for some $c_2>0$. Here we use that $\check x$ is confined to a bounded set. Therefore $g_\e^{(1b)}$ admits the bound
\be\label{ge1b-bnd}
|\e g_\e^{(1b)}(\xe)|= O(\e^{1/2})\sum_{k=A/\e^{1/2}}^{O(1/\e)}\left[\frac{(\e k)^2}\e\right]^{-3/2}=O(1/A^2),
\ee
and the last big-$O$ is uniform in $\e$.

Finally, we estimate the contribution to $\lte$ that comes from the big-$O$ term in \eqref{fe-1-nosm}. As is easily seen,
\be\label{less-sing}\begin{split}
\sum_j\ik''\left(\frac{\al\cdot\xe-\e j}{\e}\right)
O\left((\e j-H(\al))_+^{3/2}\right)=
\begin{cases} O(\e^{3/2}),& \al\in \Omega_a,\\
O(\e^2)|\al-\theta_0|^{-1},& \al\in \Omega_b.
\end{cases}
\end{split}
\ee
For example, the top case follows because the number of nonzero terms in the sum is finite, and $\al\cdot\xe-H(\al)=O(\e)$ when $\al\in\Omega_a$. Hence the big-$O$ term on the left and the sum are $O(\e^{3/2})$. See \eqref{big-F} and \eqref{F-bounds} in Appendix~\ref{sec:prflem} for more general estimates of this kind.

Substituting \eqref{less-sing} into \eqref{fe-1-nosm} shows that this remaining contribution is 
\be\label{ls-est}
\e^{-1}\left(\sum_{k=1}^{O(\e^{-1/2})}O(\e^{3/2})+\sum_{k=O(\e^{-1/2})}^{O(\e^{-1})}O(\e^2)(\e k)^{-1}\right)=O(\ln(1/\e)).
\ee

Combining \eqref{ga-lim}, \eqref{ge1b-bnd}, and \eqref{ls-est} and using that $A>0$ can be arbitrarily large gives
\be\label{final-lim}\begin{split}
\lim_{\e\to0}\e \lte(\xe)=  -\frac{4 f_+}{\pi}\ioi\int_0^1 \psi\left(h+r-\frac{\Theta_0\cdot x_0}{R}t^2, r-\left(1+\frac{\Theta_0\cdot x_0}{R}\right)t^2\right)dr dt.
\end{split}
\ee
By \eqref{psi-def} and \eqref{final-lim}, the unit edge response equals
\be\label{er-hr}\begin{split}
\Phi(h):=& -\frac{4}{\pi}\ioi\int_0^1 \psi\left(h+r-\frac{\Theta_0\cdot x_0}{R}t^2, r-\left(1+\frac{\Theta_0\cdot x_0}{R}\right)t^2\right)dr dt\\
=&  -\frac{4}{\pi}\ioi\int_{\br} \ik''(h+r)(t^2-r)_+^{1/2}dr dt.
\end{split}
\ee
The integral in \eqref{er-hr} can be significantly simplified. Skipping the prefactor and integrating by parts once gives
\be\label{er-hr-byparts}\begin{split}
&  \ioi\int_{\br} \ik''(h+r)(t^2-r)_+^{1/2}dr dt
=\frac12\ioi\int_{\br} \ik'(h+r)(t^2-r)_+^{-1/2}dr dt\\
&=\frac12\lim_{A\to\infty}\int_{-\infty}^{A^2}\ik'(h+r)\int_{r_+^{1/2}}^A (t^2-r)^{-1/2}dt dr\\
&=\frac12\lim_{A\to\infty}\int_{-\infty}^{A^2}\ik'(h+r)\left(\log((A^2-r)^{1/2}+A)-\frac12\log|r|\right)dr=\frac14 \int_{\br}\ik(h+r)\frac{dr}r.
\end{split}
\ee
When evaluating the limit as $A\to\infty$ in \eqref{er-hr-byparts} we used that $\ik$ is compactly supported. Combining \eqref{final-lim}--\eqref{er-hr-byparts} and using that a smooth part of $\hat f$ leads to a bounded contribution to $\lte$ proves the following result.
\begin{lemma}\label{lem-phi-lim} Let $f$ be given by \eqref{f_def} and satisfy conditions (f1)--(f3). Suppose $x_0\in\s$ is generic, and the line $\{x\in\br^2:\,(x-x_0)\cdot\dir_0=0\}$ is not tangent to $\s$ anywhere except at $x_0$. If $\ts(\chi)$ is contained in a  small neighborhood of $\theta_0$, $\chi(\theta_0)=1$, and $\lte$ is given by \eqref{fe-0}, one has 
\be\label{final-lim-final}
\lim_{\e\to0}\e \lte(\xe)= (f_+(x_0)-f_-(x_0))\int_{\br}\frac{\ik(h-r)}{\pi r}dr,
\ee
where $f_\pm(x_0)$ are the limiting values of $f$ at $x_0$ from the positive and negative sides of $\s$, respectively.
\end{lemma}
Note that \eqref{final-lim-final} is consistent with Theorem 5.4.1 in \cite{rk}, i.e. the edge response is just a smoothed version of the ideal response (or, CTB) given in (5.4.4) of \cite{rk}. In \cite{rk}, smoothing is due to a smoothing kernel, and here smoothing is due to finite data sampling. This is consistent with the general situation, see Theorems~\ref{lead-thrm-pos}, \ref{lead-thrm-zero} below. In these theorems, $\mu$ is the ideal transition behavior of the reconstruction from continuous data, or CTB (cf. Lemma~\ref{cont-trans}).
%\begin{proof}
%If $h\to\infty$, we can integrate by parts with respect to $r$, and obtain
%\be\label{Phi-plinf}\begin{split}
%\Phi(h)\to \frac1{\pi} \ioi \frac{d\al}{(h+\al^2)^{3/2}}=\frac1{\pi h} \ioi \frac{d\al}{(1+\al^2)^{3/2}}=\frac1{\pi h},h\to+\infty.
%\end{split}
%\ee
%Next, rewrite $\Phi$ as follows:
%\be\label{Phi-mininf}
%\Phi(h)=-\frac{2}{\pi}\ioi\ioi \ik''(h-r+v)  \frac{dv}{v^{1/2}} r^{1/2}dr.
%\ee
%Clearly,
%\be\label{inner-int}
%\ioi \ik''(v-q)  \frac{dv}{v^{1/2}} =\frac34 q^{-5/2}+O(q^{-7/2}),\ q\to+\infty.
%\ee
%Hence
%\be\label{Phi-mininf-v2}
%\Phi(h)\to-\frac3{2\pi}\ioi \frac{r^{1/2}}{(r-h)^{5/2}} dr=-\frac3{2\pi}\frac2{3|h|}
%=\frac1{\pi h},h\to-\infty.
%\ee
%The magnitude of the error term in \eqref{Phi-lims} is obvious.
%\end{proof}

\subsection{Line artifact}\label{sec:LTline}
In this subsection we consider the effect of a straight line edge in $f$ on $\lte$. We show that a line edge may create a global artifact along the line containing the edge. The goal here is not to investigate the most general situation, but to understand the artifact qualitatively. Hence we consider a simple $f$, which vanishes outside some domain $D$ with convex boundary, and equals 1 close to a flat side of $\pa D$. Assume $\theta_0=0$ (i.e., $\Theta_0=(1,0)$) and 
\be\label{coords}
x_0=(H,b_0),\ P_1=(H,b_1),\ P_2=(H,b_2),\ b_1<b_2,\ a\not\in[b_1,b_2],
\ee
see Figure~\ref{fig:line_edge}.

\begin{figure}[h]
{\centerline{
{\epsfig{file=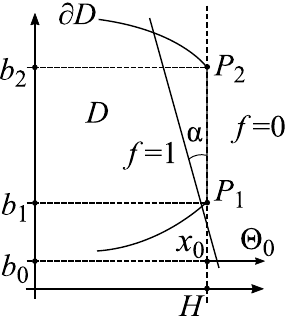, width=3.2cm}}
}}
\caption{Illustration of $f$ with jump discontinuity along a line segment.}
\label{fig:line_edge}
\end{figure}

Similarly to \eqref{two-sets}, split $\ts(\chi)$ into two sets:
\be\label{two-sets-line}
\Omega_a:=\{\al\in\ts(\chi):\,|\al|< A\e\},\ \Omega_b:=\{\al\in\ts(\chi):\,|\al|\ge A\e\},
\ee
for some sufficiently large $A>0$. We can select $A>0$ so large that no line $\{x\in\br^2:\,(x-(x_0+\e\check x))\cdot\al=0\}$, $\al\in\Omega_b$, intersects the line segment $[P_1,P_2]$ for all $\e>0$ sufficiently small. Recall that $\check x$ is confined to a bounded set. From \eqref{data_pts}, the number of $\al_k\in\Omega_a$ is uniformly bounded as $\e\to0$.

Let $g_\e^{(a)}$ and $g_\e^{(b)}$ denote the contributions to $\lte$ coming from $\al_k\in\Omega_a$ and $\al_k\in\Omega_b$, respectively. To compute $g_\e^{(a)}$, introduce the function
\be\label{aux-fn-edge}\begin{split}
\phi(p;t_1,t_2)&=\begin{cases} b_2-b_1,& p\le \min(t_1,t_2)\\
(b_2-b_1)\frac{\max(t_1,t_2)-p}{|t_2-t_1|},& \min(t_1,t_2)\le p\le \max(t_1,t_2)\\
0,&p\ge \max(t_1,t_2).
\end{cases}
\end{split}
\ee
This function models the leading singular behavior of $\hat f(\al,p)$ near $(\al,p)=(0,H)$:
\be\label{rt-appr-line}
\hat f(\al,p)=\phi(p;\al\cdot P_1,\al\cdot P_2)+O(\e),\
\al\in\Omega_a,\ p-H=O(\e).
\ee
%The summation in \eqref{lt-line-nosm} can be confined to a set of $i$ such that $|\al_k|\le A\e$ for some $A$ large enough. The reason is that $\hat f(\al,p)$ is smooth outside of this set (for $p$ in a neighborhood of $\al\cdot x_0$). Recall that $\Delta_\al={{\kappa}}\e$. Then, to leading order,
As is easily checked, 
\be\begin{split}\label{phi-props}
&\phi(p;t_1,t_2)=\phi(rp;rt_1,rt_2),\ r>0, \\
&\phi(p-r;t_1-r,t_2-r)=\phi(p;t_1,t_2),\ r\in\br,\\
&\phi(p+O(\e);t_1+O(\e),t_2+O(\e))=\phi(p;t_1,t_2)+O(\e) \text{ if } |t_1-t_2|>\de
\end{split}
\ee
for some $\de>0$. Thus,
\be\label{lt-line-nosm}\begin{split}
&g_\e^{(a)}(\xe)\\
&=-\sum_{\al_k\in\Omega_a}\frac1{2\pi\e^2}
\sum_j\ik''\left(\frac{\al_k\cdot \xe-\e j}{\e}\right)\left(\phi(\e j;\al_k\cdot P_1,\al_k\cdot P_2)+O(\e)\right){{\kappa}}\e\\
&=-\frac{\kappa}{2\pi\e}\sum_{\al_k\in\Omega_a}
\sum_j\ik''\left(h+\frac{\al_k\cdot x_0}{\e}+O(\e)-j\right)\phi\left(j;\frac{\al_k\cdot P_1}{\e},\frac{\al_k\cdot P_2}{\e}\right)+O(1).
\end{split}
\ee
%\be\label{lt-line-nosm}\begin{split}
%g_\e^{(a)}(\xe)&=\sum_{\al_k\in\Omega_a}\frac1{\e^2}
%\sum_j\ik\left(\frac{\al_k\cdot \xe-\e j}{\e}\right)(\phi(\e j;\al_k\cdot x_1,\al_k\cdot x_2)+O(\e))\chi(\al_k){{\kappa}}\e\\
%%\ee
%%\be\label{lt-line-v2}
%&=\frac{1}{\e}\sum_{|k|\le A}
%\sum_j\ik''\left(h+\frac{H}{\e}+a{{\kappa}} k-j\right)\phi\left(j;\frac{H}{\e}+b_1{{\kappa}} k,\frac{H}{\e}+b_2{{\kappa}} k\right)\\
%&=\frac{1}{\e}\sum_{|k|\le A}
%\sum_j\ik''\left(h+r_k-j\right)\phi\left(j;r_k+(b_1-a){{\kappa}} k,r_k+(b_2-a){{\kappa}} k\right),\\ r_k&:=\left\{\frac{H}{\e}+a{{\kappa}} k\right\}.
%\end{split}
%\ee
%Up to the terms of order $O(\e)$, the tilted edges can be replaced by vertical edges (see dashed lines in Figure~ref{}). Likewise, integrals over tilted lines can be replaced by integrals over vertical lines.
Suppose, for simplicity, that none of the angles $\al_k=\kappa\e(q_\al+k)$ equals zero, i.e. $q_\al\not\in\mathbb Z$ (cf. \eqref{data_pts}). In this case, $(\al_k\cdot (P_2-P_1))/\e$, is bounded away from zero, and the last equation in \eqref{phi-props} applies. 
%Also,  $\al\cdot x_0,\al\cdot P_1,\al\cdot P_2=O(\e^2)$, $\al\in\Omega_a$. 
Using the second and third lines in \eqref{phi-props} we find from \eqref{lt-line-nosm}: 
\be\label{lt-line-st2}\begin{split}
\e g_\e^{(a)}(\xe)
&=-\frac{\kappa}{2\pi}\sum_{\al_k\in\Omega_a}
\sum_j\ik''\left(h+r_k-j\right)\phi\bigl(j-r_k;(b_1-b_0){{\kappa}} (q_\al+k),\\
&\hspace{2.5cm}(b_2-b_0){{\kappa}} (q_\al+k)\bigr)+O(\e),\ r_k:=\left\{\frac{H}{\e}+b_0{{\kappa}} (q_\al+k)\right\}.
\end{split}
\ee

The remaining term is 
\be\label{lt-line-gb}
g_\e^{(b)}(\xe)=-\sum_{\al_k\in\Omega_b}\frac1{2\pi\e^2}
\sum_j\ik''\left(\frac{\al_k\cdot \xe-p_j}{\e}\right)\hat f(\al_k,p_j) \Delta\al.
\ee
By construction, $\hat f(\al,p)$ is smooth and bounded with all derivatives in a $O(\e)$-size neighborhood of any $(\al_k,p_j)$ such that $\al_k\in\Omega_b$ and $(\al_k\cdot \xe-p_j)/\e\in\ts(\ik)$. Hence it is easy to see that $g_\e^{(b)}(\check x)$ approaches a finite limit as $\e\to0$ independently of $\check x$. This limit depends on where $x_0$ is located relative to the segment $[P_1,P_2]$. For example, if $b_0<b_1$, as shown in Figure~\ref{fig:line_edge}, then $\hat f(\al,\al\cdot x_0)\equiv0$ if $\al\in\Omega_b$, $\al<0$, and we have
\be\label{gb-lim}
\lim_{\e\to0}g_\e^{(b)}(\xe)=-\frac1{2\pi}\int_{0^+}^{\pi/2} (\pa_p^2\hat f)(\al,\al\cdot x_0)\chi(\al) d\al.
\ee
Thus, $g_\e^{(b)}(\xe)=O(1)$. Since $\lte=g_\e^{(a)}+g_\e^{(b)}$, \eqref{lt-line-st2} shows that straight edges of $f$ create non-local artifacts in $\lte$ that are of the same order of magnitude as useful singularities (see \eqref{final-lim}), i.e. of order $O(1)/\e$. The $O(1)$ term has a fairly weak (and irregular) $\e$-dependence (via $r_k$).

\subsection{Effect of remote singularities.} \label{sec:LTremote}
Let $\Theta_0$ be the direction such that the line $(x-x_0)\cdot\Theta_0=0$ is tangent to $\s$ at some $z_0\not=x_0$ and $\theta_0\in(-\pi/2,\pi/2)$. Suppose that the curvature of $\s$ at $z_0$ is not zero. The main formula is \eqref{fe-psi-nosm}, where still $p(\theta_0)=0$, but $p'(\theta_0)\not=0$, i.e. $p(\al)$ is no longer quadratic near $\al=\theta_0$. As before, we suppose that $\ts(\chi)$ is sufficiently small and $\chi(\theta_0)=1$. Additionally, $p'(\al)\not=0$ on $\ts(\chi)$. Represent $\lte=g_\e^{(1)}+g_\e^{(2)}$, where $g_\e^{(1)}$ and $g_\e^{(2)}$ correspond to the leading and big-$O$ terms in \eqref{g-ass}, respectively. Thus, 
\be\label{remsing-1}\begin{split}
g_\e^{(1)}(\xe)&= \frac{\kappa}{\e^{1/2}}\sum_{k} \rho(\al_k)\psi\left(\al_k\cdot \check x+\frac{\al_k\cdot x_0}{\e}, \frac{H(\al_k)}{\e}\right)\chi(\al_k).
\end{split}
\ee
From the properties $\al\cdot z_0-H(\al)=O((\al-\theta_0)^2)$ and $\dir_0\cdot(x_0-z_0)=0$ it follows that there exists $c>0$ such that $|\al\cdot x_0-H(\al)|>c|\al-\theta_0|$, $\al\in\ts(\chi)$. Together with the second line in \eqref{psi-props} this implies that the sum in \eqref{remsing-1} is uniformly bounded as $\e\to0$. From IK1--IK3, it follows similarly to \eqref{psi-props}, \eqref{another-psi-pr} that 
\be\label{psi-bound}
|\psi(t,p+\e)-\psi(t,p)|\le c_1\begin{cases}
|\e|/|t-p|^{5/2},& |t-p|>c_2,\\
|\e|^{1/2},& |t-p|\le c_2,\end{cases}
\ee
for some $c_{1,2}>0$. Representing $H(\al)$ in the form
\be\label{H-form}\begin{split}
H(\al_k)&=\al_k\cdot z_0+O((\al_k-\theta_0)^2)
=\al_k\cdot x_0+\al_k\cdot(z_0-x_0)+O((\al_k-\theta_0)^2)\\
&=\al_k\cdot x_0+\dir_0^\perp\cdot(z_0-x_0)(\al_k-\theta_0)+O((\al_k-\theta_0)^2),
\end{split}
\ee
and combining this with \eqref{psi-bound} implies that replacing $H(\al_k)$ with the linear part on the right in \eqref{H-form}, and noticing that $\dir_0^\perp\cdot(z_0-x_0)\not=0$, changes the value of the sum in \eqref{remsing-1} by $O(\e^{1/2})$. Indeed, the error term is an expression of the kind $O(1)\sum_{k=1}^{O(1/\e)} (\e k^2)/k^{5/2}=O(\e^{1/2})$. In this calculation we assume that $\ts(\chi)$ is sufficiently small and $|H(\al)-\al\cdot x_0|\le 0.5|\dir_0^\perp\cdot(z_0-x_0)||\al-\theta_0|$, $\al\in\ts(\chi)$.

The second line in \eqref{psi-props} implies that replacing $\rho(\al_k)$ and $\chi(\al_k)$ with $\rho=\rho(\dir_0)$ and $\chi(\dir_0)=1$, respectively, changes the value of the sum by $O(\e^{1/2})$. Hence,
\be\label{remsing-2}\begin{split}
\e^{1/2}g_\e^{(1)}(\xe)= &\kappa\rho\sum_{k} \psi\left(\dir_0\cdot \check x+\frac{\al_k\cdot x_0}{\e},\frac{\al_k\cdot x_0}{\e}+\frac{\dir_0^\perp\cdot(z_0-x_0)(\al_k-\theta_0)}{\e}\right)\\
&+O(\e^{1/2}).
\end{split}
\ee
Similarly to \eqref{less-sing}, it is easy to show that
\be\label{less-sing-v2}\begin{split}
\left|\frac1{\e^2}\sum_j\ik''\left(\frac{\al\cdot\xe-\e j}{\e}\right)
O\left((\e j-H(\al))_+^{3/2}\right)\right|\le
\frac{c}{\e^{1/2}+|\al\cdot\xe-H(\al)|^{1/2}}
\end{split}
\ee
for some $c>0$. This gives $g_\e^{(2)}(\xe)=O(1)$, $\e\to0$. Combining the results produces
\be\label{remsing-3}\begin{split}
\lte(\xe)&= \frac{{{\kappa}}\rho}{\e^{1/2}}\sum_k \psi\left(h+r_k, r_k+\dir_0^\perp\cdot(z_0-x_0)\left[\kappa(q_\al+k)-\frac{\theta_0}{\e}\right]\right)+O(1),\\ 
r_k:&=\left\{\frac{\al_k\cdot x_0}{\e}\right\}.
\end{split}
\ee
As was mentioned, the sum in \eqref{remsing-3} is uniformly bounded, and there is no reason why it should identically equal zero. Thus, even convex pieces of $\s=\text{singsupp}(f)$ may create non-local artifacts when reconstructing from discrete data, and their strength grows like $\e^{-1/2}$ as $\e\to0$. These artifacts are expected to be of irregular, ripple-like shape due to the irregular behavior of the terms $r_k$ and $\theta_0/\e$. This also implies that $\lte$ does not generally converge to $f_{\Lambda}$ pointwise as $\e\to0$ if $f$ has jump discontinuities.  
%Using this information, we can estimate whether certain jumps will be visible for the given discretization $\e$.

\section{Computation of leading singularities in the continuous data case}\label{sec:gen}

Here we derive convenient formulas that are used for resolution analysis in all dimensions $n\ge 2$ and for a variety of singularities and reconstruction operators. The latter can be preserving the degree of smoothness or singularity-enhancing.

Suppose $f\in\CE'(\br^n)$ is a compactly supported distribution, and $(x_0,\xi_0)\in WF(f)$. 
{For convenience of the reader we remind the definition of the wave front set (see \cite{hor}, section 8.1).
\begin{definition} Let $f\in\mathcal D'(\br^n)$ be a distribution. The wave front set of $f$ is the complement of all pairs $(x_0,\xi_0)\in\br^n\times(\br^n\setminus 0)$ such that there exists $\phi\in C_0^\infty(\br^n)$ with $\phi(x_0)\not=0$ and an open cone $\Xi\ni\xi_0$ so that $|\mathcal F(\phi f)(\xi)|\le c_N(1+|\xi|)^{-N}$ for some $c_N>0$ and all $\xi\in\Xi$ and $N\ge 1$.
\end{definition}
}
Let $\Omega\subset S^{n-1}$ be a small neighborhood of $\dir_0:=\xi_0/|\xi_0|$.  We assume that $f$ is given by
\be\label{f-def}
f(x)=\frac1{(2\pi)^n}\int_{\br^n}\fs(\xi)e^{i(H(\xi)-\xi\cdot x)}d\xi,\ 
\fs\in C^\infty(\br^n),\ H(\xi)\in C^\infty(\br^n\setminus 0),
\ee
is sufficiently regular, and its Radon transform exists in the usual sense of functions. Here $H$ is real valued and homogeneous of degree one, $H'(\xi_0)=x_0$, and 
\be\label{f-lim}\begin{split}
&\fs(\la\al)\sim \omega(\al)\sum_{j\ge 0} \la^{-s_j}\fs_j(\al),\ \la\to+\infty;\\  
&s_0\ge (n+1)/2,\ s_0<s_1<\dots,\ s_j\to\infty,\ j\to\infty;\\ 
&\fs_j\in\coi(\Omega\cup(-\Omega)),\ j\ge 0; \\
&\omega(\al):=\frac{|\det \check H''(\al)|^{1/2}e^{-\frac{i\pi}4\text{sgn}\check H''(\al)}}{(2\pi)^{(n-1)/2}},\ \al\in \Omega\cup(-\Omega).
\end{split}
\ee 
The expansion in \eqref{f-lim} can be differentiated with respect to $\al$ and $\la$ any number of times, and all the resulting expansions are uniform with respect to $\al\in S^{n-1}$.
More specifically, for any $J\in\mathbb N$, $l=0,1,2,\dots$, and any multiindex $\nu$, there exists $a_{Jl\nu}\in \coi(\Omega\cup(-\Omega))$ so that
\be\label{vfn-exp-unif}
\left|\pa_\la^l\pa_\al^\nu \biggl(\fs(\la\al) - \omega(\al)\sum_{j=0}^{J-1} \la^{-s_j}\fs_j(\al)\biggr)\right|\le  a_{Jl\nu}(\al)\la^{-(s_J+l)},\ \la\ge 1,\al\in S^{n-1}.
\ee

In \eqref{f-lim} and everywhere below, $\check H''(\al)$, $\al\in S^{n-1}$, denotes the Hessian matrix of $H(\xi)$ {restricted to the plane tangent to $S^{n-1}$ at $\al$, $\xi\cdot\al=1$, and evaluated at $\xi=\al$. For example, if $\al=(0,\dots,0,1)$, then
\be\label{chH}
(\check H''(\al))_{jk}=\left.\frac{\pa^2 H(\xi_1,\dots,\xi_{n-1},1)}{\pa\xi_j\pa\xi_k}\right|_{\xi_1=\dots=\xi_{n-1}=0},\ 1\le j,k\le n-1.
\ee
In particular, $\check H''(\al)$ is a $(n-1)\times(n-1)$ square matrix}. Also, $\text{sgn}\check H''(\al)$ is the number of positive eigenvalues of $\check H''(\al)$ minus the number of negative eigenvalues of $\check H''(\al)$. As is seen, $H(\xi)$ is the homogeneous of degree one extension of $H(\al)$ used in Section~\ref{ltwosm} from $S^{n-1}$ to $\br^n$: $H(\xi)=|\xi|H(\xi/|\xi|)$. An additional assumption is
\be\label{more-assns}\begin{split}
&\check H''(\al) \text{ is negative definite on $\Omega$}.
\end{split}
\ee
Clearly, $\check H''(-\al)=-\check H''(\al)$, so $\check H''(\al)$ is positive definite on $-\Omega$. Therefore, $\text{sgn}\check H''(\pm\al)=\mp(n-1)$, $\al\in\Omega$, and 
\be\label{omega-pm}
\omega(-\al)=e^{-i(n-1)\frac\pi2}\omega(\al),\ \al\in\Omega.
\ee

{\begin{remark} Associated with $H$, there is a smooth surface of codimension one
\be\label{surf-def}
\s:=\{x\in\br^n:\, x=H'(\al),\al\in \Omega\}.
\ee
Recall that $H'(\al)$ is the derivative $H'(\xi)$ evaluated at $\xi=\al\in S^{n-1}$ (as opposed to the derivative on the unit sphere).
By Proposition 25.1.3 in \cite{hor4}, $f$ is a conormal distribution. In particular, its wave front set is contained in the conormal bundle of $\s$: $WF(f)\subset\{(x,\xi)\in \br^n\times(\br^n\setminus 0): x=H'(\xi),\pm\xi/|\xi|\in\Omega\}$.  See also section 18.2 and definition 18.2.6 in \cite{hor3} for a formal definition and in-depth discussion of conormal distributions. 
\end{remark}}

We want to reconstruct some image of $f$ from its Radon transform $\hat f(\al,p)$ in a neighborhood of $x_0$ using an operator $\CB$ of the form
\be\label{B-oper}
(\CB\hat f)(x):=\int_{S^{n-1}}\int_{\br} B(\al,\al\cdot x-p)\hat f(\al,p)dp d\al,\
B(\al,p)=\frac1{2\pi}\int\tilde B(\al,\la)e^{-i\la p}d\la.
\ee 
We assume that $\tilde B\in C^{\infty}(S^{n-1}\times \br)$ and
\be\label{B-lim}\begin{split}
&\tilde B(-\al,-\la)=\tilde B(\al,\la);\ \tilde B(\al,\la)\sim \sum_{j\ge0}\la^{\bt_j}\tilde B_j(\al),\ \la\to+\infty;\\
&\bt_0>\bt_1> \dots;\ \bt_j\to-\infty,\ j\to\infty;\ \tilde B_j(\al)\in C^{\infty}(S^{n-1}),\ j\ge 0,
\end{split}
\ee 
where the expansion can be differentiated with respect to $\al$ and $\la$ term by term any number of times, and it remains uniform with respect to $\al\in S^{n-1}$. Thus, $\CB:\CE'(Z_n)\to\CD'(\br^n)$ is an FIO with the same canonical relation as the adjoint Radon transform. 
Using that $\tilde B$ is even implies
\be\label{B-lim-pm}\begin{split}
\tilde B(\al,\la)\sim \sum_{j\ge0}(\la_+^{\bt_j}\tilde B_j(\al)+\la_-^{\bt_j}\tilde B_j(-\al)),\ \la\to\infty.
\end{split}
\ee

The standing assumptions are
\be\label{kappa-assns}
{\kappa_1}:=s_0-\frac{n+1}2\ge 0,\ {\kappa_2}:=\bt_0-s_0-\frac{n-3}2\ge 0.
\ee
An additional condition is
\be\label{B-adtnl}
\fs_0(-\al)\tilde B_0(-\al)=-\fs_0(\al)\tilde B_0(\al),\ \al\in \Omega\cup(-\Omega), \text{ if }\kappa_2=0.
\ee
See the text following \eqref{A-def} for the meaning of this condition.

To simplify notations, in what follows we write $\bt$ and $s$ for $\bt_0$ and $s_0$, respectively. 
The goal is to determine what the distribution $\CB\hat f$ looks like in a neighborhood of $x_0$. The first step is to determine what $f$ looks like near $x_0$. We are not interested in a complete description of $f$, but only in its leading order singularity near $x_0$, which is denoted $f_0$. 
\begin{definition}[\cite{kat99_a}]\label{lead-sing}
Given a distribution $f\in{\mathcal D}'(\br^n)$ and a point $x_0\in\br^n$, suppose there exists a distribution $f_0\in{\mathcal D}'(\br^n)$ so that for some $m_0\ge0$ and some $a\in\br$ the following equality holds 
\be\label{losing}
\lim_{\e\to0} \e^{m-a}(f,P_m(\pa_x)w_\e)=(f_0,P_m(\pa_x)w_1),\ 
w_\e(x):=\e^{-n}w((x-x_0)/\e),
\ee
for any $w\in\coi(\br^n)$, any homogeneous polynomial $P_m(x)$ of degree $m$, and any $m\ge m_0$. Then we say that $f_0$ is the leading order singularity of $f$ at $x_0$, and the corresponding notation is $f(x_0+\e\check x)\sim \e^a f_0(\check x)$, where $\hat x$ is confined to a bounded set.
\end{definition}

Let $\lceil \cdot\rceil$ denote the ceiling function: $\lceil t\rceil:=n+1$ if $t\in (n,n+1)$ for some $n\in\mathbb Z$, and $\lceil t\rceil:=n+1$ if $t=n$ for some $n\in\mathbb Z$. The following lemma is proven in Appendix~\ref{sec:mu-deriv}.

\begin{lemma}\label{lem:mu-der} For any $m\ge m_0:=\lceil\kappa_1\rceil$, one has 
\be\label{limf-res-0}
\lim_{\e\to0}\e^{m-{\kappa_1}}(f,P_m(\pa_x)w_\e)
=\frac1{2\pi}
\int P_m(-i\la \dir_0)\tilde w(-\la\dir_0)\tilde\mu(\la)d\la,
\ee 
where $\tilde w=\CF w$, and the distribution $\tilde\mu\in {\mathcal S}'(\br)$ is given by
\be\label{cpm-complete}\begin{split}
\tilde\mu(\la):=(2\pi)^{-(n-1)}(v_+\la_+^{-({\kappa_1}+1)}+v_-\la_-^{-({\kappa_1}+1)}),\ v_{\pm}:=\fs_0(\pm\dir_0).
\end{split}
\ee
\end{lemma}
%The value of $m_0$ for $f$ in \eqref{f-def} is determined later. By \eqref{losing}, we look at the asymptotics of the integral
%\be\label{wf-lim}\begin{split}
%\frac{1}{(2\pi)^n}\ioi\int_{S^{n-1}}P_m(-i\la\al)\tilde w(-\e\la\al)\fs(\la\al)e^{i\la (H(\al)-\al\cdot x_0)}\la^{n-1}d\al d\la,\ \e\to0,
%\end{split}
%\ee 
%where $\tilde w={\mathcal F} w$. The appropriate prefactor $\e^{-a}$ is determined later. Upon changing variables $\la\to\e\la$, \eqref{f-lim} implies that we have to compute the leading order behavior of the following integral as $\e\to0$:
%\be\label{keyf-lim}
%J_\e(\la):=\int_{S^{n-1}}P_m(\al)\tilde w(-\la\al) \omega(\al)\fs_0(\al)e^{i(\la/\e) (H(\al)-\al\cdot x_0)}d\al.
%\ee 
See \cite{GS} regarding the distributions $\la_\pm^a$. 
%In \eqref{limf-res} we have used that $m-{\kappa_1}>0$, which guarantees that the dominated convergence theorem can be applied to find the limit of \eqref{wf-lim} as $\e\to0$. This shows that $m_0$ in \eqref{losing} should be the smallest integer so that $m_0>{\kappa_1}$, i.e. $m_0=\lfloor\kappa_1\rfloor+1$. 
Let $\hat w:=Rw$ be the Radon transform of $w$. From \eqref{limf-res-0}, \eqref{cpm-complete},
\be\label{limf-res-rt}\begin{split}
\lim_{\e\to0}\e^{m-{\kappa_1}}(f,P_m(\pa_x)w_\e)&
=(f_0(p),P_m(\dir_0)\pa_p^m\hat w(\dir_0,p)),\ f_0(p)={\mathcal F}^{-1}(\tilde\mu),\\
f(x+\e\check x)&\sim \e^{{\kappa_1}}f_0(\check x\cdot\dir_0)=f_0(\e\check x\cdot\dir_0).
\end{split}
\ee 
From \eqref{limf-res-0}, \eqref{cpm-complete}, and equations 21, 24 and 18 in \cite{GS}, p. 360 (see also Appendix~\ref{sec:usef}), we compute $f_0$:
\begin{lemma}\label{lem:f0sing} If $f$ is given by \eqref{f-def} -- \eqref{more-assns} and \eqref{kappa-assns} holds, then the leading singularity of $f$ at $x_0$ is given by \eqref{limf-res-rt}, where
\be\label{f0-v1}\begin{split}
f_0(p)&=\frac{-1}{2(2\pi)^{n-1}\sin(\pi\kappa_1)\Gamma(\kappa_1+1)} \left[p_+^{{\kappa_1}}(q_1v_+ + \frac{v_-}{q_1} )+p_-^{{\kappa_1}}(q_1v_-+\frac{v_+}{q_1} )\right],\\ 
q_1&=\exp(i{\kappa_1}(\pi/2)),\ {\kappa_1}\not=0,1,2,\dots,
\end{split}
\ee
and
\be\label{f0-v2}\begin{split}
f_0(p)&=\frac{v_+ }{2(2\pi)^{n-1}i^{{\kappa_1+1}}{\kappa_1}!}p^{{\kappa_1}}\text{sgn}(p),\ {\kappa_1}=0,1,2,\dots,\ 
v_-=(-1)^{\kappa_1+1}v_+.
\end{split}
\ee
\end{lemma}
\begin{remark}
If $\kappa_1=0,1,2,\dots,$ and the second condition in \eqref{f0-v2} does not hold, then $f_0(p)$ can be computed using equations 18, 27, and 28 in \cite{GS}, pp. 360, 361. In this case,  $f_0(p)$ may involve logarithms for some values of $v_\pm$.
\end{remark}

If $f$ is real-valued, then $\tilde v_0(-\al)=\overline{\tilde v_0(\al)}$, and \eqref{f0-v1} simplifies slightly
\be\label{f0-v1-alt}\begin{split}
f_0(p)&=\frac{-1}{(2\pi)^{n-1}\sin(\pi\kappa_1)\Gamma({\kappa_1}+1)} \left[p_+^{{\kappa_1}}\text{Re}(q_1v_+)+p_-^{{\kappa_1}}\text{Re}(q_1\overline{v_+})\right],\\ 
{\kappa_1}&\not=0,1,2,\dots.
\end{split}
\ee
%Equations \eqref{f0-v2}, \eqref{f0-v1-alt} can be written in a more similar form:
%\be\label{f0-complete}\begin{split}
%f_0(p)&=(2\pi)^{-\frac{n+1}2}\frac{\pi i^{-{\kappa_1}}}{\Gamma({\kappa_1})}v_+ p^{{\kappa_1}-1}\text{sgn}(p),\ {\kappa_1}=1,2,\dots;\\
%f_0(p)&=(2\pi)^{-\frac{n+1}2}\frac{2\pi}{\sin(\pi{\kappa_1})\Gamma({\kappa_1})} \left[p_+^{{\kappa_1}-1}\text{Re}(qv_+)+p_-^{{\kappa_1}-1}\text{Re}(q\overline{v_+})\right],\\ &\hspace{3cm}{\kappa_1}\not=1,2,\dots.
%\end{split}
%\ee
As is seen from \eqref{limf-res-0}--\eqref{limf-res-rt}, $f_0$ is defined nonuniquely. Indeed, $\tilde\mu(\la)$ can be modified by adding $Q_{m_0-1}(\pa_\la)\de(\la)$, where $Q_{m_0-1}$ is any polynomial of degree not exceeding $m_0-1$, and \eqref{limf-res-0} will still hold. Hence, $f_0(x)$ is defined up to polynomials of degree not exceeding $\lfloor\kappa_1\rfloor$. 

In a similar fashion, to investigate $\CB \hat f$ consider the leading asymptotics of 
\be\label{w-lim}\begin{split}
&(\CB \hat f,P_m(\pa_x)w_\e)\\
&=\frac1{\pi}\ioi\int_{S^{n-1}}P_m(-i\la\al)\tilde w(-\e\la\al)\tilde B(\al,\la)\fs(\la\al)e^{i\la (H(\al)-\al\cdot x_0)}d\la d\al
\end{split}
\ee 
as $\e\to0$. In \eqref{w-lim} we used that $B$ is even. We give only an outline of the derivation. A rigorous argument follows the lines of the proof of Lemma~\ref{lem:mu-der}.
After changing variables $\eta=\e\la$, we compute similarly to \eqref{wf-lim}--\eqref{limf-explicit}: 
\be\label{key-lim}\begin{split}
&\int_{S^{n-1}} P_m(-i\al)\tilde w(-\eta\al)\tilde B(\al,\eta/\e) \fs((\eta/\e)\al)e^{i(\eta/\e) (H(\al)-\al\cdot x_0)}d\al\\
&=\left(\frac{\e}{\eta}\right)^{\frac{n-1}2}\sum_{\al\in\{\pm \dir_0\}}P_m(-i\al)\tilde w(-\eta\al) \tilde B(\al,\eta/\e)\frac{\fs((\eta/\e)\al)}{\omega(\al)}+O(\e^{\frac{n+1}2}),\ \e\to0.
\end{split}
\ee 
Combining \eqref{f-lim}, \eqref{B-lim}, \eqref{w-lim}, \eqref{key-lim}, and using $\la$ as the Fourier variable gives
\be\label{lim-res-1}\begin{split}
\lim_{\e\to0}&\e^{m+{\kappa_2}}(\CB\hat f,P_m(\pa_x)w_\e)\\
&=\frac1\pi\sum_{\al\in\{\pm \dir_0\}}\tilde B_0(\al)\fs_0(\al)\ioi P_m(-i\la \al)\tilde w(-\la\al)\la^{{\kappa_2-1}}d\la\\
&=(\mu(p),P_m(\dir_0)\pa_p^m\hat w(\dir_0,p)),\ m+\kappa_2>0,
\end{split}
\ee 
where
\be\label{mu-def-1}
\mu(t)=
\mathcal{F}^{-1}\left(\mu_+\la_+^{{\kappa_2-1}}+ \mu_-\la_-^{{\kappa_2-1}}\right),\ \mu_{\pm}:=2\tilde B_0(\pm\dir_0)\tilde v_0(\pm\dir_0).
\ee 
This leads to the following result
\begin{lemma}\label{cont-trans} If $f$ and $\CB$ are as in \eqref{f-def} -- \eqref{B-adtnl}, then the leading singularity of $\CB\hat f$ at $x_0$ is given by
\be\label{lim-res-final}
(\CB\hat f)(x_0+\e\check x)\sim\e^{-{\kappa_2}}\mu(\check x\cdot\dir_0)=\mu(\e\check x\cdot\dir_0),
\ee
where
\be\label{A-def}\begin{split}
\mu(t)&=
\frac{\Gamma({\kappa_2})}\pi\left[q_2 \mu_+ (t-i0)^{-{\kappa_2}}+\frac{\mu_-}{q_2} (t+i0)^{-{\kappa_2}}\right],\,
q_2=\exp\left(-i{\kappa_2}\frac{\pi}2\right),\, {\kappa_2} > 0,\\ 
\mu(t)&=\mu_+(-i)\text{sgn}(t),\ {\kappa_2} = 0.
\end{split}
\ee 
\end{lemma}
Note that \eqref{B-adtnl} and \eqref{mu-def-1} yield $\mu_-=-\mu_+$ if $\kappa_2=0$. Condition \eqref{B-adtnl} is not strictly necessary. We impose it for simplicity to avoid dealing with logarithmic terms in the leading singular behavior of $\CB\hat f$ \cite{ar01, GS, rz2, rz1}.

The distribution $\mu$ obtained in Lemma~\ref{cont-trans} is the CTB of $\CB\hat f$, i.e. of the reconstruction from continuous data.

The singular behavior of $\hat f(\al,p)$ near $p=H(\al)$, $\al\in\Omega\cup(-\Omega)$ is obtained analogously. Consider
\be\label{rt-lim}\begin{split}
(\hat f,\pa_p^m w_\e):=\int\hat f(\al,p)\pa_p^m w_\e(p-H(\al))dp=
\frac1{2\pi}\int (-i\la)^m\tilde w(-\e \la)\fs(\la\al)d\la,
\end{split}
\ee 
where $w\in\coi(\br)$. Following the lines of the proof of Lemma~\ref{lem:mu-der} (but without using the stationary phase lemma) leads to 
\be\label{lim-res-2}\begin{split}
&\lim_{\e\to0}\e^{m-s+1}(\hat f,\pa_p^m w_\e)
=\frac{1}{2\pi}\int (-i\la )^m\tilde w(-\la)\tilde\mu(\la)d\la,\ m>s-1,\\
&\tilde\mu(\la)=\begin{cases} \omega(\al)\tilde v_0(\al)\la_+^{-s}+\omega(-\al)\tilde v_0(-\al)\la_-^{-s},& s\not=2,3,\dots,\\
\omega(\al)\tilde v_0(\al)\la^{-s},& s=2,3,\dots,\\ 
&\text{and }\tilde v_0(-\al)=e^{-i(\kappa_1+1)\pi}\tilde v_0(\al).\end{cases}
\end{split}
\ee 
The condition on the last line in the equation for $\tilde\mu$ is analogous to the one in \eqref{f0-v2}.
When the second case in $\tilde\mu$ occurs, \eqref{omega-pm} implies $\omega(-\al)\tilde v_0(-\al)=(-1)^s\omega(\al)\tilde v_0(\al)$. The following result is now immediate.
\begin{lemma}\label{lem:fhsing} If $f$ is as in \eqref{f-def} -- \eqref{more-assns}, then the leading singularity of $\hat f(\al,p)$ at $p=H(\al)$, $\al\in\Omega\cup(-\Omega)$, is given by
\be\label{lim-res-3}\begin{split}
&\hat f(\al,H(\al)+\e\check p)\sim\e^{s-1}\hat f_0(\al,\check p)=\hat f_0(\al,\e\check p),\\
&\hat f_0(\al,\check p)=a_+(\al)\check p_+^{s-1}+a_-(\al)\check p_-^{s-1},\\
&a_\pm(\al)=\frac{1}{2\sin(\pi s)\Gamma(s)}\biggl(\omega(\al)\tilde v_0(\al)e^{\pm i(s-1)\frac{\pi}2}+\omega(-\al)\tilde v_0(-\al)e^{\mp i(s-1)\frac{\pi}2}\biggr),\\
%&\hat f_0(\al,\check p)=\frac{1}{2\sin(\pi s)\Gamma(s)}\biggl(\omega(\al)\tilde v_0(\al)e^{i(s-1)\frac{\pi}2}(\check p-i0)^{s-1}\\
%&\hspace{3cm}+\omega(-\al)\tilde v_0(-\al)e^{-i(s-1)\frac{\pi}2}(\check p+i0)^{s-1}\biggr),\
&\hspace{1cm}s\not=2,3,\dots,\\
&\hat f_0(\al,\check p)=\frac{\omega(\al)\tilde v_0(\al)}{2i^s(s-1)!}\check p^{s-1}\text{sgn}(\check p),\ s=2,3,\dots \text{ and }\tilde v_0(-\al)=e^{-i(\kappa_1+1)\pi}\tilde v_0(\al).
\end{split}
\ee 
\end{lemma}

The condition on $\tilde v_0(\pm\al)$ in \eqref{lim-res-2} and \eqref{lim-res-3} guarantees that $\hat f_0={\mathcal F}^{-1}\tilde\mu$ does not contain logarithms when $s$ is an integer. Clearly, $a_+(\al)=a_-(-\al)$, so $\hat f_0(\al,\check p)$ is even. The above formulas are precisely what one gets by (1) retaining only the leading term in \eqref{f-lim}, (2) computing the Radon transform of the resulting function (say, $f_1$) by using the Fourier slice theorem, and (3) using the results on the asymptotics of the Fourier transform at the origin (see \cite{wong}, Section VI.5).

\noindent
{\bf Example.} Consider 2D LT for a function with jump discontinuity, i.e. $\kappa_1=0$. For the purpose of normalization, multiply $f$ by a constant so that $v_+=i(2\pi)^{n-1}$. In this case $f_0(p)=\text{sgn}(p)/2$ has a unit jump. Thus, we have 
\be\label{params-LT}\begin{split}
&n=2,\ \bt=2,\ {\kappa_1}=0,\ s=3/2,\ {\kappa_2}=\bt-s-\frac{n-3}2=1,\\
&\ v_+=2\pi i,\ v_-=-v_+,\ \tilde B_0(\dir_0)=1/(4\pi),\
q_2=1/i,\ \mu_\pm=v_\pm/(4\pi).
\end{split}
\ee
Substitution into \eqref{A-def} gives that $\mu(p)=1/(\pi p)$, which coincides with the leading term in the formula (5.4.4) in \cite{rk}.

%\newpage
\section{Computing the DTB}\label{sec:conn}

\subsection{Main assumptions}\label{sec:main-ass}
In this section we use the results of Section~\ref{sec:gen} to compute the DTB of $\CB_\e g$, where $\CB$ is the same as in \eqref{B-oper}--\eqref{B-adtnl}, $\CB_\e$ is the discrete version of $\CB$, and $g\in\CE'(Z_n)$ is a distribution with similar singularities as $\hat f$ (cf. \eqref{lim-res-3}). More precisely, we assume the following.

\noindent{\bf Assumptions about $H$:}
\begin{enumerate}
\item $H(\al)\in C^\infty(\Omega\cup(-\Omega))$, $H$ is real-valued and odd: $H(\al)=-H(-\al)$;  
\item $H(\la\al)=\la H(\al)$, $\la>0$, $\al\in\Omega\cup(-\Omega)$;
\item  $H'(\xi_0)=x_0$, 
\item $\check H''(\al)$ is negative definite on $\Omega$;
\end{enumerate}
Recall that $H$ defines the surface $\s$, see \eqref{surf-def}.
%Associated with $H$, there is a surface 
%\be\label{surf-def}
%\s:=\{x\in\br^n:\, x=H'(\al),\al\in \Omega\}.
%\ee
%Recall that $H'(\al)$ is the derivative $H'(\xi)$ evaluated at $\xi=\al\in S^{n-1}$ (as opposed to the derivative on the unit sphere).

\noindent{\bf Assumptions about $g$:}
\begin{enumerate}
\item $g$ is smooth away from the surface $p=H(\al)$:
\be\label{g-smooth}
g \in C^\infty\left(\{(\al,p)\in Z_n:\, \al\in\Omega\cup(-\Omega),\ p\not= H(\al)\}\right);
\ee
\item $g$ is compactly supported:
\be\label{g-locsup}
g(\al,p)\equiv 0 \text{ if $|p|>c$ or $\al\not\in\Omega\cup(-\Omega)$};
\ee
\item $g$ is even: $g(\al,p)=g(-\al,-p)$;
\item $g$ can be written in the form
{
\be\label{gfn-exp}\begin{split}
&g(\al,H(\al)+p)= a_0(\al)p^{s_0-1}+O(p^{s_1-1}),\ p\to+0,\al\in \Omega\cup(-\Omega),\\ 
&a_0 \in \coi(\Omega\cup(-\Omega)),\
s_0<s_1,
\end{split}
\ee }
which is uniform in $\al$, and can be differentiated $L_\bt+1$ times (see \eqref{orders} below) with respect to $p$.
\end{enumerate}
To be precise, the assumption that \eqref{gfn-exp} and its derivatives are uniform means that for any $l=0,1,2,\dots, L_\bt+1$, there exists $A_{l}\in \coi(\Omega\cup(-\Omega))$ so that
\be\label{gfn-exp-unif}
\left|\pa_p^l\biggl(g(\al,H(\al)+p)-a_0(\al)p^{s_0-1}\biggr)\right|\le  A_{l}(\al)p^{s_1-1-l}, p\to+0,\al\in \Omega\cup(-\Omega).
\ee 
{
\begin{remark}
The Radon transform of $f$ defined in \eqref{f-def}--\eqref{more-assns} is a conormal distribution. The class of distributions described above is more general. For example, we do not require the differentiability of \eqref{gfn-exp} with respect to $\al$ (see the discussion in Section 18.2 leading up to the Definition 18.2.6 in \cite{hor3}). 
\end{remark}
} 

Since $g$ is even, similarly to \eqref{B-lim-pm}, we have
\be\label{gfn-exp-pm}
g(\al,H(\al)+p)= a_0(\al)p_+^{s_0-1}+a_0(-\al)p_-^{s_0-1}+O(|p|^{s_1-1}),\ p\to0,\al\in \Omega\cup(-\Omega).
\ee 

An additional assumption is 
\be\label{c-ratios-req}\begin{split}
&\frac{\tilde B_0(-\al)\left(a_0(\al)e^{-i\frac\pi2 s_0}+a_0(-\al)e^{i\frac\pi2 s_0}\right)}{\tilde B_0(\al)\left(a_0(\al)e^{i\frac\pi2 s_0}+a_0(-\al)e^{-i\frac\pi2 s_0}\right)}
=\begin{cases} e^{-i(\bt_0-s_0)\pi},& \al\in\Omega,\\
e^{i(\bt_0-s_0)\pi},& \al\in-\Omega,\end{cases}\text{ if }\kappa_2=0.
%\ \bt-s=\frac{n-3}2.
\end{split}
\ee
Here and in similar fractions below, we assume tacitly that either the denominator is not zero or both the denominator and numerator are zero. The meaning of \eqref{c-ratios-req} is discussed following equation \eqref{gfn-def} below. It is easy to see that if \eqref{c-ratios-req} holds for $\al\in\Omega$, then it automatically holds for $\al\in-\Omega$ as well.

Define
\be\label{orders}
L_\bt:=\begin{cases}\bt_0, & \bt_0 \in\mathbb N,\\
\lceil\bt_0\rceil, &\bt_0 \not\in\mathbb N.\end{cases}
\ee

\noindent {\bf Assumptions about the kernel $\ik$:}
\begin{itemize}
%\item[IK0$'$.] \red{$\ik$ is an interpolating kernel (cf. \eqref{int-ker-req});}
\item[IK1$'$.] $\ik$ is exact up to the degree $L_\bt$, i.e.
\be\label{ker-int-v2}
\sum_{j\in \mathbb Z} j^{m}\ik(t-j)=t^m,\quad 0\le m\le L_\bt,\ t\in\br;
\ee
\item[IK2$'$.] $\ik$ is compactly supported;
\item[IK3$'$.] One has 
\be\label{ik3-cond}
\ik^{(j)}\in L^\infty(\br),\ \begin{cases}0\le j\le \bt_0+1,& \text{if } \bt_0 \in\mathbb N,\\ 
0\le j\le \lceil\bt_0\rceil,& \text{if }\bt_0 \not\in\mathbb N;
\end{cases}
\ee
%\item[IK4$'$.] If the leading order term of $\CB$ is not local in $p$, then
%\be\label{sfc-ft}
%|\la|^{\bt_0}\tilde\ik^{(j)}(\la)=O(|\la|^{-c}),\ \la\to\infty,\ 0\le j\le L_\bt+1,
%\ee
%for some $c>1$; 
and
\item[IK4$'$.] $\ik$ is normalized, i.e. $\int_{\br}\ik(t)dt=1$.
\end{itemize}

%$\bt\in\mathbb N$: $\ik^{(\bt+1)}\in L^\infty$,  $\ik$ is exact up to degree $\bt$;
%
%$\bt\not\in\mathbb N$: $\ik^{(\lceil\bt\rceil)}\in L^\infty$, $\ik$ is exact up to degree ${\lceil\bt\rceil}$

The discrete data are given by 
\be\label{data_crt}
g(\al_{\vec k},p_j),\ p_j=\e j,\ j\in\mathbb Z,\ {\vec k}\in\mathbb Z^{n-1}.
\ee
%The interpolated in $p$ data $g_\e$ becomes
%\be\label{g-int-v2}
%g_\e(\al_{\vec k},p):=\sum_j g(\al_{\vec k},p_j) \ik\left(\frac{p-p_j}{\e}\right).
%\ee
Assume that there exist smooth diffeomorphisms $T_\pm(u):U\to \pm\Omega$ such that 
\be\label{alH}
\al_{\vec k}=T_\pm(\e({\vec k}+u_\e)) \text{ for any }\al_{\vec k}\in\pm\Omega,
\ee
where $U\subset\br^{n-1}$ is some domain, and $u_\e\in[0,1)^{n-1}$. The point $u_\e$ may depend on $\e$. Without loss of generality, we may suppose $0\in U$ and $T_\pm(0)=\pm\dir_0$. 
%Thus the determinant 
%\be\label{dHnz}
%\dt(T'(u))=\dt\left({\pa\al}/{\pa u}\right)\not=0
%\ee
%is bounded away from zero on $U$. 

The operator of reconstruction from discrete data (i.e., the discrete version of \eqref{B-oper}) is defined similarly to \eqref{fe-0}:
\be\label{CBe-def}\begin{split}
(\CB_\e g)(x):=&\sum_{\vec k} \sum_j\int B(\al_{\vec k},\al_{\vec k}\cdot x-p)\ik\left(\frac{p-p_j}{\e}\right)dp\, g(\al_{\vec k},p_j) |\Delta\al_{\vec k}|
%\\
%=&\sum_{\vec k} \sum_j(\CB \ik_\e)(\al_{\vec k},\al_{\vec k}\cdot x- p_j) g(\al_{\vec k},p_j)|\Delta\al_{\vec k}|,\ \ik_\e(t):=\ik(t/\e)
.
\end{split}
\ee
Here $\Delta\al_{\vec k}$ is the elementary domain on $S^{n-1}$ per each $\al_{\vec k}$, and $|\Delta\al_{\vec k}|$ is its volume. From \eqref{alH}, $|\Delta\al_{\vec k}|=\e^{n-1}|\dt(T_\pm'(\e(\vec k+u_\e)))|(1+O(\e))$.
%Its diameter satisfies $\text{diam}(\Delta\al_{\vec k})=O(\e)$, and its volume satisfies $|\Delta\al_{\vec k}|=({\kappa(\al)}+O(\e))\e^{n-1}$ whenever $|\al_{\vec k}-\al|=O(\e)$. We assume $\kappa(\al)\in C^{\infty}(S^{n-1})$ and $\kappa(\al)>0$ on $S^{n-1}$. The domains are pairwise disjoint, and their union is all of $\Omega\cup(-\Omega)$.   
%The general analogue of \eqref{fe-1-nosm}  is
%\be\label{fe-1-gen}
%(\CB_\e g_0)(x):=\sum_{\vec k} \sum_j(\CB \ik_\e)(\al_{\vec k},\al_{\vec k}\cdot x-\e j) g_0(\al_{\vec k},p_j)|\Delta\al_{\vec k}|,\ 
%\ik_\e(t):=\ik(t/\e).
%\ee

\begin{definition} Let $T_\pm(u):U\to \pm\Omega$ be the functions that specify the available directions (cf. \eqref{alH}), and $\pm\dir_0=T_\pm(0)$ be the unit vectors normal to $\s$ at $x_0$. A point $x_0\in\s$ is locally generic if each of the vectors $\left.\pa(T_\pm(u)\cdot x_0)/\pa u \right|_{u=0}$ has at least one irrational component.
\end{definition}

There is also the notion of a globally generic point, see \cite{kat19b}. Here we do not investigate  global aspects of reconstruction from discrete data, so the word `local' is omitted, and $x_0$ is called generic. 

For simplicity, in what follows we ignore the data corresponding to $\al_{\vec k}\in-\Omega$ (and drop the subscript `$\pm$' from $T_\pm$) using that $B$ and $g$ are even, and the analysis is the same in both cases $\al_{\vec k}\in\pm\Omega$. 

\subsection{Preliminary construction}\label{ssec:prelim} In view of \eqref{B-oper} -- \eqref{B-lim-pm}, \eqref{gfn-exp-pm}, and \eqref{data_crt} define the functions:
\be\label{three-fns}\begin{split}
\CBa\ik:=&{\mathcal F}^{-1}(\tilde b(\la)\tilde\ik(\la)),\ \tilde b(\la):=b_+\la_+^\bt+b_-\la_-^\bt;\ \CA(t):=a_+t_+^{s-1}+a_-t_-^{s-1};\\
\psi(t,p):=&\sum_{j}(\CBa\ik)(t-j) \CA(j-p),\ \Psi(t):=\int (\CBa\ik)(t-r) \CA(r)dr.
\end{split}
\ee
Thus, $b_\pm$ correspond to $\tilde B_j(\pm\al)$ in \eqref{B-lim-pm} for some $j\ge0$, and $a_\pm$ correspond to $a_0(\pm\al)$ in \eqref{gfn-exp-pm}.

{If $\CBa$ is a local operator (i.e., $\CBa=c\pa_p^\bt$), then $\bt\in 0\cup\mathbb N$ and $b_-=(-1)^\bt b_+$. If $\CBa$ is not local, but $\bt\in 0\cup\mathbb N$, then all this means is that $b_-\not=(-1)^\bt b_+$. Any such $\CBa$ can be written as a linear combination of the operators $\pa_p^\bt$ and $\CH\pa_p^\bt$, where $\CH$ is the Hilbert transform. In the latter case, $b_-=(-1)^{\bt+1} b_+$. Therefore, if $\bt\in 0\cup\mathbb N$, in what follows we will consider only the two cases: $\CBa=c\pa_p^\bt$ and $\CBa=c\CH\pa_p^\bt$.}

Recall that $\bt(=\bt_0)$ and $s(=s_0)$ satisfy \eqref{kappa-assns}. In particular, $s\ge 3/2$ and $\bt\ge1$. As is easy to see, $(\CBa\ik)(t)=O(|t|^{-(\bt+1)})$, $t\to\infty$. By assumption, $\bt-s+1\ge(n-1)/2>0$, and the series in \eqref{three-fns} converges absolutely. The function $\Psi(t-p)$ is the continuous analogue of $\psi(t,p)$, and the integral with respect to $r$ in \eqref{three-fns} is absolutely convergent because $\bt-s+1>0$. Also, easy computations show that
\be\label{psi-Psi}
\int_0^1\psi(t+r,r+p)dr=\Psi(t-p);\quad \psi(t,p)=\psi(t-m,p-m),\ m\in\mathbb Z;
\ee
and
\be\begin{split}\label{aux-fn-lim-pm}
\Psi(t)=&\CF^{-1}\left(\tilde\ik(\la)\left(c_+^{(1)}\la_+^{\bt-s}+c_-^{(1)}\la_-^{\bt-s}\right)\right),\\ 
c_\pm^{(1)}:=&\Gamma(s)b_\pm\left(a_+e^{\pm i(\pi/2)s}+a_-e^{\mp i(\pi/2)s}\right).
\end{split}
\ee
If $a_\pm$ are such that $\CA(p)$ coincides with $\hat f_0(\al,p)$ in \eqref{lim-res-3} for some $\al$, then
\be\label{c1-pm}
{c_\pm^{(1)}}=\tilde B_0(\pm\al) \omega(\pm \al)\tilde v_0(\pm\al).
\ee

\subsection{Computation of the leading term of the DTB}\label{ssec:clt} 
In this section we consider only the leading terms of $\CB$ and $g$. More precisely, we assume that (cf. \eqref{B-lim}):
\be\label{Bop-dfn}\begin{split}
&\tilde B(\al,\la):=\la^{\bt}\tilde B_0(\al),\ \la>0,\ \tilde B_0(\al)\in C^\infty(S^{n-1}),\\ 
&\tilde B(-\al,-\la)=\tilde B(\al,\la),\ \la\in\br,\ 
\al\in S^{n-1},
\end{split}
\ee
and replace $g(\al,p)$ with its leading term (cf. Section~\ref{sec:main-ass}):
\be\label{gfn-def}
g_0(\al,H(\al)+p):=a_0(\al)p_+^{s-1}+a_0(-\al)p_-^{s-1},\ 
a_0 \in \coi(\Omega\cup(-\Omega)),
\ee 
where $H(\al)$ is the same as in Section~\ref{sec:main-ass}. The fact that $\tilde B(\al,\la)$ is not smooth at $\la=0$ is irrelevant. We use the notation $g_0$ instead of $g$, because $g_0$ does not satisfy one of the requirements in Section~\ref{sec:main-ass}: $g_0$ is not compactly supported.

Comparing \eqref{B-lim}, \eqref{Bop-dfn}, and \eqref{gfn-def} with \eqref{three-fns}, we see that $\CBa$ represents the highest order term of $\CB$, which acts with respect to the affine variable for any fixed $\al$. Likewise, $\CA$ is the leading singular term of $g$ for any fixed $\al$.

Now we can discuss the meaning of condition \eqref{c-ratios-req}. By \eqref{aux-fn-lim-pm}, the ratio in \eqref{c-ratios-req} equals to the ratio $c_-^{(1)}(\al)/c_+^{(1)}(\al)$. Condition \eqref{c-ratios-req} holds, in particular, if $g_0=\hat f_0$ and, therefore, $c_{\pm}^{(1)}(\al)$ are the same as in \eqref{c1-pm}. Indeed, suppose that $\al\in\Omega$. By \eqref{omega-pm} and \eqref{B-adtnl},
\be\label{c-ratios}
\frac{c_-^{(1)}(\al)}{c_+^{(1)}(\al)}
=\frac{\tilde B_0(-\al) \tilde v_0(-\al)}{\tilde B_0(\al) \tilde v_0(\al)}e^{-i(n-1)\frac\pi2}=-e^{-i(n-1)\frac\pi2}=e^{-i(\bt-s)\pi}.
\ee
The case when $\al\in -\Omega$ can be considered similarly. Thus, we can view \eqref{c-ratios-req} as a generalization of \eqref{B-adtnl} to the case when $g_0(\al,p)$ is the leading singular term of a function that is not necessarily in the range of the Radon transform.

Similarly to \eqref{xe-or-def}, set 
\be\label{xeh}
\xe:=x_0+\e\check x,\ h:=\check x\cdot\dir_0. 
\ee
We assume throughout that $\check x$ is confined to a bounded set. Using \eqref{CBe-def} (with $g=g_0$) and \eqref{three-fns}, we obtain similarly to \eqref{fe-psi-nosm}:
\be\label{newlt-1}\begin{split}
 g_\e^{(1)}(\check x):=\e^{{\kappa_2}}&(\CB_\e g_0)(\xe)
=\sum_{{\vec k}}\psi_{\al_{\vec k}}\left(\al_{\vec k}\cdot\check x+\frac{\al_{\vec k}\cdot x_0}{\e},\frac{H(\al_{\vec k})}\e\right) \frac{|\Delta\al_{\vec k}|}{\e^{(n-1)/2}}.
\end{split}
\ee
If $g_0=\hat f_0$, then $\psi_\al$ used in \eqref{newlt-1} is defined using \eqref{three-fns}, where $b_\pm=\tilde B_0(\pm\al)$, and $a_\pm$ are given in \eqref{lim-res-3}.
Since $B(\al,p)$ and $g_0(\al,p)$ are even, the sum in \eqref{newlt-1} can be confined to $\al_{\vec k}\in\Omega$, and a prefactor 2 appears. Strictly speaking, the set of all directions $\al_k$ is not necessarily symmetric. However, our main results are asymptotic as $\e\to0$, and we obtain the same limits in both cases: when $\al_k\in\Omega$ and $\al_k\in-\Omega$.

Analogously to \eqref{two-sets}, introduce
\be\label{two-sets-rn}
\Omega_a:=\{\al\in \Omega:\,|\al^\perp|\le A\e^{1/2}\},\ \Omega_b:=\{\al\in\Omega:\,|\al^\perp|> A\e^{1/2}\},
\ee
where $\al^\perp$ is the projection of $\al$ onto the plane $\dir_0^\perp$. Define also the functions $g_\e^{(1a)}(\check x)$, $g_\e^{(1b)}(\check x)$ by restricting the summation in \eqref{newlt-1} to $\al_{\vec k}\in\Omega_a$ and $\al_{\vec k}\in\Omega_b$, respectively.

Similarly to \eqref{another-psi-pr}, we have the following result, which is proven in Appendix~\ref{psi-incr-prf}.

\begin{lemma}\label{psi-incr} Pick any bounded set $V\subset\br$. One has
\be\label{an-psi-pr-1}
\begin{split}
&\psi_\al(t+\e,p)-\psi_\al(t,p)=
\begin{cases} 
O(\e),& \text{$\CBa$ is local},\\
O(\e^{1-\{\bt\}}),& \text{$\CBa$ is not local, $\bt\not\in\mathbb N$},\\
O(\e\ln(1/\e)),& \text{$\CBa$ is not local, $\bt\in\mathbb N$},
\end{cases}\\
&\psi_\al(t,p+\e)-\psi_\al(t,p)=O(\e^{\min(s-1,1)}),\ \al\in\Omega,\ t-p\in V,
\end{split}
\ee
where all the big-$O$ terms are uniform in $\al$, $t$, and $p$ confined to the indicated sets. 
\end{lemma}

%The function $p(\al^\perp):=\al\cdot x_0-H(\al)$ can be viewed as the restriction of the function $\xi\cdot x_0-H(\xi)$ from $\br^n\setminus 0$ to $\Omega\subset S^{n-1}$,  
Using that $H'(\xi)=x_0$, it is easy to show that 
\be\label{scnd-der-shifted}
\left.\frac{\pa^2 (H(\al)-\al\cdot x_0)}{(\pa\al^\perp)^2}\right|_{\al^\perp=0}=\check H''(\Theta_0),
\ee
where the directions $\al\in\Omega$ are parametrized by $\al^\perp$. Expanding the function $\al\cdot x_0:\,S^{n-1}\to\br$ in the Taylor series around $\al=\dir_0$ (i.e. $\al^\perp=0$) and using that $H(\al)-\al\cdot x_0$ is quadratic in $\al^\perp$, we find using \eqref{an-psi-pr-1} and \eqref{scnd-der-shifted} (cf. \eqref{g1a-st1}):
\be\label{new-g1a-st1}\begin{split}
&g_\e^{(1a)}(\check x)\\
&=2\sum_{\al_{\vec k}\in\Omega_a} \psi_{\dir_0}\biggl(h+r_{\vec k}+\frac1\e P(\al_{\vec k}^\perp)+O(\e^{1/2}), r_{\vec k}+\frac1\e P(\al_{\vec k}^\perp)\\
&\qquad\qquad+\frac{\check H''(\Theta_0)\al_{\vec k}^\perp\cdot \al_{\vec k}^\perp}{2\e}+O(\e^{1/2})\biggr)(1+O(\e^{1/2}))\frac{|\Delta\al_{\vec k}|}{\e^{(n-1)/2}}\\
&=2\sum_{\al_{\vec k}\in\Omega_a} \psi_{\dir_0}\left(h+r_{\vec k}+\frac1\e P(\al_{\vec k}^\perp), r_{\vec k}+\frac1\e P(\al_{\vec k}^\perp)+\frac{\check H''(\Theta_0)\al_{\vec k}^\perp\cdot \al_{\vec k}^\perp}{2\e}\right)\frac{|\Delta\al_{\vec k}|}{\e^{(n-1)/2}}\\ 
&\qquad+o(1),\ r_{\vec k}:=\left(\frac{\dir_0\cdot x_0}{\e}+{\vec a}\cdot u_\e\right)+{\vec a}\cdot{\vec k}.
\end{split}
\ee
Here $P$ is some homogenous polynomial of degree 2, $\vec a=(T'(0))^t x_0$, and $u_\e$ is defined in \eqref{alH}. The factor $1+O(\e^{1/2})$ on the third line in \eqref{new-g1a-st1} appears, because we replace $\psi_{\al_k}$ with $\psi_{\dir_0}$. This means that we set $\al=\dir_0$ in \eqref{Bop-dfn}, \eqref{gfn-def} when computing $b_\pm$, $a_\pm$ in the definition of $\psi$ (cf. \eqref{three-fns}). A more accurate estimate than $o(1)$ can be obtained in \eqref{new-g1a-st1} using \eqref{an-psi-pr-1} and that the sum is uniformly bounded as $\e\to0$ (recall that $A>0$ in the definitions of $\Omega_{a,b}$ is fixed when we consider the limit as $\e\to0$). However, a more precise estimate is not necessary for our purposes. The assumption that $x_0$ is generic implies that the sequence $r_{\vec k}$ is uniformly distributed mod 1. This is easy to see by extending the arguments in Theorem 2.9 and Example 2.9 of \cite{KN_06} from double sequences to $(n-1)$-dimensional sequences. Arguing similarly to \cite{kat_2017, kat19a, kat19b} gives 
\be\label{ga1-lim-gen}\begin{split}
\lim_{\e\to0}g_\e^{(1a)}(\check x)
&= 2\int_{|u|\le A}\int_0^1 \psi_{\dir_0}\left(h+r+P(u), r+P(u)+\frac{\check H''(\Theta_0)u\cdot u}{2}\right)dr du,
\end{split}
\ee
where the integral with respect to $u$ is over a disk in the hyperplane $\Theta_0^\perp$. 

\subsection{Estimation of $g_\e^{(1b)}(\check x)$}\label{ssec:estg1b}

Define
\be\label{remainder}
R_\Psi(t)=\begin{cases} |t|^{s-2-\bt},& \text{if $\CBa$ is local or $\CBa$ is not local and }\bt-s\in \mathbb Z,\\
|t|^{-\lceil\bt-s+1\rceil},& \text{if $\CBa$ is not local and }\bt-s\not\in \mathbb Z. \end{cases}
\ee
We need the following two lemmas, which are proven in Appendices~\ref{sec:contverpsi} and \ref{sec:psi-asympt}, respectively.

\begin{lemma}\label{contverpsi} 
If ${\CBa}$ is local, i.e. $\bt\in \mathbb N$, then
\be\label{B-loc-res}
\Psi(t)=b_+(i\pa_t)^\bt\CA(t)+O(R_\Psi(t)),\ t\to\infty.
\ee
If ${\CB}$ is not local, then
\be\label{Psi-ft-mult-res}
\Psi(t)=\CF^{-1}\left(c_+^{(1)}\la_+^{\bt-s}+c_-^{(1)}\la_-^{\bt-s}\right)
+O(R_\Psi(t)),\ t\to\infty.
\ee
Moreover,  
\be\label{Psi-ass-loc}
\Psi(t)=O(R_\Psi(t))\text{ as } \begin{cases}
t\to-\infty\text{ if}& c_-^{(1)}=c_+^{(1)}e^{i(\bt-s)\pi},\,\bt-s\not\in \mathbb Z,\\
t\to+\infty\text{ if}& c_-^{(1)}=c_+^{(1)}e^{-i(\bt-s)\pi},\,\bt-s\not\in \mathbb Z,\\
t\to\pm\infty\text{ if}& c_-^{(1)}=(-1)^{\bt-s}c_+^{(1)},\,\bt-s\in \mathbb Z.
\end{cases}
\ee
\end{lemma}

\begin{lemma}\label{psi-asympt} One has
\be\label{psi-ass-full}
\psi(t,p)-\Psi(t-p)=O\left(R_\Delta(t-p)\right),\  t-p\to\infty,
\ee
where 
\be\label{remainder-v2}
R_\Delta(t)=|t|^{s-2-\bt}\times\begin{cases} 1,& \CBa\text{ is local or $\CBa$ is not local and }\bt\not\in \mathbb N,\\
\log|t|,& \CBa\text{ is not local},\bt\in\mathbb N. 
\end{cases}
\ee
\end{lemma}

Similarly to \eqref{psi-bnd} and \eqref{ge1b-bnd}, using Lemmas~\ref{contverpsi} and \ref{psi-asympt} and the fact that $|\al\cdot x_0-H(\al)|\ge c_1|\al^\perp|^2$, $\al\in\Omega_b$, for some $c_1>0$, we get
\be\label{newbnds}\begin{split}
&\left|\psi_{\al}\left(\al\cdot\check x+\frac{\al\cdot x_0}{\e},\frac{H(\al)}\e\right)\right|\le
c_2\left[\frac{|\al^\perp|^2}{\e}\right]^{-(\bt-s+1)},\ \al\in\Omega_b,
\\ 
&|g_\e^{(1b)}(\check x)|\le O(\e^{\frac{n-1}2})\sum_{O(A/\e^{1/2})\le |\vec k|\le O(1/\e)}\left[\frac{(\e |\vec k|)^2}\e\right]^{-(\bt-s+1)}=O\left(A^{-2\kappa_2}\right),
\end{split}
\ee
where the very last big-$O$ term is independent of $\e$. As before, the reason why we can select a single $c_2>0$ for all $\al\in\Omega_b$ in the first line of \eqref{newbnds} follows from the fact that $\tilde B_0(\al)\in C^\infty(\Omega)$ and $a_0(\al)\in C_0^\infty(\Omega)$ (cf. \eqref{Bop-dfn} and \eqref{gfn-def}). Also, we used here that $\check x$ is confined to a bounded set.
Therefore, 
\be\label{limsupg1b}
\lim_{A\to\infty}\limsup_{\e\to0}|g_\e^{(1b)}(\check x)|=0
\ee
provided that $\kappa_2>0$. The boundary case $\kappa_2=0$ is considered separately (see Section~\ref{ssec:trans-v2} below). Since we never proved that $\lim_{\e\to0}g_\e^{(1b)}(\check x)$ exists, we instead used `$\limsup$' in \eqref{limsupg1b}.

\subsection{DTB in the case $\kappa_2>0$.}\label{ssec:trans}
Combining \eqref{ga1-lim-gen}, \eqref{limsupg1b}, and \eqref{newlt-1} and using that $A$ can be arbitrarily large, we get
\be\label{ga1-lim-gen-st2}\begin{split}
\lim_{\e\to0}\e^{{\kappa_2}}(\CB_\e g_0)(\xe)
&= 2\int_{\br^{n-1}}\int_0^1 \psi_{\dir_0}\left(h+r, r+\frac{\check H''(\Theta_0)u\cdot u}{2}\right)dr du\\
&=\frac{2^{(n+1)/2}|S^{n-2}|}{\left|\text{det}\check H''(\Theta_0)\right|^{1/2}}\ioi\int_0^1 \psi_{\dir_0}\left(h+r, r-t^2\right)dr t^{n-2}dt.
\end{split}
\ee
We used here that $\check H''(\dir_0)$ is negative definite and that $P(u)$ drops out from both arguments due to \eqref{psi-Psi}. In view of \eqref{three-fns}, \eqref{psi-Psi}, and \eqref{ga1-lim-gen-st2} (compare with \eqref{er-hr}), we introduce
\be\label{uzf-int}
J:=\frac12\int \Psi(h+t) t_+^{(n-3)/2}dt.
\ee
Essentially, $J$ is a convolution of three distributions. 
The fact that the integral with respect to $t$ is absolutely convergent follows from the assumption $\kappa_2>0$ and Lemma~\ref{contverpsi}. 
Hence, \eqref{three-fns} and \eqref{aux-fn-lim-pm} yield
\be\label{J-ft}\begin{split}
J=&\frac{\Gamma((n-1)/2)}{4\pi}\int \tilde\ik(\la)(c_+^{(1)}\la_+^\bt+c_-^{(1)}\la_-^\bt)
 e^{-i\frac{n-1}2\frac{\pi}2}(\la-i0)^{-\frac{n-1}2}e^{-i\la h}d\la\\
=&\frac1{2\pi}\frac{\Gamma((n-1)/2)}{2}\int \tilde\ik(\la)(c_+^{(2)}\la_+^{{\kappa_2}-1}+c_-^{(2)}\la_-^{{\kappa_2}-1})e^{-i\la h}d\la,\\
c_\pm^{(2)}:=&e^{\mp i\frac{n-1}2\frac{\pi}2}c_\pm^{(1)}.
%,\
%b_\pm:=\tilde B_0(\pm\dir_0),\ a_\pm:=a_0(\pm\dir_0).
\end{split}
\ee
By assumption, ${\kappa_2}> 0$, so the above multiplication of distributions is well-defined and leads to a function $c_+^{(2)}\la_+^{{\kappa_2}-1}+c_-^{(2)}\la_-^{{\kappa_2}-1}\in L^1_{\text{loc}}(\br)$. Combining the prefactors in \eqref{ga1-lim-gen-st2}, \eqref{J-ft} and applying the inverse Fourier transform (see also Appendix~\ref{sec:usef}) gives the distribution $\mu$ such that the right-hand side of \eqref{ga1-lim-gen-st2} equals $\ik*\mu$ (cf. \eqref{fe-lim-simple}, \eqref{mu-in-thrm} below). 

Let us simplify \eqref{J-ft} in the case $g_0=\hat f_0$. From \eqref{c1-pm},
\be\label{cpm}\begin{split}
c_\pm^{(2)}=e^{\mp i\frac{n-1}2\frac{\pi}2}\tilde B_0(\pm\dir_0) \omega(\pm \dir_0)
\tilde v_0(\pm\dir_0).
\end{split}
\ee
%where $\al=\dir_0$. 
Consequently,
\be\label{cpm-all}\begin{split}
\frac{\Gamma((n-1)/2)c_\pm^{(2)}}2
%=&\tilde B_0(\pm\dir_0) \Gamma((n-1)/2)e^{\mp i\frac{n-1}2\frac\pi2}\omega(\pm \dir_0)\tilde v_0(\pm\dir_0)/2
%%\end{split}
%%\ee
%%\end{document}
%\\
=&\tilde B_0(\pm\dir_0) \Gamma((n-1)/2)\frac{|\det \check H''(\dir_0)|^{1/2}}
{2(2\pi)^{(n-1)/2}}\tilde v_0(\pm\dir_0).
\end{split}
\ee
Multiply \eqref{cpm-all} with the prefactor on the right in \eqref{ga1-lim-gen-st2} to obtain
\be\label{pref-cpm}
\tilde B_0(\pm\dir_0) 2^{(n+1)/2}|S^{n-2}|\Gamma((n-1)/2)\frac{\tilde v_0(\pm\dir_0)}{2(2\pi)^{(n-1)/2}}=2\tilde B_0(\pm\dir_0)\tilde v_0(\pm\dir_0),
\ee
which leads to the same distribution as in \eqref{mu-def-1}, \eqref{lim-res-final}.

Thus, we have proven the following result.
\begin{theorem}\label{lead-thrm-pos} Suppose $\CB$ is given by \eqref{B-oper}, where 
\begin{enumerate}
\item $\tilde B(\al,\la)\equiv\la^{\bt}\tilde B_0(\al)$, $\la>0$, $\al \in S^{n-1}$, 
\item $\tilde B(\al,\la)$ is even, and 
\item $\tilde B_0(\al)\in C^\infty(S^{n-1})$. 
\end{enumerate}
Suppose also that $g_0(\al,H(\al)+p)=a_0(\al)p_+^{s-1}+a_0(-\al)p_-^{s-1}$, where $s\ge(n+1)/2$, and
\begin{enumerate}
\item $a_0 \in \coi(\Omega\cup(-\Omega))$,  
\item $H(\xi)\in C^\infty(\br^n\setminus \{0\})$ is real-valued, homogeneous of degree 1, and
\item $\check H''(\al)$ is negative definite on $\Omega$, $\dir_0\in\Omega$.
\end{enumerate}   
If ${\kappa_2}=\bt-s-\frac{n-3}2 > 0$, and $x_0$ is generic, then
\be\label{fe-lim-simple}
\lim_{\e\to0}\e^{{\kappa_2}}(\CB_\e g_0)(\xe)=(\ik*\mu)(h),
\ee
where
\be\label{mu-in-thrm}
\mu(t)=\frac{(2\pi)^{(n-1)/2}}{\left|\text{det}\check H''(\Theta_0)\right|^{1/2}}
\frac{\Gamma(\kappa_2)}{\pi}\left(c_+^{(2)}e^{-i\kappa_2\frac\pi2}(t-i0)^{-\kappa_2}+c_-^{(2)}e^{i\kappa_2\frac\pi2}(t+i0)^{-\kappa_2}\right),
\ee
and $c_\pm^{(2)}$ are defined in \eqref{J-ft}. If $g_0(\al,p)=\hat f_0(\al,p)$, which is defined in \eqref{lim-res-3}, then $\mu$ is given by \eqref{mu-def-1}, \eqref{A-def}. 
\end{theorem}

\noindent
{\bf Example.} Return now to the 2D LT example at the end of Section~\ref{sec:gen}. Substituting $\mu(p)=1/(\pi p)$ into \eqref{fe-lim-simple}, which was computed following \eqref{params-LT}, we recover the formula \eqref{final-lim-final} with $f_+=1$:
\be\label{fe-lim-simple-LT}\begin{split}
&\lim_{\e\to0}\e(\CB_\e\hat f)(\xe)=\lim_{\e\to0}\e \lte(\xe)=\int_{\br}\frac{\ik(h-r)}{\pi r}dr.
\end{split}
\ee

\subsection{DTB in the case $\kappa_2=0$.}\label{ssec:trans-v2} From \eqref{c-ratios-req}, \eqref{aux-fn-lim-pm}, and Lemmas~\ref{contverpsi}, \ref{psi-asympt} it follows that $\psi(t,p)=O((t-p)^{s-1-b-c})$ as $t-p\to+\infty$ for some $c,0<c\le 1$. This is the correct limit to consider (i.e., not $t-p\to-\infty$), because $\check H''(\Theta_0)$ is negative definite. Computing similarly to \eqref{newbnds} gives
\be\label{newbnds-v2}\begin{split}
&\left|\psi_{\al}\left(\al\cdot\check x+\frac{\al\cdot x_0}{\e},\frac{H(\al)}\e\right)\right|\le
c\left[\frac{|\al^\perp|^2}{\e}\right]^{s-1-\bt-c},\ \al\in\Omega_b,
\\ 
&|g_\e^{(1b)}(\check x)|\le O(\e^{\frac{n-1}2})\sum_{O(A/\e^{1/2})\le |\vec k|\le O(1/\e)}\left[\frac{(\e |\vec k|)^2}\e\right]^{s-1-\bt-c}=O\left(A^{-2c}\right).
\end{split}
\ee
Hence $\lim_{A\to\infty}\limsup_{\e\to0}|g_\e^{(1b)}(\check x)|=0$, and we obtain the same integrals as in \eqref{ga1-lim-gen-st2}, \eqref{uzf-int}.

Applying \eqref{c-ratios-req}, \eqref{aux-fn-lim-pm}, and Lemma~\ref{contverpsi} one more time we conclude that $\Psi(t)$ decays sufficiently fast as $t\to+\infty$, and the integral in \eqref{uzf-int} is absolutely convergent. Hence the above derivation holds in the case $\kappa_2=0$ as well. The only modification is that now the analogue of \eqref{J-ft} becomes
\be\label{J-ft-v2}\begin{split}
J=&\frac1{2\pi}\frac{\Gamma((n-1)/2)}{2}c_+^{(3)}(-i)\int \tilde\ik(\la)(\la-i0)^{-1}e^{-i\la h}d\la,\ c_\pm^{(3)}:= e^{\mp i\frac{n-3}2\frac{\pi}2}c_\pm^{(1)}.
\end{split}
\ee
When deriving \eqref{J-ft-v2}, we used that $c_+^{(3)}=c_-^{(3)}$, which follows from \eqref{c-ratios-req} and \eqref{aux-fn-lim-pm}. Now, $\mu(p)$ can be found by applying the inverse Fourier transform. 

To prove the assertion in more detail, suppose first that $\bt-s\not\in\mathbb Z$.
By \eqref{c-ratios-req} and \eqref{aux-fn-lim-pm} (cf. \eqref{c-ratios}), 
\be\label{G-ft}\begin{split}
\Psi(h)=&\frac{c_+^{(1)}}{2\pi}\int \tilde\ik(\la)(\la-i0)^{\bt-s}e^{-i\la h}d\la\\
=&\frac{i^k c_+^{(1)}}{2\pi}\int [(-i\la)^k\tilde\ik(\la)](\la-i0)^{\nu-1}e^{-i\la h}d\la\\
=&\frac{e^{i(\bt-s)\frac{\pi}2} c_+^{(1)}}{\Gamma(1-\nu)}\int (h-p)_-^{-\nu}\ik^{(k)}(p)dp,
\end{split}
\ee
where $k=\lceil \bt-s\rceil$, $\nu=\{\bt-s\}$. Replacing $h$ with $h+t$, substituting \eqref{G-ft} into \eqref{uzf-int}, changing the order of the $p$ and $t$ integrations (the double integral is absolutely convergent, because the domain of integration is a bounded set), and integrating with respect to $t$, we get \eqref{J-ft-v2} written as a convolution. This argument is similar to the one in \eqref{er-hr-byparts}.

If $\bt-s\in 0\cup\mathbb N$, then \eqref{c-ratios-req}, \eqref{aux-fn-lim-pm}, and \eqref{Psi-ft-mult-res}, imply that $\Psi=c\ik^{(\bt-s)}$ for some $c$. Upon integrating by parts, we again get \eqref{J-ft-v2} written as a convolution. 

If $g(\al,p)=\hat f_0(\al,t)$, we find
\be\label{new-mup}
\mu(p)=2\mu_+{\mathcal F}^{-1}((\la-i0)^{-1})=2\tilde B_0(\dir_0)\tilde v_0(\dir_0)i p_-^0.
\ee
This proves the following result.
\begin{theorem}\label{lead-thrm-zero} Suppose $\CB$ and $g_0$ are the same as in Theorem~\ref{lead-thrm-pos}, and $x_0$ is generic. Suppose ${\kappa_2}=0$ and condition \eqref{c-ratios-req} is satisfied. One has
\be\label{fe-lim-simple-v2}
\lim_{\e\to0}(\CB_\e g_0)(\xe)=(\ik*\mu)(h),
\ee
where
\be\label{mu-in-thrm-k20}
\mu(t)=\frac{2(2\pi)^{(n-1)/2}}{\left|\text{det}\check H''(\Theta_0)\right|^{1/2}}
c_+^{(3)}t_-^0,
\ee
and $c_+^{(3)}$ is given in \eqref{J-ft-v2}. In particular, if $g_0(\al,p)=\hat f_0(\al,p)$, which is defined in \eqref{lim-res-3}, then $\mu$ is given by \eqref{new-mup}. 
\end{theorem}

The difference between \eqref{A-def} and \eqref{new-mup} is that the result in \eqref{A-def} is non-unique (i.e., defined up to a constant if $\kappa_2=0$), and the result in \eqref{new-mup} is unique. There is no contradiction between the two formulas, because they do match up to a constant.

\noindent
{\bf Example.} In the case of 3D exact reconstruction for a function with one-sided jump discontinuity, we have 
\be\label{params-exact}\begin{split}
&n=3,\ \bt=2,\ {\kappa_1}=0,\ s=2,\ {\kappa_2}=\bt-s-\frac{n-3}2=0,\\
&v_+=i(2\pi)^2,\ \tilde B_0(\dir_0)=1/(8\pi^2).
\end{split}
\ee
Substituting \eqref{params-exact} into \eqref{new-mup}, \eqref{fe-lim-simple-v2}, we recover the formula (2.15) of \cite{kat19a}:
\be\label{exact-simple}\begin{split}
&\lim_{\e\to0}(\CB_\e\hat f)(\xe)=-\int_h^\infty\ik(t)dt.
\end{split}
\ee
%\newpage

\section{Lower order terms}\label{sec:lots}

In the previous section we computed the DTB of $\CB_\e g$ by retaining the leading order terms (corresponding to $\bt_0$ in $\CB$, and to $s_0$ - in $g$, cf. \eqref{B-lim} and \eqref{gfn-exp}, respectively). The goal of this section is to prove that lower order terms do not contribute to the DTB. Let $\CB_j$ denote the operator, which is obtained by retaining only the $j$-th term of the expansion in \eqref{B-lim}. Select a smooth function $\chi$ that satisfies $\chi(p)\equiv 0$, $|p|\le c$, and $\chi(p)\equiv 1$, $|p|\ge 2c$, for some $c>0$ sufficiently large. Then
\be\label{Bg-split}
\CB g=\CB_0 g_0+\biggl[\sum_{j\ge 1,\bt_j\ge0}\CB_j g+\CB_0(g-(1-\chi) g_0)\biggr]-\CB_0(\chi g_0)+\bigl(\CB-\sum_{j\ge 1,\bt_j\ge0}\CB_j\bigr) g.
\ee
In the previous section we computed the DTB corresponding to the first term on the right in \eqref{Bg-split}. Here we prove that all the other terms do not contribute to the DTB. The first result of this section is that the terms in brackets do not contribute to the DTB.
\begin{lemma}\label{lemma:lot} Let $x_0\in\s$ be generic. Suppose $\CB$ and $g$ are the same as in \eqref{B-oper}-\eqref{B-adtnl} and \eqref{g-smooth}--\eqref{gfn-exp-unif}, respectively, $\kappa_1,\kappa_2\ge0$, and all the assumptions in Section~\ref{sec:main-ass} hold. Suppose, additionally, that 
\begin{enumerate}
\item either $\CB$ has a homogeneous symbol given by the $j$-th term in the expansion \eqref{B-lim}, where $j\ge 1$ and  $\bt_j\ge0$, 
\item or $\CB=\CB_0$ (i.e. $j=0$), and the leading term in \eqref{gfn-exp} identically equals zero (i.e., $a_0(\al)\equiv0$).
\end{enumerate}
Then
\be\label{finally-kappa-all}\begin{split}
&\lim_{\e\to 0}\e^{{\kappa_2}}({\CB_\e} g)(\xe)=0,\ \kappa_2>0,\\
&\lim_{\e\to0}\left[({\CB_\e} g)(\xe)-({\CB_\e} g)(x_0)\right]=0,\ \kappa_2=0.
\end{split}
\ee
\end{lemma}

\begin{proof}[Proof of Lemma~\ref{lemma:lot}.]
Let $\bt'=\bt_j$ denote the order of the term that makes up $\CB$, and $s'$ denote the order of the first term in \eqref{gfn-exp} that is not identically zero (i.e., $s'=s_0$ or $s_1$). By assumption $\bt'-s'<\bt_0-s_0$. 
%By assumption, $\CB$ identically equals to the $j$-th term in the asymptotic expansion in \eqref{B-lim}, where $0\le \bt_j<\bt_0$. 
By \eqref{g-smooth}--\eqref{gfn-exp-unif},
\be\label{Oa_st1-g}\begin{split}
&|g(\al,H(\al)+p)|\le c_0|p|^{{{s'}}-1},\ (\al,p)\in\Omega\times\br;\\
&g(\al,p)\equiv 0,\  \al\not\in\Omega\cup(-\Omega)\text{ or }|p|\ge c;\\
&|\pa_p^l g(\al,H(\al)+p)|\le c_l |p|^{{{s'}}-1-l},\ 0<l\le L_\bt+1,\ p\not=0,\al\in\Omega,
\end{split}
\ee
for some $c_l>0$. The above inequalities hold even if ${s'}$ is an integer.
%Set 
%\be\label{KbKs}
%L_\bt:=\begin{cases}\lfloor \bt\rfloor+1,& \text{ if } \bt\not\in\mathbb N\\
%\bt, & \text{ if } \bt\in\mathbb N. \end{cases} 
%\ee
By Lemma~\ref{contverpsi}, rescaling the affine variable $t\to t/\e$ gives
%From \eqref{B-lim} (see e.g. \cite{wong}, Sections IV.2 and VI.5)
\be\label{Oa_st1-bk}
|{\CBa}\ik_\e(\al,t)|\le |\tilde B_j(\al)|\frac{O(\e^{-{{\bt'}}})}{1+|t/\e|^{{{\bt'}}+1}},\ \al\in\Omega,\ \ik_\e(t):=\ik(t/\e).
\ee

Set
\be\label{G-def}
G_{\xc}(\al):=\sum_j({\CBa} \ik_\e)(\al\cdot \xe-\e j) g(\al,\e j),\ \vartheta:=s'-1-\bt'.
\ee
In what follows, $\xc$ in the definition of $\xe$ is fixed and is omitted from notations. First we estimate $G(\al)$, $\al\in\Omega_a$. By \eqref{Oa_st1-bk} and \eqref{Oa_st1-g} with $l=0$,
\be\label{Oa_st2}\begin{split}
G(\al) & =O(\e^\vartheta)
 \sum_{|j|\le O(1/\e)} \frac{|j-\frac{H(\al)}\e|^{{{s'}}-1}}{1+|\frac{\al\cdot \xe}\e-j|^{{{\bt'}}+1}}\\
& = O(\e^\vartheta)
 \sum_{|j|\le O(1/\e)} \frac{|j+(p/\e)|^{{{s'}}-1}}{1+|j|^{{{\bt'}}+1}},\ 
 p=p(\al):=\al\cdot \xe-H(\al).
\end{split}
\ee
The condition $\al\in\Omega_a$ implies that $|p(\al)|/\e$ is bounded, so
\be\label{Ja_st3}
G(\al)= O(1)+O(\e^\vartheta),\ \al\in\Omega_a.
\ee
To estimate the sum $\sum_{\al_k\in\Omega_a} G(\al_k)|\Delta\al_{\vec k}|$, integrate over the $(n-1)$-dimensional ball of radius $O(\e^{1/2})$ to obtain the factor $O(\e^{(n-1)/2})$. Therefore
\be\label{Oa_st2-int}\begin{split}
\sum_{\al_{\vec k}\in\Omega_a} G(\al_{\vec k})|\Delta\al_{\vec k}|
=O(\e^{(n-1)/2})+O(\e^{{{s'}}+(n-3)/2-{{\bt'}}}).
\end{split}
\ee

For convenience, define 
\be\label{gr-fn}
\Psi(p):=\begin{cases} p^{\vartheta},&\vartheta<0,\\ \ln(1/p),& \vartheta=0,\\ 1,&\vartheta>0,
\end{cases}\quad p>0.
\ee
An estimate of $G(\al)$, $\al\in\Omega_b$, is contained in the following lemma, which is proven in Appendix~\ref{sec:prflem}. 

\begin{lemma}\label{G-al-est} Suppose all the assumptions of Lemma~\ref{lemma:lot} are satisfied. Let $G$ be defined as in \eqref{G-def}. If ${\CBa}$ is local, then 
\be\label{B-loc-in-lem}
G(\al)=O(p(\al)^\vartheta),\ \al\in\Omega_b.
\ee
If ${\CBa}$ is non-local, then 
\be\label{G-est-II-and-III}
G(\al) = O(\Psi(p(\al))),\ \al\in\Omega_b.
\ee
\end{lemma}

As $p(\al)$, $\al\in\Omega$, is bounded, the essence of estimates \eqref{B-loc-in-lem}, \eqref{G-est-II-and-III} is to control the behavior of $G(\al)$ for $\al$ close to $\dir_0$, i.e. when $p(\al)$ is small.

Since (1) $\check H''(\dir_0)$ is negative definite, (2) $\Omega_b$ can be as small as we like (but of finite size), and (3) $|\xe-x_0|=O(\e)$,
there exists $c>0$ such that $p(\al)\ge c|\al^\perp|^2$, $\al\in\Omega_b$. Recall that $\al^\perp$ is the orthogonal projection of $\al$ onto $\dir_0^\perp$. In particular, we can assume that $p(\al)>0$ in Lemma~\ref{G-al-est}, and not use absolute value bars inside the big-$O$ terms in \eqref{B-loc-in-lem}, \eqref{G-est-II-and-III}.

Suppose first that $\vartheta<0$. To estimate the contribution of directions $\al_k\in\Omega_b$ to $(\CB_\e g)(\xe)$,  replace the sum over $\al_k\in\Omega_b$ by the integral over $\Omega_b$, replace $p(\al)$ by $cr^2$, where $r$ is the radial variable in the plane $\dir_0^\perp$, and use \eqref{B-loc-in-lem}, \eqref{G-est-II-and-III} to get
\be\label{B-loc_st2}\begin{split}
\sum_{\al_{\vec k}\in\Omega_b}G(\al_{\vec k})|\Delta\al_{\vec k}|
 &=\int_{O(\e^{1/2})}^{O(1)} O(r^{2\vartheta})r^{n-2}dr=O(1)+
 O(\e^{{{s'}}+\frac{n-3}2-{{\bt'}}}).
\end{split}
\ee
Combining \eqref{Oa_st2-int} and \eqref{B-loc_st2} gives
\be\label{Oa_st2_last}\begin{split}
&({\CB_\e} g)(\xe)=\sum_{\al_{\vec k}\in\Omega_a\cup\Omega_b}G(\al_{\vec k})|\Delta\al_{\vec k}|
=O(1)+O(\e^{{{s'}}+\frac{n-3}2-{\bt'}}).
\end{split}
\ee
Therefore,
\be\label{almost-finally}
\e^{{\kappa_2}}({\CB_\e} g)(\xe)=O(\e^{\kappa_2})+O(\e^{(\bt_0-s_0)-({\bt'}-{s'})})\to0,\
\e\to0,
\ee
if $\kappa_2>0$. The other two cases, $\vartheta=0$ and $\vartheta>0$, can be considered similarly, and the result is that $\e^{{\kappa_2}}({\CB_\e} g)(\xe)\to0$ if $\kappa_2>0$ holds there as well. 

The proof of Lemma~\ref{lemma:lot} in the case $\kappa_2=0$ is more involved and is given in Appendix~\ref{sec:prf-kappa2-zero}. 
\end{proof}

The final result of this section is that the last two terms on the right in \eqref{Bg-split} do not contribute to the DTB.
\begin{lemma}\label{lem:lot-2} Suppose all the assumptions of Lemma~\ref{lemma:lot} are satisfied, and $g_0$ is as in \eqref{gfn-def}. The last two terms on the right in \eqref{Bg-split} do not contribute to the DTB, i.e. the result of computing these terms from discrete data satisfies relations analogous to \eqref{finally-kappa-all}.
\end{lemma}

\begin{proof}
%There are two final pieces that need to be considered. The first one is of the kind $\CB_\e(\chi(p)g_0(\al,p))$, where $\CB$ and $g_0$ are the same as in Theorem~\ref{lead-thrm-pos}. Here $\chi$ is a smooth function that satisfies $\chi(p)\equiv 0$, $|p|\le c$, and $\chi(p)\equiv 1$, $|p|\ge 2c$, for some $c>0$ sufficiently large. We have to consider this term, because $g$ is compactly supported (see Section~\ref{sec:main-ass}), while $g_0$ is not compactly supported. All the other functions $g$ considered in our proof are compactly supported as well. 
By the argument following \eqref{three-fns}, the series
\be\label{G-def-remder}
G_{\xc}(\al)=\sum_j({\CBa} \ik_\e)(\al\cdot \xe-p_j) \chi(p_j)g_0(\al,p_j)
\ee
converges absolutely and uniformly in $\al\in\Omega$. Moreover, 
\be\label{G-conv}
\lim_{\e\to0}G_{\xc}(\al)=\int B(\al,\al\cdot x_0-p) \chi(p)g_0(\al,p)dp,
\ee
because $B(\al,\al\cdot \xe-p)\chi(p)$ is smooth, and $\ik$ is exact to the degree $L_\bt$. This implies that $\lim_{\e\to0}({\CB_\e} (\chi g_0))(\xe)$ exists, is independent of $\check x$, and bounded, i.e. $\chi g_0$ does not contribute to the DTB.

The final term to consider is $\CB_\e g$, where $g$ may have a non-zero leading term in \eqref{gfn-exp}, but $\CB$ is such that (1) all the terms in the expansion \eqref{B-lim} with $\bt_j\ge 0$ are identically zero, and (2) we do not assume that $\tilde B(\al,\la)$ is a homogeneous function of $\la$. In this case, the operator $\CB_{1d}$ is smoothing of finite degree. With $g$ being continuous (recall that $s_0>1$ in \eqref{gfn-exp}, cf. \eqref{kappa-assns}), the function $\CB_{1d}g$ is continuous. By an easy calculation we get that $\lim_{\e\to0}(\CB_\e g)(\xe)=(\CB g)(x_0)$, so $\CB_\e g$ does not contribute to the DTB as well.
\end{proof}

Combining Theorems~\ref{lead-thrm-pos}, \ref{lead-thrm-zero} and Lemmas~\ref{lemma:lot}, \ref{lem:lot-2} proves our main result.

\begin{theorem}\label{thrm:lot-main} Let $x_0\in\s$ be generic. Suppose $\CB$ and $g$ are the same as in \eqref{B-oper}-\eqref{B-adtnl} and \eqref{g-smooth}--\eqref{gfn-exp-unif}, \eqref{c-ratios-req}, respectively, $\kappa_1,\kappa_2\ge0$, and all the assumptions in Section~\ref{sec:main-ass} hold. One has
\be\label{fe-lim-simple-final}
\lim_{\e\to0}\e^{{\kappa_2}}(\CB_\e g)(\xe)=(\ik*\mu)(h) \text{ if } \kappa_2>0,
\ee
where $\mu$ is given by \eqref{mu-in-thrm}. Also, for some $c_\e$,
\be\label{fe-lim-simple-v2-final}
\lim_{\e\to0}\left[(\CB_\e g)(\xe)-c_\e\right]=(\ik*\mu)(h) \text{ if } \kappa_2=0,
\ee
where $\mu$ is given by \eqref{mu-in-thrm-k20}. The quantity $c_\e$ depends on $\e$, but is independent of $\check x$. 

If there exists a function $f\in L^1(\br^n)$ such that its Radon transform $\hat f(\al,p)$ satisfies \eqref{g-smooth}--\eqref{gfn-exp-unif}, \eqref{c-ratios-req} and the leading terms of the expansions of $\hat f$ and $g$ coincide, then $\mu$ is given by \eqref{mu-def-1}, \eqref{A-def} if $\kappa_2>0$, and by \eqref{new-mup} -- if $\kappa_2=0$.
\end{theorem}

{To put it simply, Theorem~\ref{thrm:lot-main} asserts that the DTB of a reconstruction equals to a suitably rescaled convolution of the interpolation kernel $\ik$ and the corresponding CTB $\mu$.}

\section{Accounting for finite detector pixel size}\label{sec:pixsize}

In the idealized model of a tomographic experiment the conventional assumption is that data represents discrete values of the Radon transform $\hat f(\al_k,p_j)$, where $p_j$ is the center of the $j$-th detector pixel. A more accurate model is that the data are the values of $\hat f(\al,p)$ averaged over the area of each pixel:
\be\label{sm-data}
\hat f_\nu(\al_{\vec k},p_j):=\int \frac1\e \nu\left(\frac{p_j-p}\e\right) \hat f(\al_{\vec k},p)dp.
\ee
Here $\nu$ is a sufficiently smooth compactly supported function, which models the detector response. This function is normalized, i.e. $\int\nu(p)dp=1$. Similarly, in a more general case we can assume that the data are $g_\nu=\nu*g$, where the convolution is in $p$.
Fortunately, all the main results and conclusions obtained in the previous sections still apply. More precisely, Theorem~\ref{thrm:lot-main} (and Lemma~\ref{lem-phi-lim} as a particular case of Theorem~\ref{thrm:lot-main}) still holds after a simple modification. The only difference is that $\ik$ is replaced by $\ik*\nu$ in \eqref{fe-lim-simple} and \eqref{fe-lim-simple-v2}.

Indeed, a simple calculation shows that the analogues of $\psi$ and $\Psi$ in \eqref{three-fns} become
\be\label{new-psi}
\psi_\nu(t,p):=\int \psi(t,q)\nu(q-p)dq,\ \Psi_\nu(t-p):=\int \Psi(t-q)\nu(q-p)dq.
\ee
Lemmas~\ref{contverpsi}, \ref{psi-asympt} apply to $\Psi_\nu$ and $\psi_\nu$ without changes.
Indeed, in the proof of Lemma~\ref{contverpsi} the properties of $\ik$ that we used were that $\ik$ is sufficiently smooth, compactly supported, and $\tilde\ik(0)=1$. Clearly, $\ik*\nu$ has all of these  properties. To prove the statement about $\psi_\nu$ we follow the proof of Lemma~\ref{psi-asympt}, but replace $\CA$ with $\CA*\nu$.
Hence the derivation \eqref{newlt-1}--\eqref{uzf-int} works with $\psi$ replaced by $\psi_\nu$. The analogue of \eqref{uzf-int} becomes
\be\label{uzf-int-v2}
J=\frac12\int \Psi_\nu(h+t) t_+^{(n-3)/2}dt.
\ee
Thus, the only modification to \eqref{J-ft} is to insert the factor $\tilde\nu(\la)={\mathcal F}\nu$ in the integrand, and the desired assertion is obvious. Similar modifications are done in the case $\kappa_2=0$.

Since smoothing the data increases smoothness of the reconstruction, it is clear that Theorem~\ref{thrm:lot-main} holds when $\hat f$ is smoothed as well.

In a more specific case of 2D LT, the conclusions in Sections~\ref{sec:LTline} and \ref{sec:LTremote} hold also. In \eqref{lt-line-nosm}, \eqref{lt-line-st2}, we replace 
\be\label{phi-repl}
\phi(t,\cdot,\cdot) \to \phi_\nu(t,\cdot,\cdot):= \int \nu(t-p)\phi(p,\cdot,\cdot) dp.
\ee
In \eqref{remsing-3}, we replace $\psi$ with $\psi_\nu$ (cf. \eqref{new-psi}). Qualitatively, everything remains the same. The line artifact from a straight edge in $\text{singsupp}(f)$ is of strength $O(1/\e)$, and the oscillations in $\lte$ away from $\text{singsupp}(f)$ are of magnitude $O(\e^{-1/2})$ even with data smoothing. To change the conclusions qualitatively, the smoothing should be on a scale $\de$ such that $\de/\e\to\infty$, $\e\to0$.

\section{Numerical Experiments}\label{sec:numerix}

In all the experiments below that use a 2D reconstruction grid, the latter is of size $1001 \times 1001$ and covers the square $[-L,L]\times [-L,L]$ with $L=5$. The Radon data are given at the points 
\be\label{angles}\begin{split}
\al_k&=\Delta\al(k+\sqrt 2),\ \Delta\al=2\pi/n_0;\\
p_j&=-p_{\text{max}}+j\Delta p,\ \Delta p = \e=2p_{\text{max}} / n_0,\ 
p_{\text{max}} = 1.1L\sqrt 2. 
\end{split}
\ee
The shift $\sqrt 2$ in the formula for $\al_k$ is introduced to avoid any special angles. The coefficient 1.1 when computing $p_{\text{max}}$ is introduced to ensure that the data cover a region slightly larger than the selected reconstruction area.

As the kernel $\ik$ that satisfies conditions IK1--IK4 in Section~\ref{ltwosm} we used the function (cf. \cite{kat19a}):
\be\label{final-form}
\ik(t)=0.5(B_3(t)+B_3(t-2))+4B_3(t-1)-2(B_4(t)+B_4(t-1)).
\ee  
Here $B_n$ denotes the cardinal B-spline of degree $n$ supported on $[0,n+1]$.

\begin{figure}[h]
{\centerline{
{\hbox{
{\epsfig{file=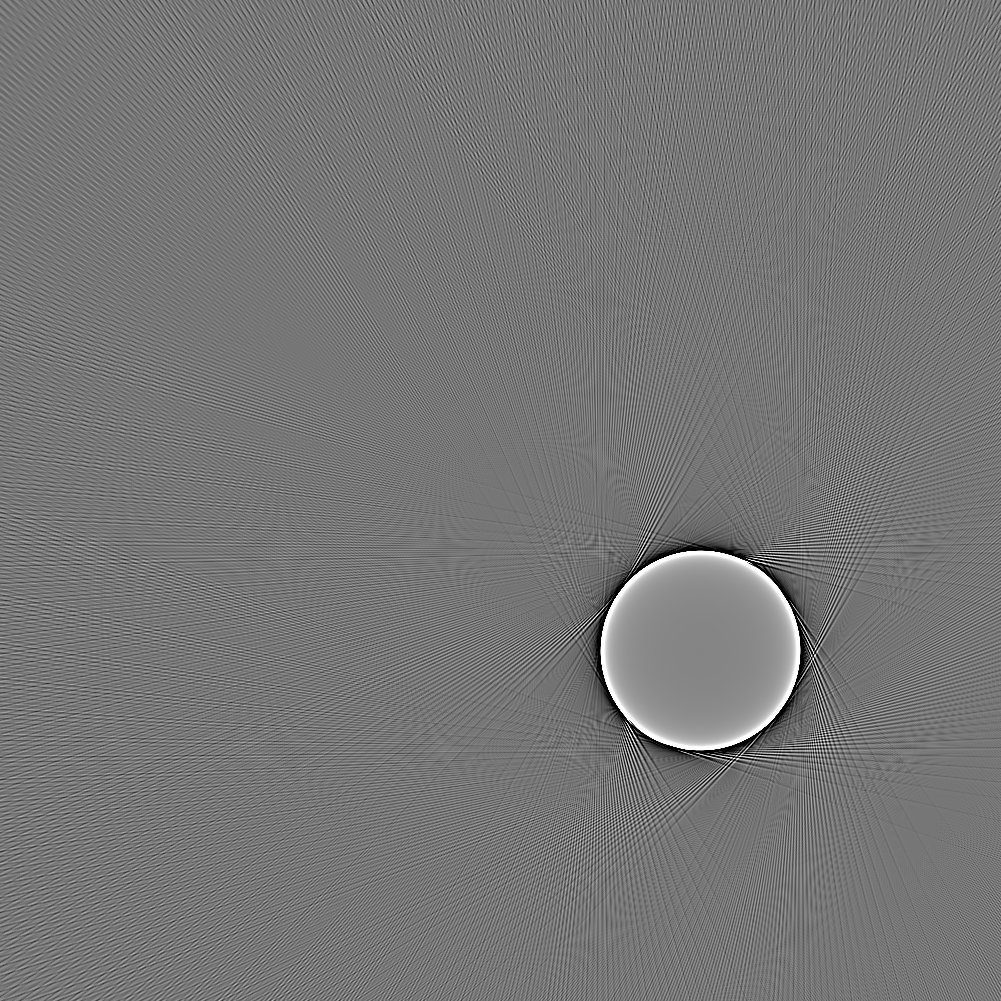, width=5cm}}
{\epsfig{file=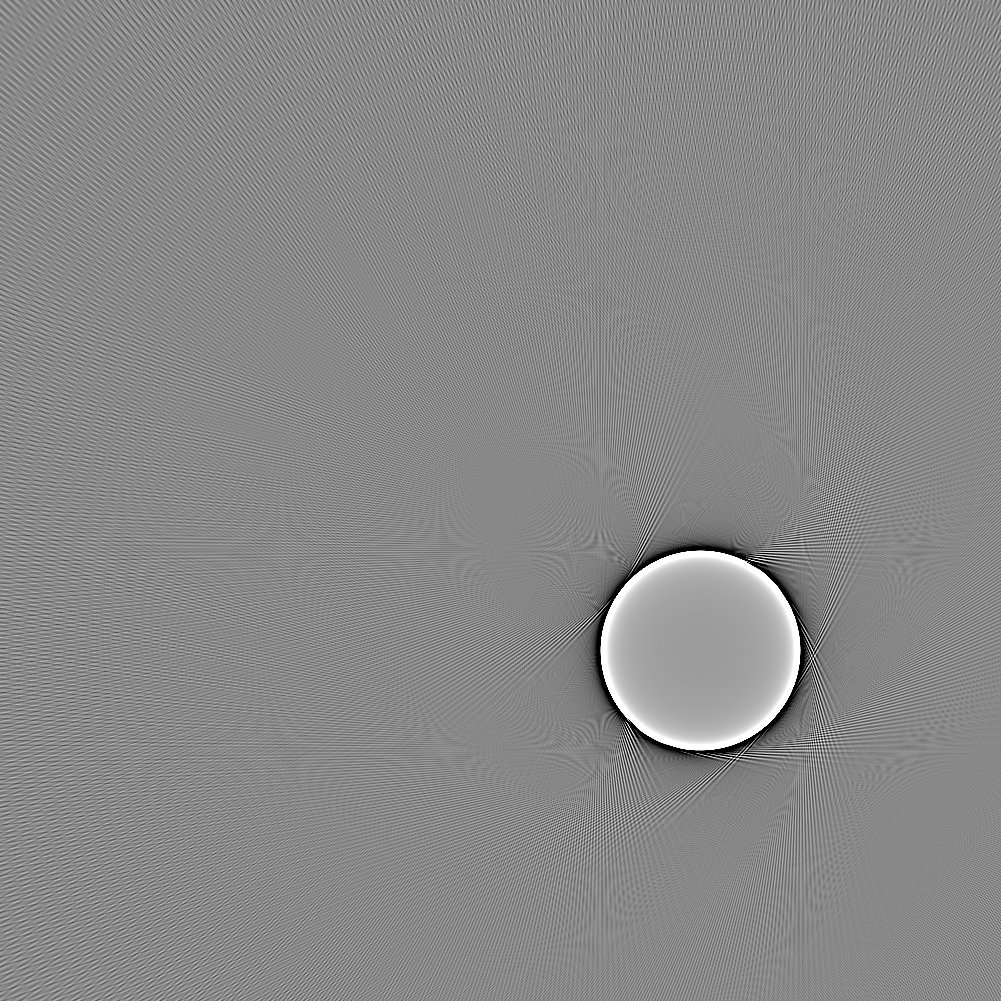, width=5cm}}
}}}}
\caption{Reconstructed $\lte$, $n_0=1000$. Left panel: without data smoothing, right panel: with data smoothing.}
\label{fig:1000}
\end{figure}

\begin{figure}[h]
{\centerline{
{\hbox{
{\epsfig{file=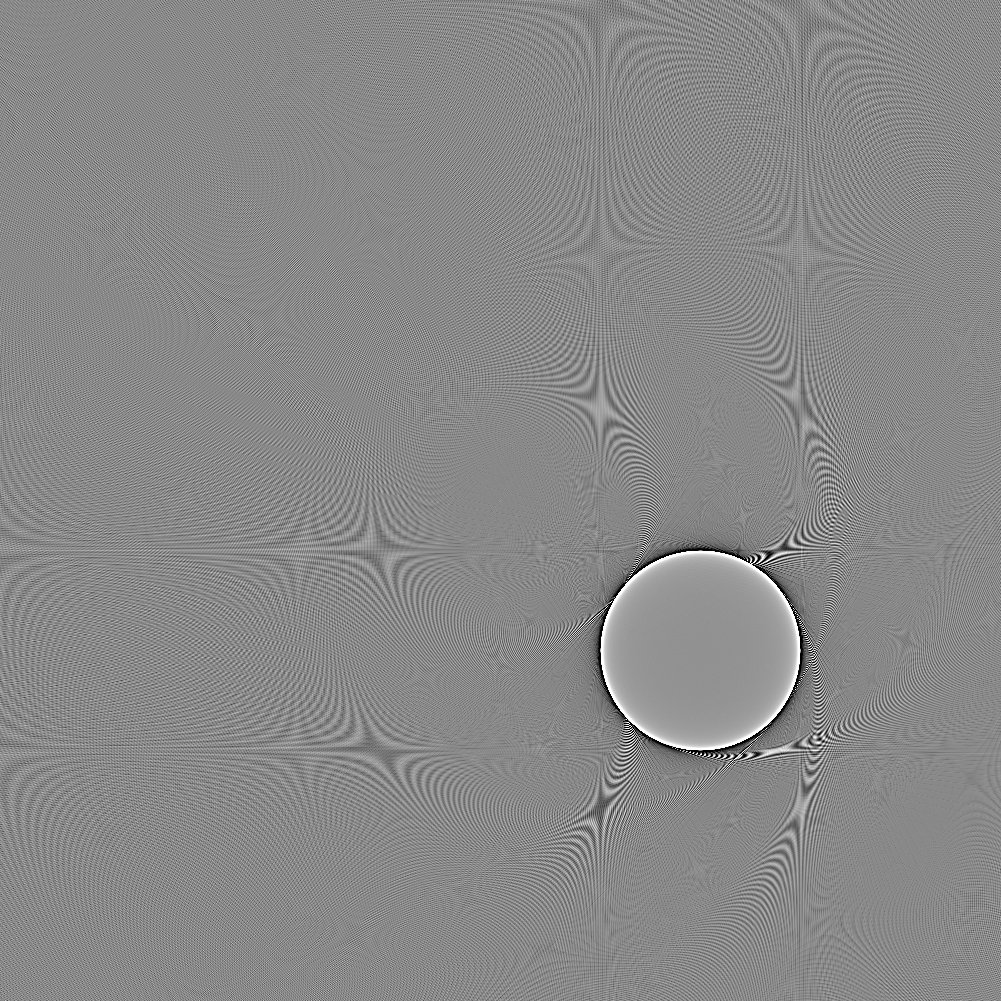, width=5cm}}
{\epsfig{file=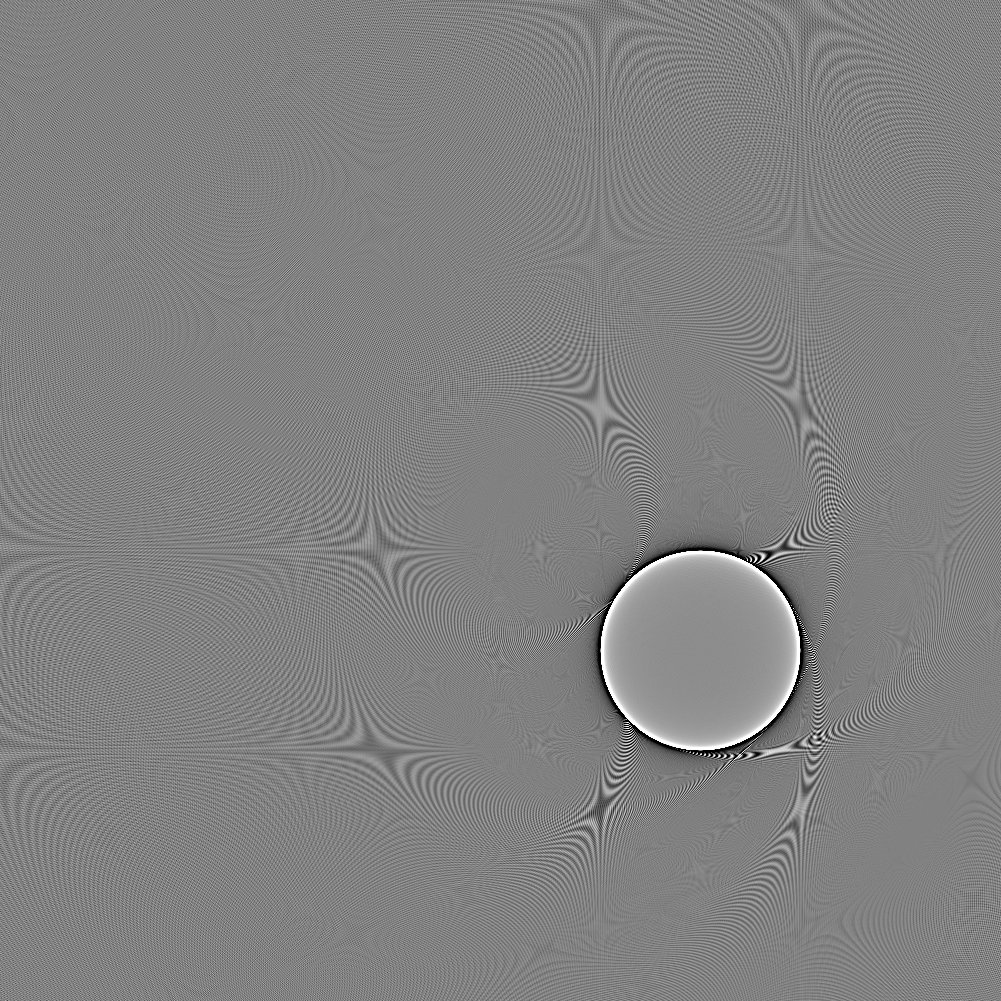, width=5cm}}
}}}}
\caption{Reconstructed $\lte$, $n_0=2500$. Left panel: without data smoothing, right panel: with data smoothing.}
\label{fig:2500}
\end{figure}

\begin{figure}[h]
{\centerline{
{\hbox{
{\epsfig{file=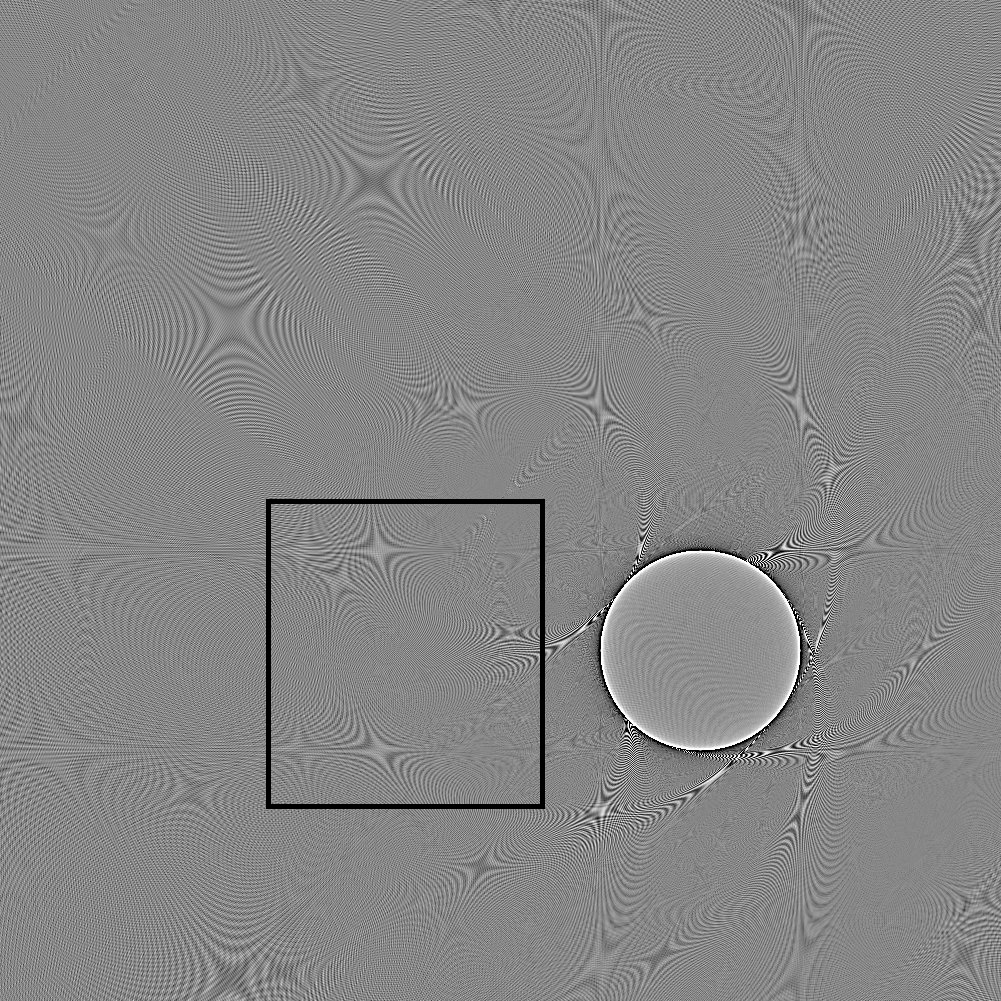, width=5cm}}
{\epsfig{file=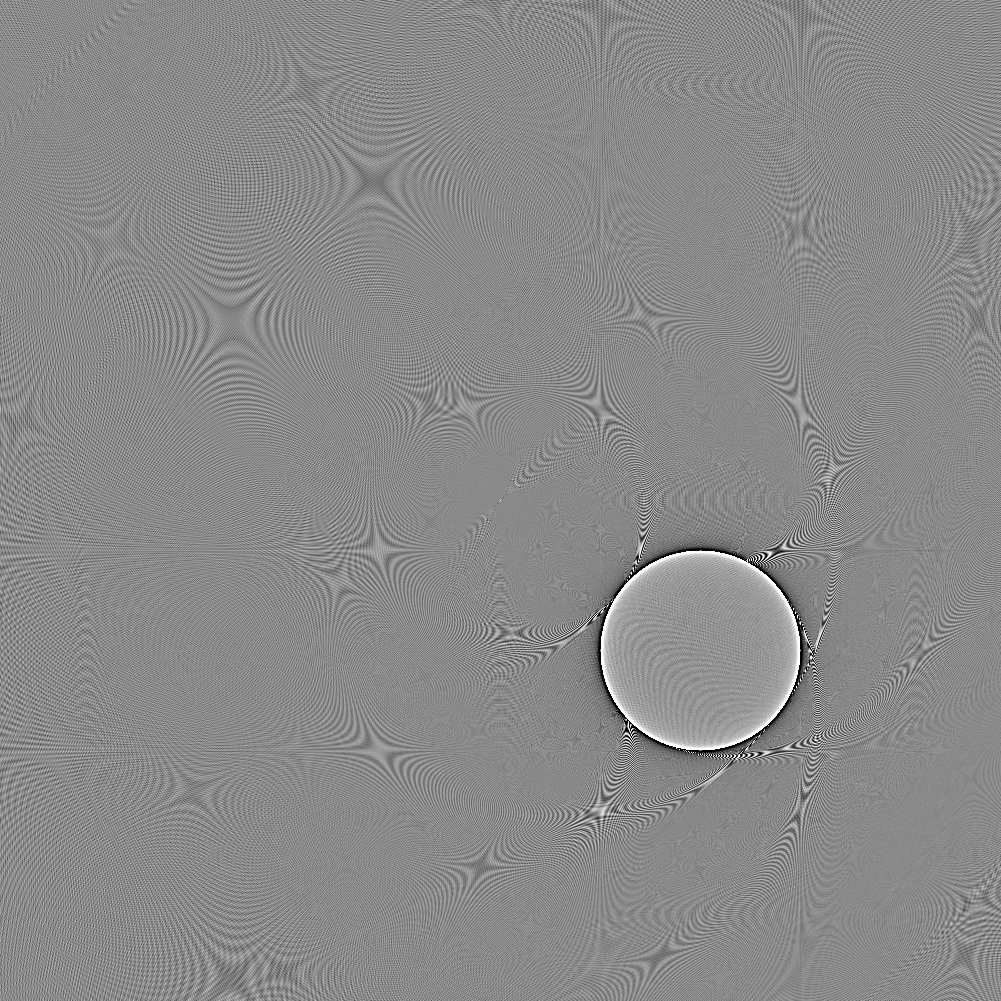, width=5cm}}
}}}}
\caption{Reconstructed $\lte$, $n_0=5000$. Left panel: without data smoothing, right panel: with data smoothing. The rectangle on the left panel is used for computing the standard deviation in a region of the image.}
\label{fig:5000}
\end{figure}

The first group of experiments uses a disk phantom with center $x_c=(2,1.5)$, radius $R=1$, and uniform density 1. The reconstructed $f_{\Lambda\e}$ are shown in Figures~\ref{fig:1000}, \ref{fig:2500}, and \ref{fig:5000}. They correspond to $n_0=1000$, 2500, and 5000, respectively. Left panels show reconstructions from the discrete values of the Radon transform (as described in Section~\ref{ltwosm}), while right panels show reconstructions from the Radon transform averaged over detector pixels, cf. \eqref{sm-data}. For each pixel the window function $\nu$ is constant inside the interval of length $\Delta p$, its support is centered at $p_j$, and it is normalized so that $\int\nu(p)dp=1$. As expected, no qualitative difference is visible between the left and right panels corresponding to the same value of $n_0$. 

\begin{figure}[h]
{\centerline{
{\vbox{
{\hbox{
{\epsfig{file=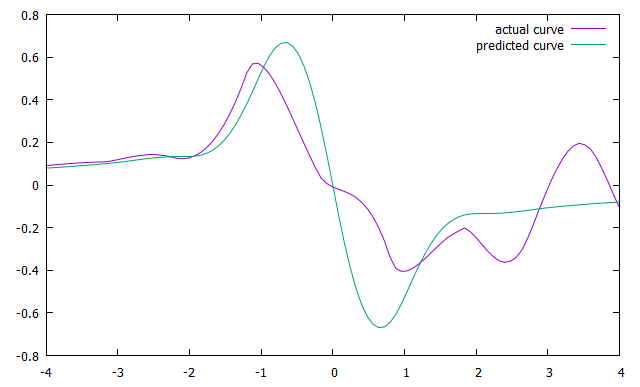, width=6.5cm}}
{\epsfig{file=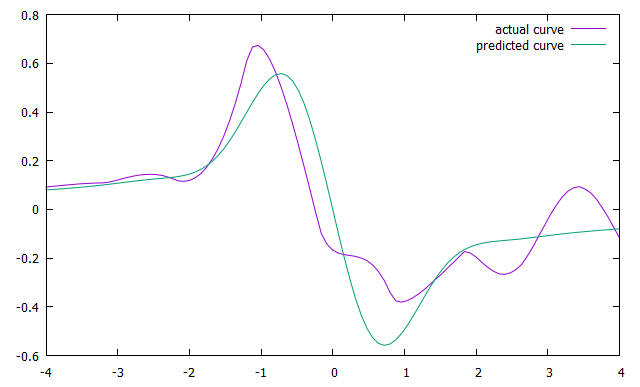, width=6.5cm}}
}}
{\hbox{
{\epsfig{file=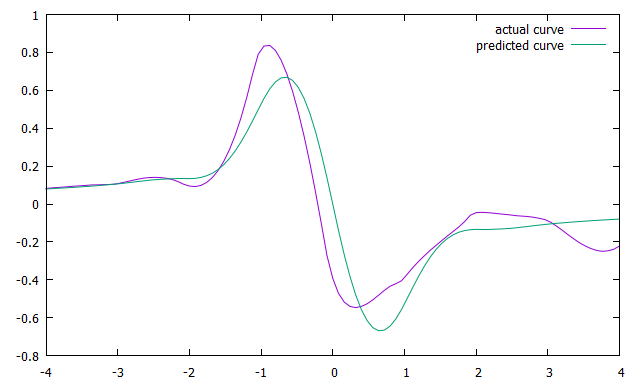, width=6.5cm}}
{\epsfig{file=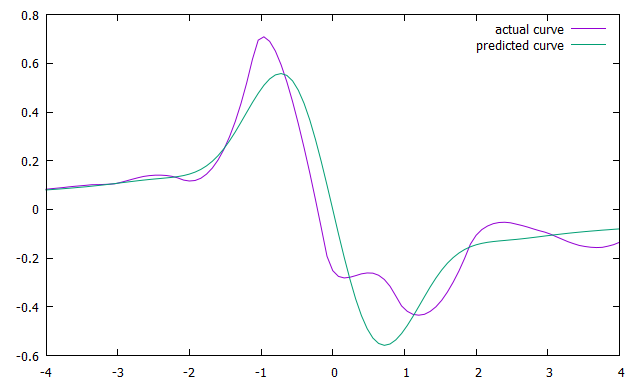, width=6.5cm}}
}}
{\hbox{
{\epsfig{file=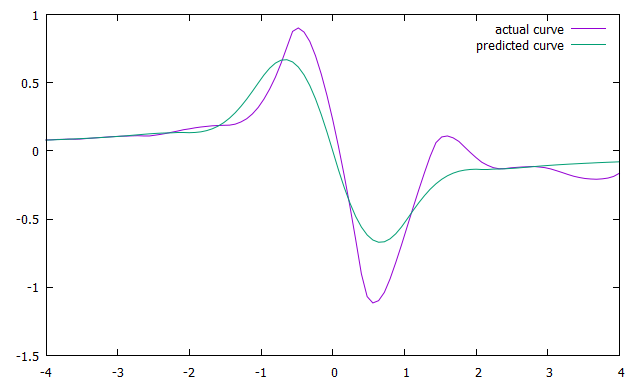, width=6.5cm}}
{\epsfig{file=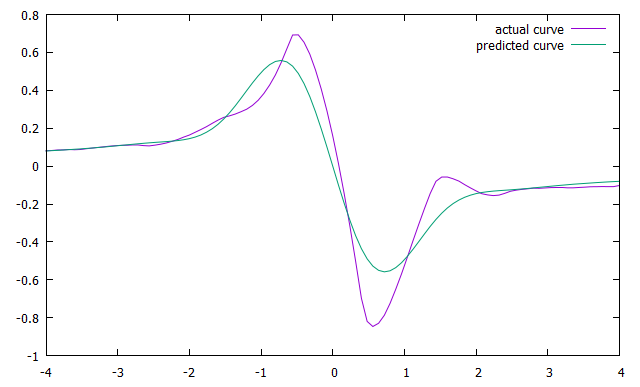, width=6.5cm}}
}}
}}}}
\caption{Reconstructed and predicted edge response, $\al_0=0.73\pi$. $a=-1.006592$. Top to bottom: $n_0=1000$, 2500, 5000. Left panels: without smoothing, right panels: with smoothing.}
\label{fig:edge_073pi}
\end{figure}

\begin{figure}[h]
{\centerline{
{\vbox{
{\hbox{
{\epsfig{file=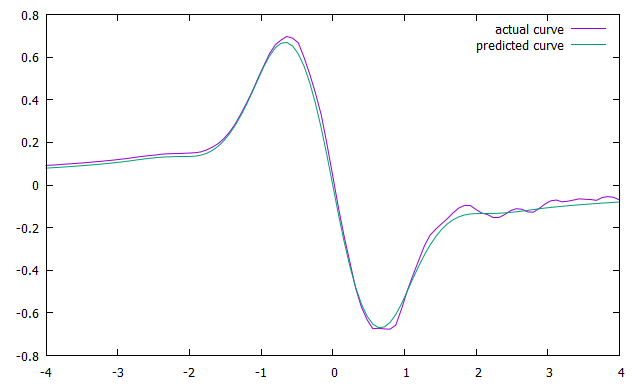, width=6.5cm}}
{\epsfig{file=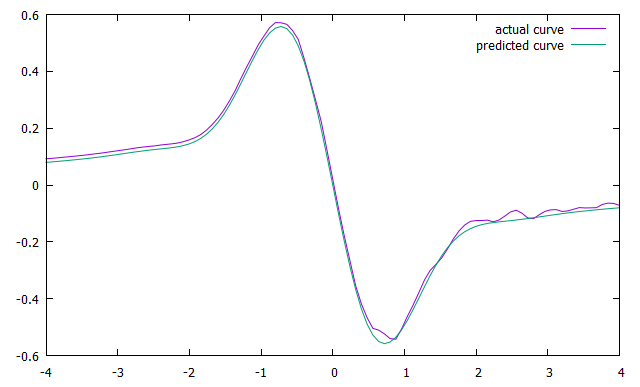, width=6.5cm}}
}}
{\hbox{
{\epsfig{file=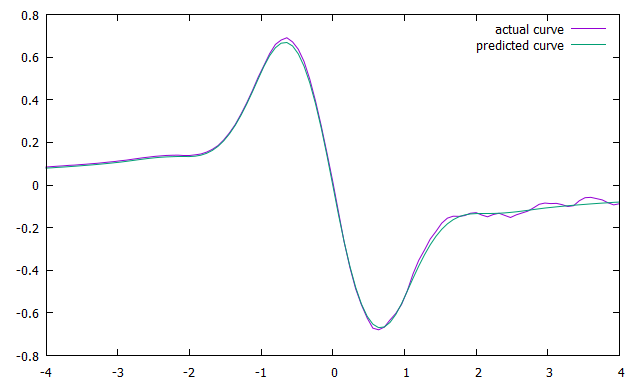, width=6.5cm}}
{\epsfig{file=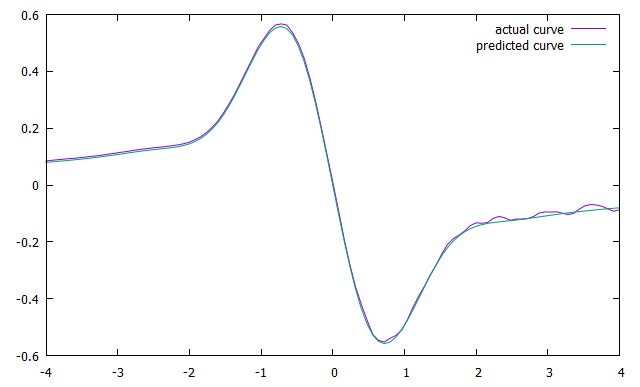, width=6.5cm}}
}}
{\hbox{
{\epsfig{file=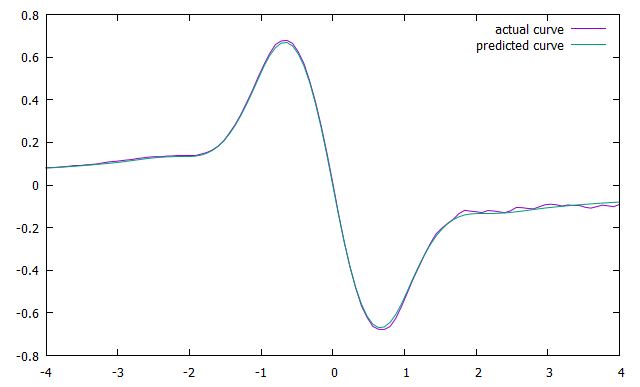, width=6.5cm}}
{\epsfig{file=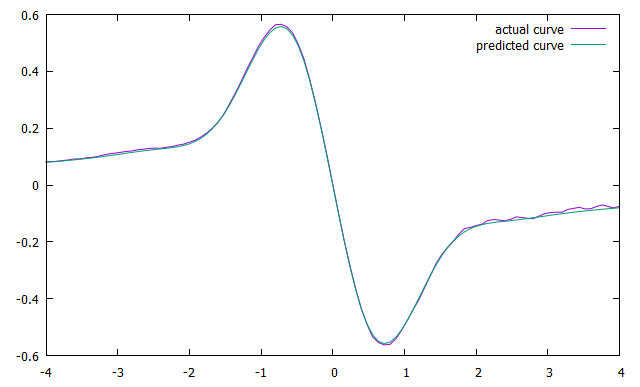, width=6.5cm}}
}}
}}}}
\caption{Reconstructed and predicted edge response, $\al_0=\sqrt 2 \pi$. $a=0.617327$. Top to bottom: $n_0=1000$, 2500, 5000. Left panels: without smoothing, right panels: with smoothing.}
\label{fig:edge_sqrt2pi}
\end{figure}

To verify that the predicted edge response (cf. Lemma~\ref{lem-phi-lim}) is accurate, we  compute $\e f_{\Lambda\e}(\xe)$ (i.e., with a factor of $\e$) on a fine grid along two radial lines through the boundary of the disk. The intersection points with the boundary are $x_0 = x_c + R\dir_0$, $\dir_0=(\cos\al_0,\sin\al_0)$, where $\al_0=0.73\pi$ for the first line, and $\al_0=\sqrt 2\pi$ for the second line. The reconstruction grid covers the region $\xe=x_0+ h \e\dir_0$, $|h|\le 4$. The predicted and actual edge responses are shown in Figures~\ref{fig:edge_073pi} and \ref{fig:edge_sqrt2pi}. The value of $h$ is shown on the $x$-axis of each of the plots. We also compute the value of the parameter $a:=(\Theta_0^\perp\cdot x_0){{\kappa}}$ for each $\al_0$ (cf. \eqref{ri-term}). In \eqref{ri-term} we assume that $\dir_0$ is an interior normal, while in this section $\dir_0$ is an exterior normal, so the values of $a$ here and in \eqref{ri-term} have opposite signs. Note that according to \eqref{angles}, $\kappa=\Delta\al/\Delta p$ is independent of $n_0$. 

In Figure~\ref{fig:edge_073pi} the match between the predicted and actual edge responses is bad, while in Figure~\ref{fig:edge_sqrt2pi} it is quite accurate. Recall that the edge response is derived under the assumption that $x_0$ is generic, i.e. $a$ is irrational. We have $a=-1.006592$ in Figure~\ref{fig:edge_073pi}, and $a=0.617327$ in Figure~\ref{fig:edge_sqrt2pi}. {In the first case, $a$ can be accurately approximated by a rational number of the form $j/m$, where $j\in\mathbb Z$, $m\in\mathbb N$, and $m$ is small. In the second case, to accurately approximate $a$ by a rational number requires a larger denominator $m$.} Consequently, $x_0$ is almost non-generic in the first case, and generic - in the second case. This experiment demonstrates that local tomography is sensitive to {whether $a$ is close to a rational number with a small denominator.} This is in contrast with exact reconstruction (see \cite{kat19a}), which is much less sensitive to how non-generic a point $x_0\in\text{singsupp}(f)$ is. 

\begin{figure}[h]
{\centerline{
{\hbox{
{\epsfig{file=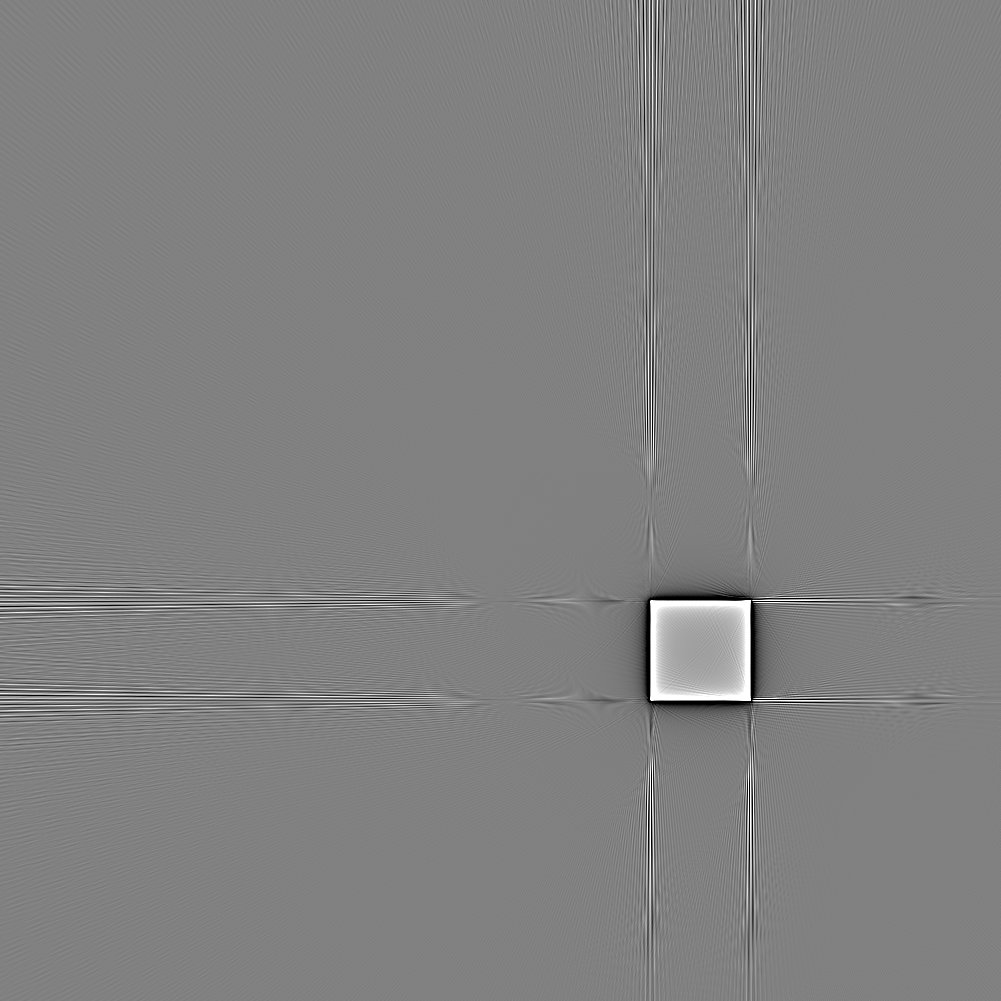, width=4.5cm}}
{\epsfig{file=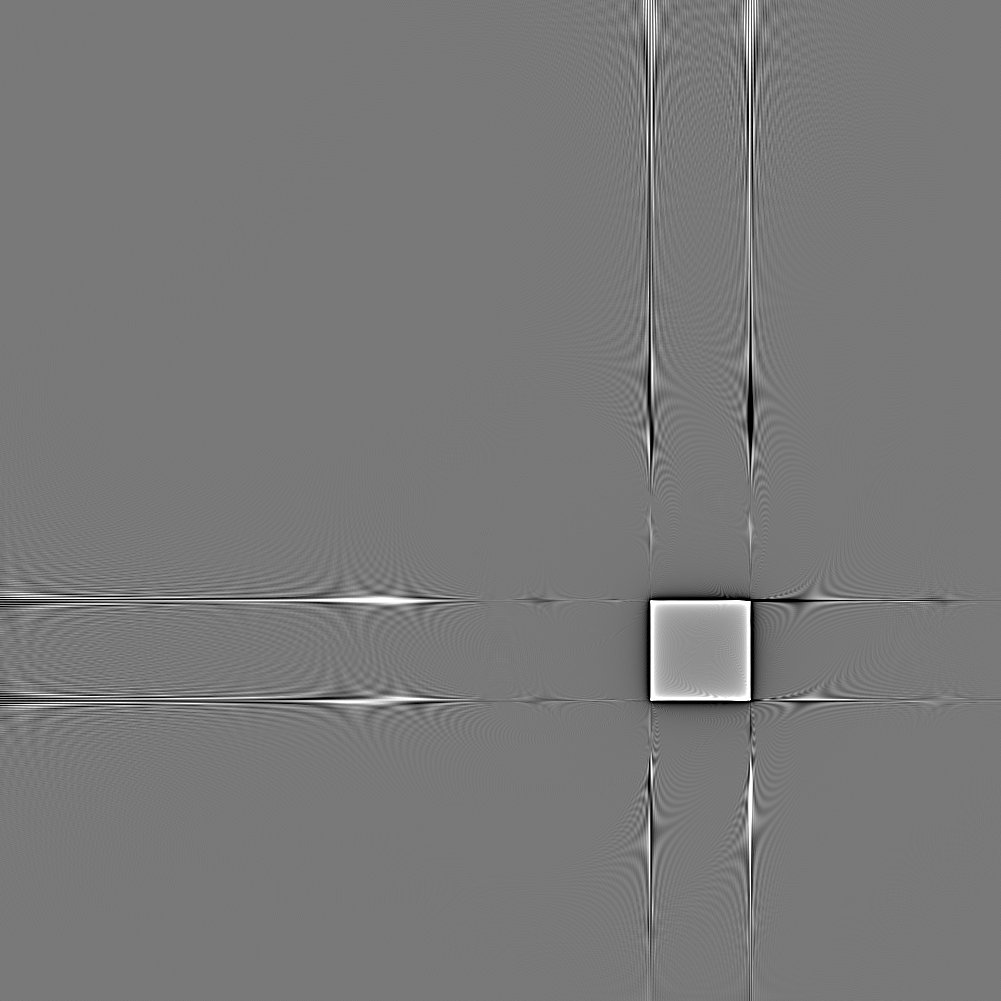, width=4.5cm}}
{\epsfig{file=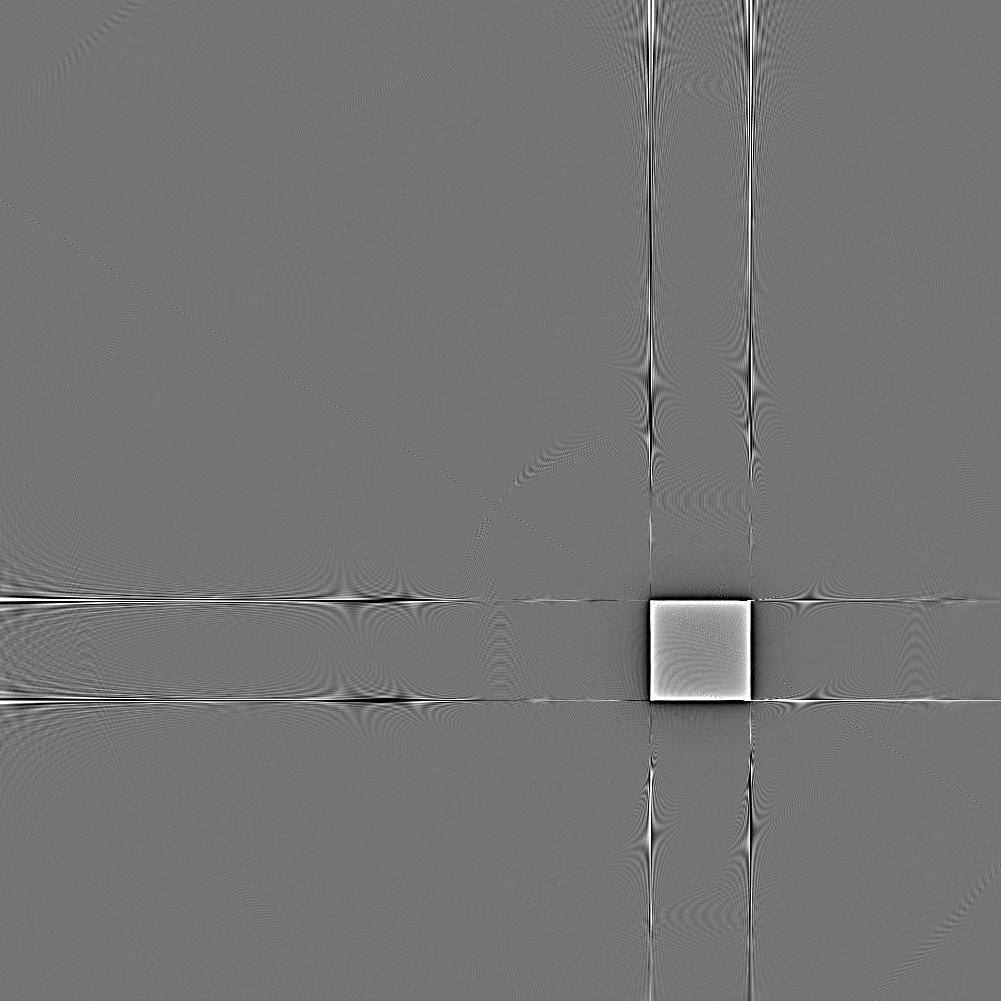, width=4.5cm}}
}}}}
\caption{Reconstructed $f_{\Lambda\e}$. Left to right: $n_0=1000$, 2500, 5000. The phantom is a square, all reconstructions are with smoothing.}
\label{fig:square}
\end{figure}

To demonstrate non-local artifacts (cf. \eqref{lt-line-st2}), we simulate a square with center $x_c=(2,1.5)$, side length 1, and uniform density 1. To avoid irrelevant complications related to $\hat f(\al,p)$ being discontinuous, we only show the results with $\hat f$ averaged over detector pixels, see Figure~\ref{fig:square}. The nonlocal artifacts are clearly visible. They extend far from the square itself, and exhibit a complicated pattern.

Finally, we verify that away from $\text{singsupp}(f)$, the reconstructed $\lte$ does not converge pointwise as $\e\to0$, $n_0\to\infty$  (cf. \eqref{remsing-3}). We select an identical rectangle (a total of 84036 pixels) in all six images and compute the standard deviation of $\lte$ within each rectangle. The obtained values are as follows: without smoothing -- 1.2751, 2.0912, 3.0335, and with smoothing -- 0.8639, 1.4079, 2.0881. The values are given in the order $n_0=1000$, 2500, and 5000. The rectangle is shown in Figure~\ref{fig:5000}, left panel. The observed values of standard deviation are in qualitative agreement with the $\sim\e^{-1/2}$ (or, $\sim n_0^{1/2}$) dependence in \eqref{remsing-3}. The ratios $2.0912/1.2751=1.6400$,  $3.0335/1.2751=2.3790$ in the no smoothing case, and $1.4079/0.8639=1.6297$, $2.0881/0.8639=2.4171$ in the data smoothing case are fairly close to the expected values $2.5^{1/2}\approx 1.5811$, $5^{1/2}\approx 2.2361$.

\bibliographystyle{plain}
\bibliography{bibliogr_A-K,bibliogr_L-Z}

\appendix

\section{Proof of Lemma \ref{lem:mu-der}}\label{sec:mu-deriv}

By \eqref{losing} and \eqref{f-def}, we compute the asymptotics of the integral
\be\label{wf-lim}\begin{split}
&\e^{m-{\kappa_1}}(f,P_m(\pa_x)w_\e)=\frac{\e^{m-\kappa_1}}{(2\pi)^n}\ioi J_\e(\e\la)\la^{m+(n-1)}d\la\\
&\hspace{3.1cm}=\frac{1}{(2\pi)^n}\ioi \e^{-(s_0+\frac{n-1}2)}J_\e(\eta)\eta^{m+(n-1)}d\eta
,\ \e\to0,\\
&J_\e(\eta):=\int_{S^{n-1}}P_m(-i\al)\tilde w(-\eta\al)\fs((\eta/\e)\al)e^{i(\eta/\e) (H(\al)-\al\cdot x_0)}d\al,
\end{split}
\ee 
where $\tilde w={\mathcal F} w$. 

Due to $\tilde v_j\in\coi(\Omega\cup(-\Omega))$, $j\ge 0$ (cf. \eqref{f-lim}), integration on the last line in \eqref{wf-lim} can be restricted to $\Omega\cup(-\Omega)$.
By construction, on this set the exponent in \eqref{wf-lim} has two stationary points: $\al=\pm\dir_0$. At these points $H'(\pm\dir_0)=x_0$, and the Hessians $\check H''(\pm\dir_0)$ are nondegenerate. Clearly, $\check H''(-\dir_0)=-\check H''(\dir_0)$. 

When working with the integral that defines $J_\e(\eta)$ it is convenient to parametrize $\al\in\pm\Omega$ in terms of $\al^\perp$. The expansion in \eqref{f-lim} is uniform and can be differentiated with respect to $\al$ (cf. \eqref{vfn-exp-unif}), therefore the stationary phase method (see eq. (7.7.13) in \cite{hor}) and \eqref{scnd-der-shifted} imply for each $\eta>0$:
\be\label{limf-explicit}
J_\e(\eta)=(\e/\eta)^{\frac{n-1}2}\sum_{\al\in\{\pm\dir_0\}}P_m(-i\al)\tilde w(-\eta\al)\fs((\eta/\e)\al)/\omega(\al)+O(\e^{\frac{n+1}2}),\ \e\to0,
\ee 
where $\omega(\al)$ is defined in \eqref{f-lim}, and 
\be\label{Je-est}\begin{split}
|J_\e(\eta)|\le W(\eta)(\e/\eta)^{s_0+\frac{n-1}2},\ 
W(\eta):=c\sum_{k=0}^K(1+\eta)^k\sum_{|\nu|=k}\max_{\al\in S^{n-1}}|\pa_\al^{\nu}\tilde w(-\eta\al)|,
\end{split}
\ee 
for some $K\in\mathbb N$ and $c>0$ independent of $\eta$. 
%Upon changing variables $\la\to\e\la$, \eqref{wf-lim} implies that we have to compute the  limit of the following expression as $\e\to0$:
%\be\label{keyf-lim}
%\ioi \left[\e^{-(s_0+\frac{n-1}2)}J_\e(\eta/\e)\right]\eta^{m+(n-1)}d\eta.
%\ee 
Consequently,
\be\label{Je-bound}
\left|\e^{-(s_0+\frac{n-1}2)}J_\e(\eta)\eta^{m+(n-1)}\right|
\le \eta^{m-\kappa_1-1}W(\eta),\ \eta>0.
\ee 
Recall that $\tilde w(\eta\al)$ and, therefore, $W(\eta)$, are rapidly decreasing functions. Thus, the integrand on the second line in \eqref{wf-lim} admits a uniform as $\e\to0$ bound in $L^1(\br_+)$ if $m\ge\lceil\kappa_1\rceil$. In this case we can envoke the dominated convergence theorem and take the limit as $\e\to0$ inside the integral in \eqref{wf-lim}. Clearly, $\lim_{\e\to0}\fs((\eta/\e)\al)\e^{-s_0}=\omega(\al)\fs_0(\al)/\eta^{s_0}$ for any $\eta>0$. Combining \eqref{wf-lim}--\eqref{limf-explicit} now gives
\be\label{limf-res}\begin{split}
\lim_{\e\to0}\e^{m-{\kappa_1}}(f,P_m(\pa_x)w_\e)
&=\frac1{(2\pi)^n}\sum_{\al\in\{\pm\dir_0\}}
\fs_0(\al)\ioi P_m(-i\la \al)\tilde w(-\la\al)\la^{-({\kappa_1}+1)}d\la\\
&=\frac1{2\pi}
\int P_m(-i\la \dir_0)\tilde w(-\la\dir_0)\tilde\mu(\la)d\la,
\end{split}
\ee 
where the distribution $\tilde\mu\in {\mathcal S}'(\br)$ is the same as in \eqref{cpm-complete}, and the lemma is proven. A somewhat related argument is in the proof of Proposition 18.2.2 in \cite{hor3}.

\section{Proof of Lemma \ref{psi-incr}}\label{psi-incr-prf}

Due to \eqref{psi-Psi}, we may assume that both $t$ and $p$ are bounded, e.g. $|t|,|p|\le a<\infty$ for some $a$. This is assumed in all of the proof. We begin by considering the first equation in \eqref{an-psi-pr-1}. 

\noindent
{\bf Case I. ${\CBa}$ is local.} In this case, $\CBa\ik=c\tilde B_0(\al)\ik^{(\bt)}$. By IK2$'$, the sum in \eqref{three-fns} contains finitely many terms, and the desired assertion follows from the assumption $\ik^{(\bt+1)}\in L^\infty(\br)$ (cf. IK3$'$)

\noindent{\bf Case II. ${\CBa}$ is not local, $\bt\not\in\mathbb N$.}
Set $k:=\lceil\bt\rceil$, $\nu:=\{\bt\}$, $0<\nu<1$. Then $(\CBa\ik)(r)$ is a linear combination of the following terms:
\be\label{del-terms}
\int \ik^{(k)}(q)(q-r)_\pm^{-\nu} dq,
\ee
and the coefficients are $C^\infty(\Omega)$ functions of $\al$. Since $|t|\le a$, we can find $N$ so that $|j|\ge N$ implies $r=t-j\not\in\ts(\ik)$. Integrating by parts $k$ times and differentiating with respect to $r$ in \eqref{del-terms} gives $(\CBa\ik)'(r)=O(|r|^{-(\bt+2)})$, $r\to\infty$. Together with $s-1-\bt<0$ this yields
\be\label{psi-bj-part-II}
\sum_{|j|\ge N}\left[(\CBa\ik)(t+\e-j)-(\CBa\ik)(t-j)\right] \CA(j-p)=O(\e).
\ee
In the remaining terms, $j$ is bounded (together with $t$ and $p$). To estimate the remainder, we look at the integrals
\be\label{psi-smj-part-II}
J_\pm(t):=\int \ik^{(k)}(q)\left[(q-(t+\e))_\pm^{-\nu}-(q-t)_\pm^{-\nu}\right]dq,
\ee
By IK3$'$, $\ik^{(k)}(q)\in L^\infty(\br)$. Each version of the expression in brackets (with `$+$' and with `$-$') changes sign only once: when $q=t+\e$ or $q=t$. Therefore, by IK2$'$,
\be\label{Jpm-II}
|J_\pm(t)|\le O(1)\sup_{q}\left[(q+\e)_\pm^{1-\nu}-q_\pm^{1-\nu}\right]=O(\e^{1-\nu}),
\ee
and the assertion is proven.

\noindent{\bf Case III. ${\CB}$ is not local, $\bt\in\mathbb N$.}
In this case, $\CBa\ik=c\tilde B_0(\al)\CH\ik^{(\bt)}$, where $\CH$ is the Hilbert transform and $\tilde B_0(\al)\in C^\infty(\Omega)$. Choose $N$ as in Case II. Clearly, \eqref{psi-bj-part-II} still holds. Similarly to \eqref{psi-smj-part-II}, to estimate the remainder, we look at the integral
\be\label{psi-smj-part-III}
J(t):=\int\frac{g(q)}{q-t}dq,\ g(q):=\ik^{(\bt)}(q+\e)-\ik^{(\bt)}(q).
\ee
By IK3$'$, $g(q)=O(\e)$ and $g'(q)=O(1)$. Writing 
\be\label{psi-part-est}
J(t)=\left(\int_{-a}^{t-\e}+\int_{t+\e}^a
\right)\frac{g(q)}{q-t}dq+\int_{t-\e}^{t+\e}\frac{g(q)-g(t)}{q-t}dq,
\ee
where $a$, $0<a<\infty$ is sufficiently large, and applying elementary estimates we obtain the desired result.

The proof of the second estimate in \eqref{an-psi-pr-1} is fairly similar. In case I, the result immediately follows by noticing that
\be\label{A-dif}
\sup_{|p|\le a}|\CA(p+\e)-\CA(p)|=O(\e^{\min(s-1,1)}).
\ee
In cases II and III, we find $N$ so that $|j|\ge N$ implies $j-p$ is bounded away from zero. Then, similarly to \eqref{psi-bj-part-II},
\be\label{psi-bj-part-II-alt}
\sum_{|j|\ge N}(\CBa\ik)(t-j) [\CA(j-(p+\e))-\CA(j-p)]=O(\e).
\ee
The required estimate of the remainder then follows from \eqref{A-dif}.

The fact that all the estimates are uniform with respect to $t$ and $p$ is obvious. The uniformity with respect to $\al\in\Omega$ follows from the assumption $a_\pm(\al)\in C_0^\infty(\Omega)$ (cf. \eqref{gfn-def}).

{The proofs of cases I and III go through if $\ik^{(\bt+1)}\in L^\infty(\br)$. The proof of case II goes through if $\ik^{(\lceil\bt\rceil)}\in L^\infty(\br)$. In the remaining proofs, we will be keeping track of the degree of exactness of $\ik$ (denoted by $\CE_\ik$) and the highest order derivative of $\ik$ (denoted by $\CD_\ik$) that is required at each step. The maximum values of $\CE_\ik$ and $\CD_\ik$ are then stated in IK1$'$ and IK3$'$, respectively.}

\section{Proof of Lemma \ref{contverpsi}}\label{sec:contverpsi}

In what follows we assume $t\to+\infty$. The proof when $t\to-\infty$ is completely analogous. 
%
%\noindent
%{\bf Case I. ${\CBa}$ is local.} Here $\bt\in \mathbb N$ and $\tilde b(\la)=b_+\la^{\bt}$. Then, by IK4$'$,
%\be\label{B-loc-v2}\begin{split}
%\Psi(t)&=b_+i^\bt\int \ik^{(\bt)}(t-r) \CA(r)dr=b_+i^\bt\int \ik(t-r) \CA^{(\bt)}(r)dr\\
%&=b_+i^\bt \CA^{\bt}(t)+O(t^{s-2-\bt}),\ t\to\infty.
%\end{split}
%\ee
%To investigate the remaining cases, 
%
We begin by computing $\CF\CA$ using \eqref{three-fns} and \eqref{main-eqs}:
\be\label{FT-A}\begin{split}
\tilde a(\la):=&(\CF\CA)(\la)=\Gamma(s)\begin{cases}a_+q(\la+i0)^{-s}+(a_-/q)(\la-i0)^{-s},& s\not\in\mathbb Z,\\
\left[a_+q+(a_-/q)\right]\la^{-s}+\pi\left[a_+(i/q)-a_-iq\right]\de^{(s-1)},&
s\in\mathbb N,
\end{cases}\\ 
q:=&e^{i(\pi/2)s}.
\end{split}
\ee
Therefore (cf. \eqref{aux-fn-lim-pm}):
\be\label{Psi-ft}
\Psi(t)=\frac1{2\pi}\int \tilde b(\la)\tilde\ik(\la)\tilde a(\la)e^{-i\la t}d\la=\frac1{2\pi}\int \tilde\ik(\la)\left(c_+^{(1)}\la_+^{\bt-s}+c_-^{(1)}\la_-^{\bt-s}\right)e^{-i\la t}d\la.
\ee
To prove the lemma consider two cases. 

\noindent{\bf Case I. $\bt-s\not\in\mathbb Z$.}
Set $k:=\lceil\bt-s+1\rceil$. Since $\tilde\ik(0)=1$, Theorem 1 in Section IV.2 of \cite{wong} gives  
\be\label{Psi-ass-IIa}\begin{split}
\Psi(t)&=\CF^{-1}\left(c_+^{(1)}\la_+^{\bt-s}+c_-^{(1)}\la_-^{\bt-s}\right)
+O(t^{-k})\int \left|\left(|\la|^{\bt-s}(\tilde\ik(\la)-1)\right)^{(k)}
\right|d\la.
\end{split}
\ee
The following condition ensures that the last integral in \eqref{Psi-ass-IIa} is finite:
\be\label{ik-cond-IIa}
\tilde\ik^{(j)}(\la)=O(|\la|^{s-1-\bt-c}),\ \la\to\infty,\ 0\le j\le \lceil\bt-s+1\rceil,
\ee
for some $c>0$.

\noindent{\bf Case II. $\bt-s\in 0\cup\mathbb N$.} The assumption $\bt\ge s+(n-3)/2$ implies that if $\bt-s\in\mathbb Z$, then $\bt-s\ge0$. The asymptotics of $\Psi$  is obtained integrating by parts in \eqref{Psi-ft}:
\be\label{Psi-ass-IIb}\begin{split}
\Psi(t)=&\CF^{-1}\left(c_+^{(1)}\la_+^{\bt-s}+c_-^{(1)}\la_-^{\bt-s}\right)+O(t^{s-2-\bt})\\
&\times\left[\int_0^\infty \left|\left(\la^{\bt-s}\tilde\ik(\la)\right)^{(\bt-s+2)}\right|d\la+\int_0^\infty \left|\left(\la^{\bt-s}\tilde\ik(-\la)\right)^{(\bt-s+2)}\right|d\la
+|\tilde\ik'(0)|\right].
\end{split}
\ee
The following condition ensures that \eqref{Psi-ass-IIb} holds (including that all the boundary terms of order less than $O(t^{s-1-\bt})$ vanish):
\be\label{ik-cond-IIb}
\tilde\ik^{(j)}(\la)=O(|\la|^{s-1-\bt-c)}),\ \la\to\infty,\ 0\le j\le \bt-s+2,
\ee
for some $c>0$.

To complete the proof, it remains to show that if $c_-^{(1)}=c_+^{(1)}e^{-i(\bt-s)\pi}$, then the leading term of the asymptotics disappears. 
%In Case I, $\bt\in \mathbb N$ and $b_+/b_-=(-1)^\bt$. From \eqref{c-ratios-req} and \eqref{aux-fn-lim-pm},
%\be\label{ratio-caseI}
%\frac{a_+ q+(a_-/q)}{(a_+/q)+a_-q}=q^{-2}.
%\ee
%If $s\not\in\mathbb Z$, then $q^4\not=1$, and \eqref{ratio-caseI} implies that $a_+=0$ in \eqref{three-fns}. If $s\in \mathbb N$, then $\CA^{(\bt)}(t)\equiv0$, $t\not=0$. 
From \eqref{main-eqs},
\be
\CF^{-1}\left(c_+^{(1)}\la_+^{\bt-s}+c_-^{(1)}\la_-^{\bt-s}\right)
=\begin{cases} c_1t_-^{s-1-\bt},&\bt-s\not\in\mathbb Z,\\
c_2\de^{(\bt-s)}(t),&\bt-s\in 0\cup\mathbb N,\end{cases}
\ee
for some $c_{1,2}$, and the assertion follows. 

{Condition IK3$'$ guarantees that \eqref{ik-cond-IIa} and \eqref{ik-cond-IIb} hold. Since $\ik$ is compactly supported, $\ik^{(l)}\in L^1(\br)$ implies that $\tilde \ik^{(j)}(\la)=O(|\la|^{-l})$, $\la\to\infty$, for any $j\ge0$. Therefore, we have to make sure that the following inequality is satisfied 
\be\label{ik-ineq}
\begin{cases}
\bt+1,& \bt\in\mathbb N\\
\lceil\bt\rceil ,& \bt\not\in\mathbb N
\end{cases} \ge
\bt+1+c-s
\ee
for some $c>0$. The above inequality holds if $s\ge 1+c$ for some $c>0$. This is clearly true, since we assume that $s\ge (n+1)/2\ge 3/2$.}

\section{Proof of Lemma \ref{psi-asympt}}\label{sec:psi-asympt}

In view of $\psi(t,p)=\psi(t-m,p-m)$, $m\in\mathbb Z$, we may assume without loss of generality that $p\in[0,1]$. Define
\be\label{big-F-v1}
F_s(t):=\sum_j \ik(t-j) \CA(j-p),\ F_i(t):=\int \ik(t-r) \CA(r-p)dr.
\ee
The subscripts `$s$' and `$i$' stand for the `sum' and `integral', respectively. To simplify notations, the dependence of $F_s$ and $F_i$ on $p$ is ignored.
First, we have
\be\label{F-bounds-v1}\begin{split}
F_*^{(l)}(t)&=\CA^{(l)}(t)+O(|t|^{s-2-l}), \ t\to\infty,\ 0\le l\le L_\bt,\\ 
F_*^{(l)}(t)&=O(|t|^{s-1-l}), \ t\to\infty,\ l=L_\bt+1,\ \bt\in\mathbb N,
\end{split}
\ee
where $*=s,i,$ and the big-$O$ terms are uniform with respect to $p\in [0,1]$. The statement for $F_i$ is trivial in view of IK4$'$. The statement for $F_s$ follows easily too by using that $\ik$ is exact for polynomials of degree up to $L_\bt$, representing $\CA(j-p)$ as the sum of the Taylor polynomial of degree $L_\bt$ centered at $t$ and the remainder, and differentiating $F_s$ the required number of times. 
%, find $c>0$ sufficiently large so that $|t|\ge c$ and $p\in [0,1]$ imply $\ik(t-p)=0$. Since , we get
%\be\label{Taylor-v1}\begin{split}
%F_s^{(k)}(t)&=\sum_j \ik^{(k)}(t-j) \left(\sum_{l=0}^k \CA^{(l)}(t-p)\frac{(j-t)^l}{l!}+R_{k+1}(j,t)\right)\\
%&=\CA^{(k)}(t-p)+\sum_j \ik^{(k)}(t-j) R_{k+1}(j,t),\ |t|\ge c,
%0\le k\le L_\bt,
%\end{split}
%\ee
%and
%\be\label{Taylor-v1-Lb}\begin{split}
%F_s^{(k)}(t)&=\sum_j \ik^{(k)}(t-j) \left(\sum_{l=0}^{k-1}\CA^{(l)}(t-p)\frac{(j-t)^l}{l!}+R_k(j,t)\right)\\
%&=\sum_j \ik^{(k)}(t-j) R_k(j,t),\ |t|\ge c,
%k=L_\bt+1,
%\end{split}
%\ee
%where 
%\be\label{Rkp1-v1}
%R_k(r,t):=\CA(r-p)-\sum_{l=0}^{k-1}\CA^{(l)}(t-p)\frac{(r-t)^l}{l!}.
%\ee
%If $k=0$, the sums with respect to $l$ in \eqref{Taylor-v1} and \eqref{Rkp1-v1} disappear. By construction, $\ik(t-j)\not=0$ implies $|j-t|\le c$, and
%\be\label{Rk-bnd-v1}
%|R_k(j,t)|= \max_{r\in\ts(\ik)} |\CA^{(k)}(t+r-p)|.
%\ee
%The desired assertion is now obvious.

Denote
\be\label{deltas}
\Delta F(t):=F_s(t)-F_i(t),\ \Delta\psi(t):=\psi(t,p)-\Psi(t-p).
\ee
The $p$-dependence of $\Delta F$, $\Delta\psi$, and various other quantities below is omitted for simplicity. Clearly, $\Delta\psi(t)=\CBa\Delta F$. From \eqref{F-bounds-v1}, 
\be\label{F-diff}
\Delta F^{(l)}(t)=\begin{cases} O(|t|^{s-2-l}),& 0\le l\le L_\bt,\\
O(|t|^{s-1-l}),& l=L_\bt+1,\ \bt\in\mathbb N, 
\end{cases}
\quad t\to\infty.
\ee

{For \eqref{F-diff} to hold when $l\le L_\bt$, $\ik$ should be exact to the degree $l$ so that the leading terms in the asymptotics of $F_s^{(l)}(t)$ and $F_i^{(l)}(t)$ cancel each other. Thus, \eqref{F-diff} for any $0\le l\le L_\bt$ requires $\CE_\ik=l$, $\CD_\ik=l$. When $l=L_\bt+1$ in \eqref{F-diff}, no cancellation is needed, and in this case $\CE_\ik=l-1$, $\CD_\ik=l$.}

In what follows we assume $t\to+\infty$. The proof when $t\to-\infty$ is completely analogous. 
To prove the lemma we consider three cases, which correspond to the three lines in \eqref{remainder-v2}. Denote $\vartheta:=s-2-\bt$. The condition $\kappa_2\ge0$ implies $\vartheta\le-3/2$.

\noindent
{\bf Case I. ${\CBa}$ is local, $\bt\in \mathbb N$}. By \eqref{F-diff} with $l=\bt$ ({$\CE_\ik=\bt$, $\CD_\ik=\bt$}),
\be\label{B-loc-v1}
\Delta \psi(t)=b_+(i\pa_p)^\bt  \Delta F(t)=O(t^\vartheta).
\ee

\noindent{\bf Case II. ${\CBa}$ is not local, $\bt\not\in\mathbb N$.}
Set $k:=\lceil\bt\rceil$, $\nu:=\{\bt\}$, $0<\nu<1$. Then $\Delta\psi(t)$ is a linear combination of the following terms:
\be\label{aux_terms-v1}
J_\pm(t):=\int \Delta F^{(k)}(q)(q-t)_\pm^{-\nu} dq.
\ee
By \eqref{F-diff} with $l=k=\lceil\bt\rceil$ ({$\CE_\ik=\lceil\bt\rceil$, $\CD_\ik=\lceil\bt\rceil$}):
\be\label{jp-est-v1}\begin{split}
J_+(t)=&\int_t^{\infty}\Delta F^{(k)}(q)(q-t)^{-\nu} dq=\int_t^{\infty} O(q^{s-2-k})(q-t)^{-\nu} dq
=O(t^\vartheta),
\end{split}
\ee
where we have used that $s-1-k-\nu=\vartheta<0$. The term $J_-(t)$ is estimated by splitting it into two integrals:
\be\label{jm-est-v1}\begin{split}
&J_-^{(1)}(t):=\int_{-\infty}^{t/2} \Delta F^{(k)}(q)(t-q)^{-\nu} dq,\
J_-^{(2)}(t):=\int_{t/2}^t \Delta F^{(k)}(q) (t-q)^{-\nu} dq.
\end{split}
\ee
Integrating by parts and using \eqref{F-diff} with all $l$, $0\le l\le k-1$ ({$\CE_\ik=\lceil\bt\rceil-1$, $\CD_\ik=\lceil\bt\rceil-1$}), gives
\be\label{jm-est-1-v1}\begin{split}
|J_-^{(1)}(t)|&= O(1)\int_{-\infty}^{t/2}| \Delta F(q)|(t-q)^{-(\bt+1)} dq+O(t^\vartheta)\\
&=O(1)\int_{-\infty}^{t/2} \frac{|q|^{s-2}}{(t-q)^{\bt+1}}dq+O(1)\int_{-1}^{1} \frac{1}{(t-q)^{\bt+1}}dq+O(t^\vartheta)
=O(t^\vartheta),
\end{split}
\ee
where we have used that $\kappa_1\ge0$, i.e. $s-1\ge 1/2$. The term $J_-^{(2)}(t)$ is estimated analogously to \eqref{jp-est-v1}, and we get the same estimate as in \eqref{jp-est-v1}. Therefore, $\Delta\psi(t)=O(t^\vartheta)$.

\noindent{\bf Case III. ${\CBa}$ is not local, $\bt\in\mathbb N$.}
Now we have to look at only one expression
\be\label{hilb-v1}
J(t):=\int \Delta F^{(k)}(q)\frac1{q-t}dq,\ \bt=k,
\ee
which is split into five integrals:
\be\label{hilb-v1-pieces}
J_1(t)+\dots J_5(t):=\left(\int_{-\infty}^{-t}+\int_{-t}^{t/2}+\int_{t/2}^{t-1}
+\int_{t-1}^{t+1}+\int_{t+1}^{\infty}\right) \Delta F^{(k)}(q)\frac1{q-t}dq.
\ee
By \eqref{F-diff} with $l=k$ ({$\CE_\ik=\bt$, $\CD_\ik=\bt$}), we immediately get $J_1(t)=O(t^\vartheta)$. 

Integrating by parts in the definition of $J_2$ and using that all the boundary terms are of order $O(t^\vartheta)$ ({$\CE_\ik=\bt-1$, $\CD_\ik=\bt-1$}), we get similarly to \eqref{jm-est-1-v1} that $J_2(t)=O(t^\vartheta)$. Using \eqref{F-diff} with $l=k$ ({$\CE_\ik=\bt$, $\CD_\ik=\bt$}) in $J_3$ and $J_5$ gives $J_{3,5}(t)=O(t^\vartheta\log t)$.

Consider now $J_4$. By \eqref{F-diff} with $l=L_\bt+1$ ({$\CE_\ik=\bt$, $\CD_\ik=\bt+1$}):
\be\label{hilb-st2-v1}\begin{split}
J_4(t)&=\int_{t-1}^{t+1} \frac{\Delta F^{(k)}(q)-\Delta F^{(k)}(t)}{q-t}dq\\
&=\int_{t-1}^{t+1} \frac{[F_s^{(k)}(q)-F_s^{(k)}(t)]-[F_i^{(k)}(q)-F_i^{(k)}(t)]}{q-t}dq
=O(t^\vartheta).
\end{split}
\ee
%Obviously, $F_i^{(k+1)}=O(t^\vartheta)$. To show that $F_s^{(k+1)}=O(t^\vartheta)$, we represent $\CA(j-p)$ as the Taylor polynomial of degree $k=\bt$ plus the remainder, which is of the order $O(t^\vartheta)$: $\CA(j-p)=T_k(j-p;t)+R_{k}(j-p;t)$. Substitute this expression into \eqref{big-F-v1}. Assuming that $\ik$ is exact to the degree $k$, the sum with respect to $j$ of the terms involving $T_k(j-p;t)$ becomes a polynomial in $t$ of degree $k$, and its $(k+1)$-st derivative is zero.  {Thus, $\CE_\ik\ge\bt$.}
%Finally, using \eqref{F-diff} with $l=k+1$ ($k=L_\bt$) gives
%\be\label{hilb-st2-v1}
%J_4(t)=\int_{t-1}^{t+1} \frac{\Delta F^{(k)}(q)-\Delta F^{(k)}(t)}{q-t}dq
%=O(t^\vartheta).
%\ee

Combining all the results we finish the proof.

%The estimates of all $J_j$, except $j=4$, work if $\ik$ is exact up to order $\bt$.
%
%{For the above proof to work, $\ik$ has to satisfy the following requirements. In Case I, i.e. ${\CBa}$ is local, $\bt\in\mathbb N$, $\ik$ is exact up to order $\bt-1$ and $\ik^{(\bt)}\in L^\infty$. In Case II, i.e. ${\CBa}$ is not local, $\bt\not\in\mathbb N$, $\ik$ is exact up to order $\lceil\bt\rceil-1$ and $\ik^{(\lceil\bt\rceil)}\in L^\infty$. In Case III, ${\CBa}$ is not local, $\bt\in\mathbb N$, $\ik$ is exact up to order $\bt$ and $\ik^{(\bt+1)}\in L^\infty$.}

\section{Proof of Lemma~\ref{G-al-est}}\label{sec:prflem}

Define
\be\label{big-F}
F_{\al}(t):=\sum_j \ik_\e((H(\al)+t)-\e j) g(\al,\e j).
\ee
We begin by showing that there exists $c>0$ so that:
\be\label{F-bounds}
F_{\al}^{(l)}(t)=\begin{cases} 0,& |t|>c,\\
O(\e^{{{s'}}-1-l}),& |t|\le c\e,\\
O(|t|^{{{s'}}-1-l}),& c\e \le |t| \le c,
\end{cases}\quad {0\le l\le \begin{cases}\bt_0+1,& \text{if } \bt_0 \in\mathbb N,\\ 
\lceil\bt_0\rceil,& \text{if }\bt_0 \not\in\mathbb N.
\end{cases}}.
\ee
In \eqref{F-bounds}, $O(\e^{{{s'}}-1-l})$ is uniform with respect to $t$ provided that $|t|\le c\e$, and $O(t^{{{s'}}-1-l})$ is uniform with respect to $\e$ provided that $c\e \le |t|\le c$. Additionally, each of these big-$O$ terms is uniform with respect to $\al\in\Omega$. For this property to hold the {requirement \eqref{gfn-exp-unif} is essential}. The essence of the estimate \eqref{F-bounds} is to control the behavior of $F_{\al}^{(l)}(t)$ for small $t$.

Let us prove \eqref{F-bounds} for a given $l$. The top case in \eqref{F-bounds} follows because $\ik$ and $g$ are compactly supported. The middle case follows from the top line in \eqref{Oa_st1-g} and the fact that the number of terms in the sum in \eqref{big-F} is uniformly bounded for all $\al\in\Omega$ and all $\e>0$ sufficiently small. We also use that $\ik^{(l)}\in L^\infty$. To prove the bottom case, assume that $c>0$ is sufficiently large and $\ik_\e(t)\equiv 0$ when  $|t|\ge c\e$. The rest of the argument follows by representing $g(\al,\e j)$ as the Taylor polynomial of degree $l-1$ centered at $t$ plus the remainder, differentiating $l$ times, and then using \eqref{Oa_st1-g}. The degree of the Taylor polynomial is $l-1$ instead of $l$ as in Appendix~\ref{sec:psi-asympt}, because no cancellation is needed now. In other words, in \eqref{F-bounds} the precise knowledge of the leading order term of the asymptotics of $F_\al^{({l})}(t)$, $t\to0$, is not required, we only need its order of magnitude. Therefore, using \eqref{F-bounds} for some $l$ requires {$\CE_\ik=l-1$ and $\CD_\ik=l$}.

The rest of the proof of the lemma is largely very similar to \eqref{B-loc-v1}--\eqref{hilb-st2-v1}. The difference between the proofs is due to the fact that $s'$ can be large, and $s'-1-\bt'$ is no longer necessarily negative. In particular, the integrals over infinite intervals may diverge, and we need to use that $F_\al$ is compactly supported. Also, we need to use that all our estimates are uniform with respect to $\al\in\Omega$.

In what follows the standing assumption is $\al\in\Omega_b$, and we introduce the notation $p=p(\al)=\al\cdot \xe-H(\al)$. Since $\check H''(\dir_0)$ is negative definite, and $\check x$ is confined to a bounded set, we have $p(\al)\ge c(A)\e$, $\al\in\Omega_b$, and $c(A)\to\infty$ as $A\to\infty$. Denote $\vartheta:=s'-1-\bt'$.

\noindent
{\bf Case I. ${\CB}$ is local.} Here ${{\bt'}}\in 0\cup \mathbb N$. By the bottom case in \eqref{F-bounds} with $l=\bt'$ ({$\CE_\ik=\bt'-1$, $\CD_\ik=\bt'$}),
\be\label{B-loc}\begin{split}
G(\al)=\sum_j({\CBa} \ik_\e)(\al\cdot \xe-\e j) g(\al,\e j)=O(1)F_\al^{({\bt'})}(p(\al))=O(p(\al)^{\vartheta}),
\end{split}
\ee
where $O(1)$ is a $\coi(\Omega)$ function of $\al$. Here we assume that $A>0$ in the definitions of $\Omega_a$, $\Omega_b$ (cf. \eqref{two-sets-rn}) is sufficiently large, so that $p(\al)\ge c\e$ for any $\al\in\Omega_b$ and $\check x$, and the bottom case in \eqref{F-bounds} indeed applies. Similar assumptions are made in Cases II and III below. 

%Therefore, using the same argument as at the end of Section~\ref{sec:psi-asympt}, we see that {\eqref{B-loc} works for all $\bt'\le\bt$ if $\CE_\ik\ge\bt-1$.} 

\noindent{\bf Case II. ${\CB}$ is not local, $\bt'\not\in\mathbb Z$.}
Set $k:=\lceil{\bt'}\rceil$, $\nu:=\{{{\bt'}}\}$, $0<\nu<1$. Then $G(\al)$ is a linear combination of the following terms:
\be\label{aux_terms}
J_\pm(p):=\int F_{\al}^{(k)}(t)(t-p)_\pm^{-\nu} dt,\ p=p(\al).
\ee
The coefficients of the linear combination are $\coi(\Omega)$ functions of $\al$. The dependence of $J_\pm$ on $\al$ is omitted for simplicity. We begin by estimating $J_+(p)$ (with $c$ the same as in \eqref{F-bounds}):
\be\label{jp-est}\begin{split}
J_+(p)=&\int_p^{c} F_{\al}^{(k)}(t)(t-p)^{-\nu} dt=\int_p^{c} O(t^{{s'}-1-k})(t-p)^{-\nu} dt
=O(\Psi(p)),
\end{split}
\ee
where $\Psi$ is defined in \eqref{gr-fn}. 
Similarly to \eqref{B-loc}, in \eqref{jp-est} we assumed that $A>0$ is sufficiently large, so $p=p(\al)>c\e$, and the bottom case in \eqref{F-bounds} with $l=k=\lceil{\bt'}\rceil$ applies  ({$\CE_\ik=\lceil{\bt'}\rceil-1$, $\CD_\ik=\lceil{\bt'}\rceil$}).

The term $J_-(p)$ is estimated by splitting it into two expressions:
\be\label{jm-est}\begin{split}
&J_-^{(1)}(p):=\int_{-c}^{p/2} F_{\al}^{(k)}(t)(p-t)^{-\nu} dt,\
J_-^{(2)}(p):=\int_{p/2}^{p} F_{\al}^{(k)}(t) (p-t)^{-\nu} dt.
\end{split}
\ee
Integrating by parts, using that $F_\al(t)\equiv0$, $t\le -c$, and appealing to the bottom case in \eqref{F-bounds} with $0\le l\le k-1$ ({$\CE_\ik=\lceil{\bt'}\rceil-2$, $\CD_\ik=\lceil{\bt'}\rceil-1$}) gives
\be\label{jm-est-1}\begin{split}
|J_-^{(1)}(p)|&\le O(1)\int_{-c}^{p/2}|F_{\al}(t)|(p-t)^{-(\bt'+1)} dt
+O(p^\vartheta)\\
&\le O(1)\left(\int_{-c}^{-c\e}+\int_{-c\e}^{c\e}+\int_{c\e}^{p/2}\right)\frac{\left| F_{\al}(t) \right|}{(p-t)^{{\bt'}+1}} dt+O(p^\vartheta)\\
&=O(1)\left[\int_{-c}^{p/2} \frac{|t|^{{s'}-1}}{(p-t)^{\bt'+1}}dt+
\frac{\e^{s'}}{p^{\bt'+1}}\right]+O(p^\vartheta).
\end{split}
\ee
%In \eqref{jm-est-1} we combined all three cases in \eqref{F-bounds} with $l=0$ to get $F_{\al}(t)=O(\e^{{{s'}}-1})+O(|t|^{{{s'}}-1})$, $t\in\br$. 
Considering the same three cases as in \eqref{gr-fn} and using that $\e=O(p(\al))$, it is easy to see that $J_-^{(1)}(p)=O(\Psi(p))$. 

To estimate $J_-^{(2)}(p)$, assume as before that $A>0$ is sufficiently large, $0.5p(\al)>c\e$, and the bottom case in \eqref{F-bounds} with $l=k$ applies ({$\CE_\ik=\lceil{\bt'}\rceil-1$, $\CD_\ik=\lceil{\bt'}\rceil$}). Then \eqref{jm-est} gives $J_-^{(2)}(p) = O(p^\vartheta)$. Combinining with the estimate for $J_-^{(1)}(p)$ this yields $J_-(p)=O(\Psi(p))$. Therefore,
\be\label{G-est-II}
G(\al) = O(\Psi(p(\al))).
\ee

\noindent{\bf Case III. ${\CB}$ is not local, $\bt'\in 0\cup\mathbb N$.}
In this case we look at only one expression
\be\label{hilb}
J(p):=\int\frac{F_{\al}^{(k)}(t)}{t-p}dt,\ {\bt'}=k.
\ee
Then
\be\label{hilb-st2}
J_1(p):=\int_{0.5p}^{1.5p} \frac{F_{\al}^{(k)}(t)-F_{\al}^{(k)}(p)}{t-p}dt.
\ee
By the bottom line in \eqref{F-bounds} with $l=k+1$ ({$\CE_\ik=\bt'$, $\CD_\ik=\bt'+1$}), upon assuming $0.5p(\al)\ge c\e$,
\be\label{hilb-bnd}
J_1(p)=O(p^{{s'}-1-(k+1)})O(p)=O(p^\vartheta).
\ee
The other two terms:
\be\label{hilb-2-terms}
J_2(p)=\int_{1.5p}^{c} F_{\al}^{(k)}(t)\frac{1}{t-p}dt,\
J_3(p)=\int_{-c}^{0.5p} F_{\al}^{(k)}(t)\frac{1}{t-p}dt,
\ee
are estimated similarly to \eqref{jp-est} and \eqref{jm-est-1}, respectively ({$\CE_\ik=\bt'-1$, $\CD_\ik=\bt'$}):
\be\label{hilb-2-ests}\begin{split}
&J_2(p)= \int_{1.5p}^{c} O\left(t^{s'-1-k}\right)\frac{1}{t-p} dt=O(\Psi(p)),\\
&|J_3(p)|\le O(1) \int_{-c}^{0.5p} |F_{\al}(t)|\frac{1}{(p-t)^{k+1}}dt+O(p^\vartheta)
=O(\Psi(p)).
\end{split}
\ee
Combining \eqref{hilb-bnd} and \eqref{hilb-2-ests} yields
\be\label{hilb-bnd-all}
G(\al)=O(1)J(p(\al))=O(\Psi(p(\al))).
\ee

%{Since $\bt'\le\bt$ and $s'\ge s$, we will assume $\bt'=\bt$ and $s'=s$ for the purpose of finding conditions on $\ik$. 
%For the above proof to work, $\ik$ has to satisfy the following requirements. In Case I, i.e. ${\CBa}$ is local, $\bt\in\mathbb N$, $\ik$ is exact up to order $\bt-1$ and $\ik^{(\bt)}\in L^\infty$. In Case II, i.e. ${\CBa}$ is not local, $\bt\not\in\mathbb N$, $\ik$ is exact up to order $\lceil\bt\rceil-1$ and $\ik^{(\lceil\bt\rceil)}\in L^\infty$. In Case III, ${\CBa}$ is not local, $\bt\in\mathbb N$, $\ik$ is exact up to order $\bt$ and $\ik^{(\bt+1)}\in L^\infty$.}

\section{Proof of Lemma~\ref{lemma:lot} in the case $\kappa_2=0$.} \label{sec:prf-kappa2-zero}

The goal is to show that lower order terms contribute a constant to the DTB at $x_0$, i.e.
\be\label{finally}
\lim_{\e\to0}\left[({\CB_\e} g)(\xe)-({\CB_\e} g)(x_0)\right]=0,
\ee
so the DTB is independent of $\check x$ confined to bounded sets. 
From \eqref{Oa_st2-int},
\be\label{Oa_st2-v2}
\lim_{\e\to0}\sum_{\al_{\vec k}\in\Omega_a} G_{\xc}(\al_{\vec k})|\Delta\al_{\vec k}| = 0.
\ee
In what follows we introduce notations like these
\be\label{delJ}
\Delta F_\al(p):=F_\al(p+h\e)-F_\al(p),\ \Delta J_+(p):=J_+(p+h\e)-J_+(p),\dots,
\ee
where $h=\al\cdot\check x$. From \eqref{F-bounds},
\be\label{delF-bounds}
\Delta F_{\al}^{(l)}(t)=\begin{cases} 0,& |t|>c,\\
O(\e^{s'-1-l}),& |t|\le c\e,\\
O(\e|t|^{s'-2-l}),& c\e \le |t| \le c,
\end{cases}\quad {0\le l\le \begin{cases}\bt_0,& \text{if } \bt_0 \in\mathbb N,\\ 
\lceil\bt_0\rceil-1,& \text{if }\bt_0 \not\in\mathbb N.
\end{cases}}
\ee
The estimate of $\Delta F_{\al}^{(l)}(t)$ in \eqref{delF-bounds} is based on the estimate of $F_{\al}^{(l+1)}(t)$ in \eqref{F-bounds}, and the latter requires {$\CE_\ik=l$ and $\CD_\ik=l+1$.}

Analogously to \eqref{gr-fn}, define
\be\label{gr-fn-v2}
\Psi_2(p):=\begin{cases} p^\vartheta,&\vartheta<0,\\ \ln(1/p),& \vartheta=0,\\ 1,&\vartheta>0,
\end{cases}\quad p>0,\ \vartheta:=s'-2-\bt'.
\ee
To estimate the contribution of $\al_{\vec k}\in\Omega_b$, we replace $F_\al(t)$ with $\Delta F_\al(t)$ in Cases I--III in Appendix~\ref{sec:prflem}. Case I is the easiest. By the bottom line in \eqref{delF-bounds} with $l=\bt'$ ({$\CE_\ik=\bt'$, $\CD_\ik=\bt'+1$}), the analogue of \eqref{B-loc} becomes
\be\label{B-loc-v2}\begin{split}
G_{\xc}(\al)-G_{0}(\al)&=O(1)\Delta F_\al^{({\bt'})}(p)=O(\e p^\vartheta),\ p=p(\al).
\end{split}
\ee
%When considering Cases II and III, we introduce notations like this
%\be\label{delJ}
%\Delta J_+(p):=J_+(p+h\e)-J_+(p)
%\ee
%for all the quantities in \eqref{aux_terms}--\eqref{hilb-2-ests}. 

\noindent In Case II, we need to estimate $\Delta J_{\pm}(p)$ (cf. \eqref{aux_terms}). As usual, set $k:=\lceil{\bt'}\rceil$, $\nu:=\{{{\bt'}}\}$, $0<\nu<1$. Then
\be\label{jp-v2}\begin{split}
\Delta J_\pm(p)&=\int_{-c}^{c} F_{\al}^{(k)}(t)(t-(p+h\e))_\pm^{-\nu}dt-\int_{-c}^{c} F_{\al}^{(k)}(t)(t-p)_\pm^{-\nu}dt\\
&=\int_{-c}^{c} \Delta F_{\al}^{(k)}(t)(t-p)_\pm^{-\nu}dt,
\end{split}
\ee
where we assumed that $c>0$ is sufficiently large. First, consider $\Delta J_+(p)$:
\be\label{jp-split-v2}\begin{split}
\Delta J_+(p)&=\left(\int_{p}^{2p}+\int_{2p}^{c}\right) F_{\al}^{(k)}(t)\left[(t-(p+\e))_+^{-\nu}-(t-p)^{-\nu}\right]dt\\
&=:\Delta J_+^{(1)}(p)+\Delta J_+^{(2)}(p).
\end{split}
\ee
After simple transformations,
\be\label{jp1-split-v2}\begin{split}
\Delta J_+^{(1)}(p)&=\int^{2p}_{p+\e}F_\al^{(k)}(t)\left[(t-(p+\e))^{-\nu}-(t-p)^{-\nu}\right]dt
-\int_{p}^{p+\e} \frac{F_{\al}^{(k)}(t)}{(t-p)^\nu}dt\\
&=:\Delta J_+^{(11)}(p)-\Delta J_+^{(12)}(p).
\end{split}
\ee
Using \eqref{F-bounds} with $l=k$ ({$\CE_\ik=\lceil{\bt'}\rceil-1$, $\CD_\ik=\lceil{\bt'}\rceil$}) gives
\be\label{jp12}\begin{split}
|\Delta J_+^{(11)}(p)|=&O(1)\int^{2p}_{p+\e}t^{s'-1-k}\left[(t-(p+\e))^{-\nu}-(t-p)^{-\nu}\right]dt\\
=&O(p^{s'-1-k})\int^{2p}_{p+\e}\left[(t-(p+\e))^{-\nu}-(t-p)^{-\nu}\right]dt
=O(\e^{1-\nu} p^{s'-1-k}),
\end{split}
\ee
and
\be\label{jp11-jp13}\begin{split}
\Delta J_+^{(12)}(p)=O(\e^{1-\nu} p^{s'-1-k}).
\end{split}
\ee
Combining \eqref{jp1-split-v2}--\eqref{jp11-jp13} gives
\be\label{jp-estim-v2}
\Delta J_+^{(1)}(p)=O(\e^{1-\nu} p^{s'-1-k}).
\ee
Estimation of $\Delta J_+^{(2)}(p)$ also is based on \eqref{F-bounds} with $l=k$ ({$\CE_\ik=\lceil{\bt'}\rceil-1$, $\CD_\ik=\lceil{\bt'}\rceil$}):
\be\label{djp2}\begin{split}
|\Delta J_+^{(2)}(p)|=&O(1)\int_{2p}^c t^{s'-1-k}\left[(t-(p+\e))^{-\nu}-(t-p)^{-\nu}\right]dt\\
=&O(\e)\int_{2p}^{c}t^{s'-1-k}t^{-(1+\nu)}dt=O(\e \Psi_2(p)).
\end{split}
\ee
Therefore, from \eqref{jp-estim-v2} and \eqref{djp2}
\be\label{djp-estim-v2}
\Delta J_+(p)=O(\e^{1-\nu} p^{s'-1-k})+O(\e \Psi_2(p)).
\ee

%\be\label{jp12}\begin{split}
%\Delta J_+^{(12)}(p)=&(1-\nu)^{-1}\int^{2p}_{p+\e}F_\al^{(k)}(t)d\left((t-(p+\e))^{1-\nu}-(t-p)^{1-\nu}\right)dt\\
%=&(1-\nu)^{-1}\biggl[\left.F_\al^{(k)}(t)d\left((t-(p+\e))^{1-\nu}-(t-p)^{1-\nu}\right)\right|_{t=p+\e}^{2p}\\
%&-\int^{2p}_{p+\e}F_\al^{(k+1)}(t)d\left((t-(p+\e))^{1-\nu}-(t-p)^{1-\nu}\right)\biggr]
%\end{split}
%\ee

Next, we investigate $\Delta J_-(p)$:
\be\label{jm-split-v2}\begin{split}
\Delta J_-(p)&=\left(\int_{-c}^{p/2}+\int_{p/2}^{p}\right) \Delta F_{\al}^{(k)}(t)(p-t)^{-\nu}dt=:\Delta J_-^{(1)}(p)+\Delta J_-^{(2)}(p).
\end{split}
\ee
Estimation of $\Delta J_-^{(1)}(p)$ is analogous to \eqref{jm-est-1}:
\be\label{deljm-est-1}\begin{split}
|\Delta J_-^{(1)}(p)|&\le O(1)\int_{-c}^{p/2}|\Delta F_{\al}(t)|(p-t)^{-(\bt'+1)} dt
+{O(\e p^\vartheta)}\\
&\le O(1)\left(\int_{-c}^{-c\e}+\int_{-c\e}^{c\e}+\int_{c\e}^{p/2}\right)\frac{\left|\Delta F_{\al}(t) \right|}{(p-t)^{{\bt'}+1}} dt+O(\e p^\vartheta)\\
&=O(1)\left[\e\int_{-c}^{p/2} \frac{|t|^{s'-2}}{(p-t)^{\bt'+1}}dt+
\e\frac{\e^{s'-1}}{p^{\bt'+1}}\right]+O(\e p^\vartheta)=O(\e \Psi_2(p)).
\end{split}
\ee
By assumption, $s'\ge s\ge (n+1)/2$. Therefore, $s'-2\ge -1/2$, and the first integral on the last line is absolutely convergent at $t=0$. 
When integrating by parts, \eqref{delF-bounds} is used for $l=k-1,k-2,\dots,0$. Hence {$\CE_\ik=\lceil{\bt'}\rceil-1$, $\CD_\ik=\lceil{\bt'}\rceil$}. 

Estimation of $\Delta J_-^{(2)}(p)$ is analogous to that of $\Delta J_+^{(1)}(p)$. Hence $\Delta J_-(p)$ also satisfies \eqref{jp-estim-v2}, and combining with \eqref{deljm-est-1} and \eqref{djp-estim-v2} gives
\be\label{CaseII-v2}
G_{\xc}(\al)-G_{0}(\al)=O(\e^{1-\nu} p^{s'-1-k})+O(\e \Psi_2(p)),\ p=p(\al).
\ee

In Case III, $k=\bt'$, we have
\be\label{cIII-st1}
\Delta J(p)=\int \Delta F_{\al}^{(k)}(t)\frac1{t-p}dt,
\ee
which is analogous to \eqref{hilb}. Similarly to \eqref{hilb-st2},
\be\label{hilb-st2-v2}
\Delta J_1(p)=\int_{p-\e}^{p+\e} \frac{\left[F_{\al}^{(k)}(t+h\e)-F_{\al}^{(k)}(p+h\e)\right]-\left[F_{\al}^{(k)}(t)-F_{\al}^{(k)}(p)\right]}{t-p}dt.
\ee
Estimating each of the two terms in \eqref{hilb-st2-v2} separately using the bottom case in \eqref{delF-bounds} with $l=k$ ({$\CE_\ik=\bt'$, $\CD_\ik=\bt'+1$}) and adding the two estimates yields
\be\label{hilb-st3-v2}
\Delta J_1(p)= O(\e p^\vartheta).
\ee
The two remaining terms, $\Delta J_2(p)$ and $\Delta J_3(p)$, are as follows
\be\label{J2J3-v2}
\Delta J_2(p)=\int_{-c}^{p-\e} \Delta F_{\al}^{(k)}(t)\frac1{t-p}dt,\
\Delta J_3(p)=\int_{p+\e}^c \Delta F_{\al}^{(k)}(t)\frac1{t-p}dt.
\ee
To estimate $\Delta J_2$, we write
\be\label{J2-parts}
\Delta J_{21}(p)=\int_{-c}^{p/2} \Delta F_{\al}^{(k)}(t)\frac1{t-p}dt,\
\Delta J_{22}(p)=\int_{p/2}^{p-\e} \Delta F_{\al}^{(k)}(t)\frac1{t-p}dt.
\ee
Integrating by parts in $\Delta J_{21}$ and using \eqref{delF-bounds} with $l=k-1,k-2,\dots,0$ ({$\CE_\ik=\bt'-1$, $\CD_\ik=\bt'$}) gives similarly to \eqref{deljm-est-1}
\be\label{J21-byparts}\begin{split}
\Delta J_{21}(p)&=O(1)\int_{-c}^{p/2} |\Delta F_{\al}(t)|\frac1{(p-t)^{k+1}}dt+ O(\e p^\vartheta)\\
&=O(1)\left(\e\int_{-c}^{p/2} \frac{|t|^{s'-2}}{(p-t)^{k+1}}dt
+\int_{-c\e}^{c\e} \frac{\e^{s'-1}}{(p-t)^{k+1}}dt\right)+O(\e p^\vartheta)\\
&=O(\e)\left(\Psi_2(p)
+\e^{s'-1}p^{-(k+1)}\right)+O(\e p^\vartheta)=O(\e \Psi_2(p)).
\end{split}
\ee
By \eqref{delF-bounds} with $l=k$ ({$\CE_\ik=\bt'$, $\CD_\ik=\bt'+1$})
\be\label{J3-est}\begin{split}
\Delta J_{22}(p)&=O(\e)\int_{p/2}^{p-\e}\frac{t^{s'-2-k}}{p-t}dt=\e\ln(p/\e)O(p^\vartheta),\\
\Delta J_3(p)&=O(\e)\int_{p+\e}^c \frac{t^{s'-2-k}}{t-p}dt=O(\e)(\ln(p/\e)p^\vartheta+\Psi_2(p)).
\end{split}
\ee
Combining \eqref{hilb-st3-v2} and \eqref{J21-byparts}, \eqref{J3-est} gives
\be\label{CaseIII-v2}\begin{split}
G_{\xc}(\al)-G_{0}(\al)&=\e \left[\ln(p/\e)O(p^\vartheta)+O(\Psi_2(p))\right],\  p=p(\al).
\end{split}
\ee

Now we prove \eqref{finally}. Suppose $\vartheta<0$. Comparing \eqref{B-loc-v2}, \eqref{CaseII-v2}, and \eqref{CaseIII-v2}, it is clear that we have to consider only the last two cases. 
In Case II, the analogue of \eqref{B-loc_st2} becomes:
\be\label{B-loc_II-excep}\begin{split}
&\sum_{\al_{\vec k}\in\Omega_b}(G_{\xc}(\al_{\vec k})-G_{0}(\al_{\vec k}))|\Delta\al_{\vec k}|\\
&=O(\e^{1-\nu})\int_{\e^{1/2}}^{1} r^{2(s'-1-k)+(n-2)}dr+O(\e)\int_{\e^{1/2}}^{1} r^{2\vartheta+(n-2)}dr\\
&=O\left(\e^{(s'-\bt')-(s_0-\bt_0)}\right)\to0,\ \e\to0.
\end{split}
\ee
In Case III, the computation is
\be\label{B-loc_st2-excep}\begin{split}
&\sum_{\al_{\vec k}\in\Omega_b}(G_{\xc}(\al_{\vec k})-G_{0}(\al_{\vec k}))|\Delta\al_{\vec k}|\\
&=O(\e)\int_{\e^{1/2}}^{1} \ln(r^2/\e)r^{2\vartheta+(n-2)}dr
=O(\e)\int_{\e}^{1} \ln(r/\e)r^{\vartheta+(n-3)/2}dr\\
&=O(\e)+O\left(\ln(1/\e)\e^{(s'-\bt')-(s_0-\bt_0)}\right)\to0,\ \e\to0.
\end{split}
\ee
Combining with \eqref{Oa_st2-v2} finishes the proof. The other two cases $\vartheta=0$ and $\vartheta>0$ can be considered analogously.

\section{Useful formulas}\label{sec:usef}
For convenience, we state here the key formulas used in the paper extensively (see \cite{GS}):
\be\label{main-eqs}
\begin{split}
&\mathcal F(x_\pm^{a})=e^{\pm i({a}+1)\pi/2}\Gamma({a}+1)(\la\pm i0)^{-({a}+1)},\ 
{a}\not=-1,-2,\dots,\\
&\mathcal F((x\pm i0)^{a})=\frac{2\pi e^{\pm i{a}\pi/2}}{\Gamma(-{a})}\la_{\mp}^{-({a}+1)},\ 
{a}\not=0,1,2,\dots,\\
&(x\pm i0)^{a}=x_+^{a}+e^{\pm i{a}\pi}x_-^{a},\ {a}\not=-1,-2,\dots,\\
&(x\pm i0)^{-n}=x^{-n}\mp \frac{i\pi(-1)^{n-1}}{(n-1)!}\de^{(n-1)}(x),\ n=1,2,\dots.
\end{split}
\ee
Also, the $n-1$ dimensional area of the sphere $S^{n-1}$ in $\br^n$ is $|S^{n-1}|=2\pi^{n/2}/\Gamma(n/2)$. Another useful identity is $\Gamma(s)\Gamma(1-s)=\pi/\sin(\pi s)$.
%\newpage
%The other way is to 
%\begin{enumerate}
%\item observe that the sum with respect to $i$ becomes the integral with respect to $\al$, 
%\item use the stationary phase method to find the asymptotics of the integral with respect to $\al$ as $\la/\e\to\infty$
%\end{enumerate}
%to obtain that the result equals to the convolution $(\ik * \CA)(h)$, where $\CA$ is defined in \eqref{lim-res-final}.

\end{document}